\documentclass[12pt]{amsart}
\usepackage{amsmath,amstext,amssymb,amsopn,amsthm}
\usepackage{enumerate}

\usepackage{mathrsfs}   
\usepackage{bm}
\numberwithin{equation}{section}

\allowdisplaybreaks

\usepackage[utf8]{inputenc}
\usepackage{hyperref}

\makeatletter
\g@addto@macro\th@plain{\thm@headpunct{}}
\makeatother

\numberwithin{equation}{section}

\newtheorem{thm}{Theorem}[section]
\newtheorem*{thm*}{Theorem}
\newtheorem{lemma}[thm]{Lemma}
\newtheorem{Corollary}[thm]{Corollary}
\newtheorem{proposition}[thm]{Proposition}

\theoremstyle{definition}

\newtheorem{defin}[thm]{Definition}

\newtheorem{remark}[thm]{Remark}
\newtheorem{example}[thm]{Example}

\def \C {{\mathbb C}}
\def \E {{\mathbb E}}
\def \P {{\mathbb P}}
\def \R {{\mathbb R}}
\def \eps {\varepsilon}
\def \A {{\mathbb A}}

\def \B {{\mathbb B}}
\def \X {{\mathbb X}}
\def \Y {{\mathbb Y}}
\def \Ab {{\mathbf A}}
\def \Bb {{\mathbf B}}
\def \Xb {{\mathbf X}}

\newcommand{\dd}{\mathrm{d}}
\newcommand{\cA}{ {\mathcal A} }
\newcommand{\tcA}{\tilde{\mathcal{A}}}

\newcommand{\NC}{ {\mathrm{NC}}}
\newcommand{\Kr}{ {\mathrm{Kr}}}

\makeatletter
\@namedef{subjclassname@2020}{%
        \textup{2020} Mathematics Subject Classification}
\makeatother

\title[Free Perpetuities I]{Free Perpetuities I: Existence, Subordination\\and Tail Asymptotics}

\author[S. Belinschi]{Serban Belinschi}

\author[B. Ko\l{}odziejek]{Bartosz Ko\l{}odziejek}
\email{bartosz.kolodziejek@pw.edu.pl}
\address{Faculty of Mathematics and Information Sciences, Warsaw University of Technology, Koszykowa 75, \mbox{00-662} Warsaw, Poland}

\author[K. Szpojankowski]{Kamil Szpojankowski}
\email{kamil.szpojankowski@pw.edu.pl}
\address{Faculty of Mathematics and Information Sciences, Warsaw University of Technology, Koszykowa 75, \mbox{00-662} Warsaw, Poland}

\thanks{KSz: This research was funded in part by National Science Centre, Poland WEAVE-UNISONO grant BOOMER 2022/04/Y/ST1/00008.\\
BK: This research was funded in part by National Science Centre, Poland, 2023/51/B/ST1/01535.
\\ For the purpose of Open Access, the authors have applied a CC-BY public copyright licence to any Author Accepted Manuscript (AAM) version arising from this submission.}

\subjclass[2020]{Primary 46L54; Secondary 60E05.}

\begin{document}

\begin{abstract}
        We study the free analogue of the classical affine fixed-point (or perpetuity) equation
        \[
        \X \stackrel{d}{=} \A^{1/2}\X\,\A^{1/2} + \B,
        \]
        where $\X$ is assumed to be $*$-free from the pair $(\A,\B)$, with $\A\ge 0$ and $\B=\B^*$. Our analysis covers both the subcritical regime, where $\tau(\A)<1$, and the critical case $\tau(\A)=1$, in which the solution $\X$ is necessarily unbounded. When $\tau(\A)=1$, we prove that the series defining $\X$ converges bilaterally almost uniformly (and almost uniformly under additional tail assumptions), while the perpetuity fails to have higher moments even if all moments of $\A$ and $\B$ exist.
       
        Our approach relies on a detailed study of the asymptotic behavior of moments under free multiplicative convolution, which reveals a markedly different behavior from the classical setting. By employing subordination techniques for non-commutative random variables, we derive precise asymptotic estimates for the tail of the distributions of $\X$ in both one-sided and symmetric cases. Interestingly, in the critical case, the free perpetuity exhibits a power-law tail behavior that mirrors the phenomenon observed in the celebrated Kesten’s theorem.
\end{abstract}
\maketitle

\section{Introduction}
\subsection{Problem description and motivation}

Let $(\A,\B)$ be a pair of self-adjoint random variables, where $\A$ is positive. In this paper we study the solution of the free analogue of the classical affine fixed-point equation and investigate properties of the distribution of $\X$ satisfying
\begin{align}\label{eqn:free_perp}
\X\stackrel{d}{=}\A^{1/2}\X\A^{1/2}+\B,  
\end{align}
where $\X$ is assumed to be $*$-free from the pair $(\A,\B)$. When $\X$ satisfies \eqref{eqn:free_perp}, we refer to it as a  free perpetuity.

In the classical setting, given a pair of random variables $(\Ab,\underline{B})$, with $\Ab$ an $N\times N$ matrix, $\underline{B}\in\mathbb{R}^N$,
a perpetuity is defined as a solution to
\begin{align}\label{eqn:classical_perp}
\underline{X} \stackrel{d}{=} \Ab\,\underline{X} + \underline{B},
\end{align}
where $\underline{X}$ and the pair $(\Ab,\underline{B})$ are independent.

Perpetuities have long been an active area of research in probability theory, with applications ranging from queueing theory to financial and actuarial mathematics. Once the existence of $\underline{X}$ is established, one typically investigates how its distribution depends on that of $(\Ab,\underline{B})$---with particular emphasis on the asymptotic tail behavior of $\underline{X}$, see below.

In the classical setting, the usual strategy to prove the existence of $\underline{X}$ is to iterate \eqref{eqn:classical_perp}. Let $\{(\Ab_n,\underline{B}_n)\}_{n\ge1}$ be an i.i.d.\ sequence of copies of $(\Ab,\underline{B})$ and define
\[
\Pi_{0}=1,\quad \Pi_n=\Ab_1\cdots \Ab_n,\quad n\ge1.
\]
Let $|\cdot|$ denote any norm on $\R^N$ and $\|\cdot\|$ the corresponding operator norm.
Under the conditions 
\[
-\infty\le \E[\log\|\Ab\|] < 0,\quad \E[\log^+|\underline{B}|]<+\infty
\]
one can show that the series
\[
\underline{S} = \sum_{n=1}^{+\infty }\Pi_{n-1} \underline{B}_n,
\]
converges almost surely, and the unique solution of \eqref{eqn:classical_perp} is given by $\underline{X}\stackrel{d}{=}\underline{S}$ (see, e.g., \cite{BDM}).

One of the most celebrated results in the theory of perpetuities originates from Kesten’s seminal work \cite{Kes73} and its subsequent refinements (see, e.g., \cite[Section 4.4]{BDM} for an extensive discussion of variants of Kesten’s theorem). The key assumption in these studies is the existence of a positive parameter $\alpha>0$ satisfying
\[
\inf_{n\in\mathbb{N}}\E[ \|\Pi_n\|^\alpha ]^{1/n} = 1.
\]
Under this condition---and additional technical assumptions on the distribution of $(\Ab,\underline{B})$---one can show that there exists a non-null Radon measure $\mu$ such that the solution to \eqref{eqn:classical_perp} fulfills
\[
t^{\alpha}\,\P\left(t^{-1}\underline{X}\in \cdot\right) \stackrel{v}{\longrightarrow}\mu,\qquad t\to+\infty,
\]
where $\stackrel{v}{\longrightarrow}$ denotes vague convergence of measures.

In the univariate case, the parameter $\alpha>0$ is determined by the equation $\E[|A|^\alpha]=1$. Moreover, under additional assumptions, Kesten's theorem guarantees that $\lim_{t\to+\infty}t^\alpha \P(X>t)$ exists and is strictly positive. Interestingly, as we explain below, free perpetuity $\X$ exhibits a similar power-law tail behavior under the condition that $\A$ has unit mean.

Recent advances in machine learning theory have demonstrated that the heavy-tail behavior observed in Stochastic Gradient Descent can be analyzed within a probabilistic framework using stochastic recursions of the form
\[
\underline{X}_k = \Ab_k \,\underline{X}_{k-1}+\underline{B}_k,
\] see \cite{ML1, ML2, DM24}. Notably, \eqref{eqn:classical_perp} characterizes the stationary distribution of the resulting Markov chain.

Our primary motivation for studying \eqref{eqn:free_perp} arises from the matricial case:
\[
\Xb \stackrel{d}{=} \Ab\Xb\Ab^\top + \Bb,\quad(\Ab,\Bb)\mbox{ and }\Xb\mbox{ are independent,}
\]
where $\Ab$, $\Bb$ and $\Xb$ are $N\times N$ random matrix, $\Bb^\top=\Bb$ a.s.
Although one may vectorize this equation and apply classical (vector) perpetuity theory, we propose an approach that preserves the matrix structure. This framework paves the way for studying properties such as the asymptotic behavior of the empirical spectral distribution of $\Xb$ as $N\to+\infty$ --- a topic that, to the best of our knowledge, has not been previously explored.

Before turning to matricial perpetuities, our first goal is to develop the corresponding theory for free random variables. Since many families of independent random matrices are asymptotically free as the dimension of matrices tend to $+\infty$, the free perpetuity naturally emerges as a limiting object. In a forthcoming work, we will investigate the connection between matricial and free perpetuities.

\begin{remark}
        The assumption $\A\ge 0$ in \eqref{eqn:free_perp} is made without loss of generality. Indeed, starting from the more general equation
        \[
        \X \stackrel{d}{=} \A\,\X\,\A^* + \B,
        \]
and writing the polar decomposition $\A=U|\A|$ (with $U$ unitary), we obtain
        \[
        \X \stackrel{d}{=} U\Bigl(|\A|\X|\A|+U^*\B\,U\Bigr)U^*
        \stackrel{d}{=} |\A|\X|\A|+U^*\B\,U,
        \]
        where the pair $(|\A|,U^*\B\,U)$ is $*$-free from $\X$. Thus, we may assume $\A$ is positive and redefine $(|\A|, U^* \B U)\mapsto(\A^{1/2},\B)$ to obtain \eqref{eqn:free_perp}.
\end{remark}

\subsection{Free multiplicative convolution powers and existence of free perpetuity}

Our strategy for proving the existence of the free perpetuity mirrors the classical approach: we analyze the convergence of the series
\begin{align}\label{intro:perpseries}
\mathbb{S}=\sum_{n=1}^{+\infty} \A_1^{1/2}\cdots \A_{n-1}^{1/2}\B_n\,\A_{n-1}^{1/2}\cdots \A_1^{1/2},
\end{align}
where $\{(\A_n,\B_n)\}_{n\geq 1}$ is an appropriate sequence of free copies.

In our setting, it suffices to establish convergence in distribution; in fact, we prove stronger notions of convergences. Assume that $\cA$ 
is a von Neumann algebra and $\tau\colon\cA \to \C$ is a faithful, normal, tracial state. By $\tcA$ we denote the algebra of unbounded operators affiliated with $\cA$.

In the particular case when $\A = \alpha\cdot 1_{\cA}$ for some $\alpha\in(0,1)$, the series \eqref{intro:perpseries} is a sum of free random variables. In this situation, results from \cite{Ber05} imply that the series converges almost uniformly (a.u.) provided that
\[
\tau\bigl(\log^+|\B|\bigr) < +\infty.
\]
Almost uniform convergence is the non-commutative analogue of almost sure convergence.
More generally, if $\tau(\A) < 1$, one may show that the series converges in norm when $\A,\B\in\cA$, and in $L^1$ if $\A,\B$ are unbounded operators affiliated with $(\cA,\tau)$ (with $\B\geq 0$ and $\tau(\B)<+\infty$). If $\tau(\A) > 1$, the perpetuity fails to exist. Thus, the most challenging--and interesting--case is when $\tau(\A)=1$. Then, $\X$ necessarily must be unbounded.

A key component in our analysis of the critical case $\tau(\A)=1$ is understanding the asymptotic behavior of moments under free multiplicative convolution. In particular, we prove the following theorem (see Theorems \ref{thm:28} and \ref{lem:1} for complete statements and proofs). Here, we write $f(x)\sim g(x)$ if the quotient
$f(x)/g(x)$ goes to $1$ as $x$ tends to $+\infty$. We also denote $m_\alpha(\mu) = \int_{\R} t^\alpha\,\mu(\dd t)$.

\begin{thm}
        Assume that $\mu$ is a probability measure on $[0,+\infty)$, which is non-Dirac and has unit mean. Then:
        \begin{enumerate}
    \item[(i)] If for $p\in\mathbb{N}$ we have $m_p(\mu)<+\infty$, then
                \[
    m_p(\mu^{\boxtimes n})\sim c_1 \,n^{p-1}.
                \]
                \item[(ii)] If $m_2(\mu)<+\infty$, then for any $\gamma\in(0,1)$,
                \[
                m_\gamma(\mu^{\boxtimes n})  \sim c_2\,n^{\gamma-1 }.
                \]
                \item[(iii)] If $\mu\big((t,+\infty)\big)\sim c\, t^{-\alpha}$ for some $\alpha\in(1,2)$ and $c>0$, then $m_2(\mu)=+\infty$ and, for any $\gamma\in(0,1)$,
                \[
                 m_\gamma(\mu^{\boxtimes n})  \sim c_3\, n^{\frac{\gamma-1}{\alpha-1} }.
                \]
        \end{enumerate}
The constants $c_1$, $c_2$ and $c_3$ are explicit.
\end{thm}

It is worth noting that while the classical multiplicative convolution $\circledast$ satisfies
\[
m_\gamma(\mu^{\circledast n}) = \bigl(m_\gamma(\mu)\bigr)^n,
\]
the free multiplicative convolution exhibits a markedly different asymptotic behavior.

The above theorem allows us to establish existence of free perpetuity in the critical case $\tau(\A)=1$.
\begin{thm}\label{intro:existence}
        Assume that $\A\ge0$ is non-Dirac, $\B$ is bounded and that $\tau(\A)=1$. Then the series \eqref{intro:perpseries}
        converges bilaterally almost uniformly. Moreover, if the tail of $\mu_\A$ satisfies
        \[
        \mu_\A\big((t,+\infty)\big)\sim c\,t^{-\alpha}\quad\text{for some } c>0 \text{ and } \alpha\in(1,2),
        \]
        then the series converges almost uniformly.
\end{thm}

\subsection{Subordination}
Subordination is a powerful tool in free probability. In its simplest form, if $\X$ and $\Y$ are free self-adjoint random variables, then their Cauchy transforms satisfy
\[
G_{\X+\Y}(z)=G_\X\bigl(\omega(z)\bigr),
\]
for some analytic self-map $\omega$ of $\C^+$. By combining the linearization trick  with Voiculescu's operator-valued subordination results (see \cite{BMS17}), similar subordination principles hold for general non-commutative polynomials in free variables.

In this work we apply subordination to analyze the tails of the free perpetuity, focusing on the polynomial
\[
\A\X\A+\B.
\]
The following theorem generalizes results from \cite{LS19} by permitting unbounded operators and removing the assumption that $\A$ is a function of $\B$.

\begin{thm}
Let $\X$, $\A$, $\B\in\tilde{\mathcal A}$ be three selfadjoint operators affiliated with the tracial $W^*$-probability
space $(\mathcal A,\tau)$, with $\A, \X\not\in\mathbb C\cdot1_{\cA}$.
Then there exists a pair of analytic self-maps of $\mathbb C^+$, denoted $(f,\mathsf f)$, such that
$$
\tau\left((\mathsf f(z)-\mathbb X)^{-1}\right)=\frac{1}{\mathsf f(z)+f(z)}
=\tau\left(\mathbb A(z-\mathbb B+f(z)\mathbb A^2)^{-1}\mathbb A\right),\quad z\in\mathbb C^+.
$$
In addition, the point $(f(z),\mathsf f(z))\in\mathbb C^+\times\mathbb C^+$ is the unique attracting fixed point of the map
$$
\mathscr F_z\colon\mathbb C^+\times\mathbb C^+\to\mathbb C^+\times\mathbb C^+,\quad
\mathscr F_z\begin{pmatrix} w_1\\ w_2\end{pmatrix}=
\begin{pmatrix}
\frac{1}{\tau\left((w_2-\mathbb X)^{-1}\right)}-w_2\\
\frac{1}{\tau\left(\mathbb A(z-\mathbb B+w_1\mathbb A^2)^{-1}\mathbb A\right)}-w_1
\end{pmatrix}.
$$
Moreover, if $\X$ is *-free form $(\A,\B)$, then the functions $f(z),\mathsf f(z)$ satisfy
\begin{gather*}
    \tau\left((\mathsf f(z)-\mathbb X)^{-1}\right)  =  \tau\left(\mathbb A(z-\mathbb B-\mathbb{AXA})^{-1}\mathbb A\right),\\
\tau\left((z-\mathbb B+f(z)\mathbb A^2)^{-1}\right)  =  \tau\left((z-\mathbb B-\mathbb{AXA})^{-1}\right).
\end{gather*}
\end{thm}

The subordination functions $\mathsf f$ and $f$ associated with the polynomial $\A\X\A+\B$ play a central role in analyzing the asymptotic behavior of the tails of free perpetuities.

\subsection{Tails of free perpetuities}

Once the existence of the free perpetuity in the critical case $\tau(\A)=1$ is established, it is natural to investigate its tail behavior. In this regime the perpetuity necessarily has unbounded support, and subordination becomes an essential tool.
By exploiting the asymptotic properties of the subordination functions, we obtain the following tail estimates.

\begin{thm}\label{intro:tails+}
        Assume $\tau(\A)=1$, where $\A\in\tcA$ is non-Dirac and $\B\in\cA$, with $\B\geq 0$ and
$\B\neq 0$. If $\A, \X\not\in\mathbb C\cdot1_{\cA}$, then the unique solution $\X$ to
\eqref{eqn:free_perp} exists, and moreover:
 \begin{enumerate}
     \item[(i)] if $\tau(\A^2)<+\infty$, then
         \begin{align*}
        \lim_{t\to+\infty}t^{1/2}\mu_{\X}\big((t,+\infty)\big) = \frac{2\sqrt{2}}{\pi}\sqrt{\frac{\tau(\B)}{\mathrm{Var}(\A)}}.
        \end{align*}
     \item[(ii)] if $\mu_\A\big((t,+\infty)\big)\sim c\, t^{-\alpha}$ for $\alpha\in(1,2)$ and $c>0$, then
\[
 \lim_{t\to+\infty}t^{1/\alpha}\mu_{\X}\big((t,+\infty)\big) = \frac{\sin(\pi/\alpha)}{\pi/\alpha}\left( \frac{-\sin(\pi\alpha)}{\pi\,c}\tau(\B)\right)^{1/\alpha}.
\]
 \end{enumerate}
 \end{thm}
Note that condition $\tau(\A)=1$ plays a similar role to Kesten's condition for determining the parameter $\alpha$.

        A surprising phenomenon emerges from these results. In case (i) the tail of $\mu_\X$ decays as $t^{-1/2}$ (so that $m_\gamma(\X)<+\infty$ for $\gamma\in(0,1/2)$), whereas in case (ii) when $\A$ has a regularly varying tail the decay is $t^{-1/\alpha}$, implying $m_\gamma(\X)<+\infty$ for $\gamma\in(0,1/\alpha)$. Since $\alpha\in(1,2)$, in the second case the free perpetuity $\X$ possesses more finite moments even though $\A$ itself has fewer moments. A similar phenomenon occurs in the next theorem, which treats the case when $\B$ is symmetric.

\begin{thm}
        Assume $\tau(\A)=1$, where $\A\in\tcA$ is non-Dirac, and $(\A,\B)\stackrel{d}{=}(\A,-\B)$ with
$\B\in\cA$. If $\A, \X\not\in\mathbb C\cdot1_{\cA}$, then the unique solution $\X$ to
\eqref{eqn:free_perp} exists, and moreover:
 \begin{enumerate}
     \item[(i)] if $\tau(\A^2)<+\infty$, then
        \[
        \lim_{t\to+\infty}t\,\mu_{|\X|}\big((t,+\infty)\big) = \frac{2}{\pi}\sqrt{\frac{\tau(\B^2)}{\mathrm{Var}(\A)}}.
        \]
 \item[(ii)] if $\mu_\A\big((t,+\infty)\big)\sim c\, t^{-\alpha}$ for $\alpha\in(1,2)$ and $c>0$, then
         \begin{align*}
 \lim_{t\to+\infty}t^{2/\alpha}\,\mu_{|\X|}\big((t,+\infty)\big) = \frac{\sin(\pi/\alpha)}{\pi/\alpha}
\left( \frac{-\sin(\pi\alpha)}{2\pi c}\tau(\B^2)\right)^{1/\alpha}.
        \end{align*}
 \end{enumerate}
 \end{thm}

\subsection{An explicit example of a perpetuity}
Explicit examples of perpetuities—cases in which the density can be expressed in closed form—are rare in both the classical and free settings. To illustrate our theory, we consider an example based on the free beta prime distribution. Following \cite{Yoshida}, define the free beta prime law by

Explicit examples of perpetuities (i.e., cases where the distribution can be written in closed form) are rare in both the classical and free settings. To illustrate our theory, we consider an example based on the free beta prime distribution. Following Yoshida \cite{Yoshida}, define the free beta prime law by
\[
f\mathcal{B}'_{a,b}(\dd x) = f_{a,b}(x)\dd x+\max\{1-a,0\}\delta_0(\dd x),
\]
where
\begin{align*}
        f_{a,b}(x)=\frac{(b-1)\sqrt{(\gamma_+-x)(x-\gamma_-)}}{2\pi x(1+x)}I_{[\gamma_-,\gamma_+]}(x)
\end{align*}
with
\[
\gamma_\pm = \left(\frac{\sqrt{ab}\pm\sqrt{a+b-1}}{b-1}\right)^2.
\]
This formula extends to the case $b=1$, in which case
\[
f\mathcal{B}'_{a,1}(\dd x) = \frac{\sqrt{4a x-(a-1)^2}}{2\pi x(1+x)}I_{[(a-1)^2/(4 a),+\infty)}(x)\dd x+\max\{1-a,0\}\delta_0(\dd x).
\]
A direct calculation (e.g., via the corresponding $S$-transforms) shows that for $a>0$, $b\ge1$, if $\A\sim f\mathcal{B}'_{a,a+b}$ and $\X\sim f\mathcal{B}'_{a,b}$ (with $\A$ and $\X$ free), then
\begin{align*}
        \X \stackrel{d}{=} \A^{1/2}\X\A^{1/2}+\A.
\end{align*}
Therefore, $\X$ is an example of a free perpetuity. Note that while the existence of perpetuity is verified here by a direct calculation, it is nevertheless guaranteed by Theorem \ref{intro:existence}.

Since $\tau(\A) = a/(a+b-1)<1$ if and only if $b>1$, in which case the support of $\X$ is compact. When $\tau(\A)=1$, then necessarily $b=1$, and thus $\X\sim f\mathcal{B}'_{a,1}$. In this critical case, one obtains
\[
S_\X(-1/x)=\frac{1}{ax-1}\sim \frac{1}{a}\,x^{-1}\quad\text{as } x\to+\infty,
\]
and by \cite[Theorem 4.5]{KK22} (with $\alpha=1/2$) it follows that
\[
\lim_{t\to+\infty} t^{1/2}\,\mu_{\X}\big((t,+\infty)\big) = \frac{2\sqrt{a}}{\pi}.
\]
A direct computation shows that this expression can be rewritten as
\[
\frac{2\sqrt{2}}{\pi}\sqrt{\frac{\tau(\A)}{\mathrm{Var}(\A)}},
\]
which is in full agreement with Theorem~\ref{intro:tails+}.

Another explicit example comes from \cite{KSzMY}. For a special choice of parameters in the free GIG  distribution  (see subsection \ref{sec:ex1} for detailed discussion), suppose $\X$ and $\Y$ are free, where $\X$ follows a GIG distribution and $\Y$ follows a free Poisson distribution. Define $\mathbb{V}=\X^{-1}-(\X+\Y)^{-1}$. Then, one obtains $\mathbb{V}\stackrel{d}{=}\Y$. Moreover, Hua's identity implies
\begin{align*}
\Y^{-1}\stackrel{d}{=} \X \Y^{-1}\X+\X,
\end{align*}
which is an affine fixed-point equation.

In \cite[Theorem 4.1]{Hasebe}, the free GIG distribution is characterized via a continued fraction fixed-point equation:
\[
\X\stackrel{d}{=}(\X+\Y)^{-1},
\]
where $\X$ and $\Y$ are free and $\Y$ has a free Poisson distribution. The method of the proof of \cite[Theorem 4.1]{Hasebe}
heavily relied on subordination techniques. In our work, we also employ subordination; however, we do not use it to establish the existence of a perpetuity.

\subsection{Organization of the paper}

The paper is organized as follows. In the next section, we introduce the necessary background and notation. We review several topics, including various modes of convergence in von Neumann algebras, the essential free probability transforms, and key Tauberian theorems. This material provides the theoretical framework required for the later sections.

In Section \ref{sec:3}, we present results on the free multiplicative random walk, a central object in the study of free perpetuities. We analyze the asymptotic behavior of moments and fractional moments of free multiplicative convolutions. We also analyze the conditions for series convergence in the critical case when $\tau(\A)=1$, laying the groundwork for our fixed point analysis.

Section \ref{sec:4} focuses on the affine fixed point equation. We illustrate these concepts with two examples. In addition, we address the existence and uniqueness of a free perpetuity and study their moments.

In Section \ref{sec:subordination}, we focus on the subordination phenomenon in the context of the affine transformation $\A\X\A+\B$. This section provides key insights into how subordination techniques can be applied to better understand the underlying structure of our model.

Finally, Section \ref{sec:tails} investigates the tail behavior of free perpetuities in the critical case, revisiting our two running examples.

An Appendix is included, which presents new results on the limit theorem from \cite{SY13} that are not directly related to free perpetuities.

\section{Background and notation}

Throughout this paper, we consider the framework of a tracial $W^{*}$ non-commutative probability space, by which we mean the a pair
$\left(\cA,\tau\right)$ where $\cA$ is a von Neumann algebra and $\tau\colon\cA \to \C$ is a faithful, normal, tracial state. By $\tcA$ we will denote the algebra of unbounded operators affiliated with $\cA$.

\subsection{Modes of convergence in von Neumann algebras}
One of the central themes of this work is the convergence of sequences of non-commutative random variables. In this subsection, we review the essential definitions and results that underpin our analysis.

In this work we deal with freeness of unbounded operators, for definition of freeness in this case we refer to \cite{BV93} or \cite[Chapter 8]{MS17}.
\begin{defin}\
\begin{enumerate}
    \item  A sequence $(x_n)_n$ in $\tcA$ converges almost uniformly to $x\in\tcA$ if, for any $\varepsilon>0$, there is a projection $p$ in $\cA$ such that $\tau(1-p)<\varepsilon$ and
    \[
    \lim_{n\to+\infty} \|(x_n-x)p\|=0.
    \]
    See \cite{Jajte91} for details.
\item A sequence $(x_n)_n$ in $\tcA$  converges bilaterally almost uniformly to $x\in\tcA$ if, for any $\varepsilon>0$, there is a projection $p$ in $\cA$ such that $\tau(1-p)<\varepsilon$ and
     \[
    \lim_{n\to+\infty} \|p(x_n-x)p\|=0.
    \]

\item
We say that $(x_n)_n$ is Cauchy bilaterally almost uniformly if for any $\eps, \eps'>0$ there are a projection $p$ with $\tau(1-p)\leq\eps$ and a positive integer $N$ such that
\[
\|p(x_n-x_m)p\|\leq \eps'\qquad\mbox{for all }n,m\geq N.
\]

\item Let $p\in(0,+\infty)$ and $\|x\|_p=\tau(|x^*x|^{p/2})^{1/p},x\in\mathcal A.$ It is known (and can easily be verified) that $x\mapsto\|x\|_p$ is a norm on $\mathcal A$
if $p\ge1$ and a quasinorm is $p\in(0,1)$. Thus, it generates a topological vector space structure on $\mathcal A$. We define $L^{p}(\mathcal A,\tau)$ to be the completion
of $\mathcal A$ in the topology generated by $\|\cdot\|_p$. If $p\ge1$, then $L^{p}(\mathcal A,\tau)$ is a Banach space (for more details, see \cite[Chapter IX]{Takesaki2}).

\item If $\mathcal A$ is a von Neumann algebra acting on the Hilbert space $\mathcal H$,
the strong operator (so) topology is the locally convex topology on $\mathcal A$
generated by the family of seminorms $\mathcal A\ni x\mapsto\|x\xi\|_2$, $\xi\in\mathcal H$,
and the weak operator (wo) topology is the locally convex topology on $\mathcal A$
generated by the family of seminorms $\mathcal A\ni x\mapsto|\langle x\xi,\eta\rangle|$, $\xi,\eta
\in\mathcal H$ (see \cite{Takesaki}).

\end{enumerate}
\end{defin}

It is known that $\tcA$ is complete with respect to a.u. and b.a.u. convergences, see
\cite[Theorem 2.3]{CLS05} and \cite[Theorem 2.2]{Lit24}. We note that the notion of ‘almost uniform convergence’ is stronger than ‘bilateral almost uniform convergence', but they coincide in the commutative case. By \cite{Cu71}, any noncommutative $L^1$-bounded martingale converges b.a.u., and any $L^2$-bounded martingale converges a.u. A recent article \cite{BAD} shows that, in contrast to the classical case, almost uniform convergence can fail for noncommutative $L^p$-bounded martingales when $1\leq p<2$.

\subsection{Free probability transforms}\label{sec:tr}
Let $\mathcal{M}$ and $\mathcal{M}_+$ denote the sets of probability measures on $\R$ and $[0,+\infty)$, respectively. Let $\mathcal{M}_S$ denote the set of symmetric probability measures on $\R$.

The moments and the variance of a measure $\mu$ are given by
\[
m_\alpha(\mu)=\int_{\R} t^\alpha\mu(\dd t)\quad\mbox{and}\quad \mathrm{Var}(\mu)=m_2(\mu)-m_1(\mu)^2.
\]

The moment transform of $\mu\in\mathcal{M}$ is defined by
        \[
        \psi_\mu(z)= \int_{[0,+\infty)} \frac{z t}{1-zt}\mu(\dd t),\qquad 1/z\in\mathbb{C}\setminus\mathrm{supp}(\mu).
        \]
 
Let $\delta=\mu(\{0\})<1$. If $\mu\in\mathcal{M}_+$, the function $\psi_\mu\colon (-\infty,0)\to (\delta-1,0)$ is invertible, and we denote its inverse by $\chi_\mu$.
 If $\mu\in\mathcal{M}_S$, then the function
 $\psi_\mu\colon i(-\infty,0)\to (\delta-1,0)$ is invertible, and its inverse is also denoted by $\chi_\mu$.
 
The $S$-transform of $\mu\in\mathcal{M}_+\cup\mathcal{M}_S$ is defined by 
        \[
        S_\mu(z) = \frac{z+1}{z} \chi_\mu(z),\qquad z\in(\delta-1,0).
        \]
The $S$-transform is multiplicative under the free multiplicative convolution $\boxtimes$, that is,
\[
S_{\mu\boxtimes \nu} = S_\mu \, S_\nu,
\]
where $\mu \in \mathcal{M}_+$ and $\nu \in \mathcal{M}_+\cup\mathcal{M}_S$, \cite[Theorem 7]{APA09}.

If $m_2(\mu)<+\infty$, then as $z\to 0-$,
\begin{align}\label{eq:S-2-terms}
S_\mu(z)= \frac{1}{m_1(\mu)} - \frac{\mathrm{Var}(\mu)}{m_1(\mu)^3} z + o(z).
\end{align}

 If $\mu$ is the distribution of $\X$ with respect to $\tau$, then we write $\psi_\mu=\psi_\X$,
$\chi_\mu=\chi_\X$ and $S_\mu=S_\X$.

\subsection{Tauberian theorems}
For real functions $f$, $g$ defined on a neighbourhood of infinity, we write $f(t)\sim g(t)$ if $f(t)/g(t)\to 1$ as $t\to+\infty$.
A positive measurable function $L$ defined in a neighbourhood of infinity is called slowly varying if $L(\lambda t)\sim L(t)$ for each $\lambda>0$.

\begin{thm}\label{thm:taub}
Let $\mu\in\mathcal{M}_+$ and let $L$ be a slowly varying function.
\begin{enumerate}
    \item[(i)] If $\alpha\in(0,1)$, then the following two conditions are equivalent: 
\[
\mu\big( (t,+\infty)\big) \sim \frac{L(t)}{t^\alpha} \qquad\iff\qquad -\psi_\mu\left( -\frac1t\right)\sim \frac{\pi\alpha}{\sin(\pi\alpha)}\frac{L(t)}{t^\alpha}.
\]
    \item[(ii)] If $\alpha\in(1,2)$, then the following two conditions are equivalent: 
\[
\mu\big( (t,+\infty)\big) \sim \frac{L(t)}{t^\alpha} \quad\iff\quad -\psi_\mu\left( -\frac1t\right)= \frac{m_1(\mu)}{t}+ \frac{\pi\alpha}{\sin(\pi\alpha)}\frac{L(t)}{t^\alpha}(1+o(1)).
\]
\end{enumerate}
\end{thm}
\begin{proof}
    Let $U_n(t)= \int_{[0,t]} x^n\mu(\dd x)$, $n=1,2$.

For $\alpha\in(0,1)$, we write
\begin{align*}
-\psi_\mu\left( -\frac1t\right) = \int_{[0,+\infty)} \frac{x}{t+x}\mu(\dd x) = \int_{[0,+\infty)} \frac{ \dd U_1(x)}{t+x}.
\end{align*}
If $\alpha\in(1,2)$, then  $m_1(\mu)<+\infty$ and we express $\psi_\mu$ as
\[
-\psi_\mu\left( -\frac1t\right) = \frac{m_1(\mu)}{t}-  \frac{1}{t} \int_{[0,+\infty)} \frac{\dd U_2(x)}{t+x}.
\]
By \cite[Theorem 1.6.4]{BGT89}, for $0<\alpha<n$, as $t\to+\infty$,
\[
\mu\big( (t,+\infty)\big) \sim \frac{L(t)}{t^\alpha}  \qquad\iff\qquad U_n(t)\sim \frac{\alpha}{n-\alpha} t^{n-\alpha}L(t).
\]   
Moreover, by \cite[Theorem 1.7.4]{BGT89} (with $\rho=1$ and $\sigma=n-\alpha$), we have as $t\to+\infty$,
\[
U_n(t)\sim \frac{\alpha}{n-\alpha} t^{n-\alpha}L(t)\qquad\iff\qquad \int_{[0,+\infty)} \frac{\dd U_n(x)}{t+x}\sim c_{n,\alpha} t^{n-\alpha-1} L(t),
\]
where by Euler's reflection formula, we have
\[
c_{n,\alpha} = \frac{\alpha}{n-\alpha} \Gamma(1-(n-\alpha)) \Gamma(n-\alpha+1) = \frac{\pi\alpha}{\sin(\pi(n-\alpha))} = (-1)^{n+1} \frac{\pi\alpha}{\sin(\pi\alpha)}.
\]
\end{proof}

\begin{thm}\label{thm:taub2}
Let $\mu\in\mathcal{M}_S$ and let $L$ be a slowly varying function.
 For $\alpha\in(0,2)$, the following two conditions are equivalent: as $t\to+\infty$,
\[
\mu\big( (t,+\infty)\big) \sim \frac{L(t)}{2\,t^\alpha} \qquad\iff\qquad -\psi_\mu\left( -\frac1{it}\right)\sim \frac{\pi\alpha/2}{\sin(\pi\alpha/2)}\frac{L(t)}{t^\alpha}.
\]
\end{thm}
\begin{proof}
    Let $\mu^2\in\mathcal{M}_+$ be the pushforward measure of $\mu$ by the mapping $x\mapsto x^2$. Then,
\[
\mu\big( (t,+\infty)\big) \sim \frac{L(t)}{2\,t^\alpha} \qquad\iff\qquad \mu^2\big( (t,+\infty)\big) \sim \frac{L(t^{1/2})}{t^{\alpha/2}}.    
 \]
Since $t\mapsto L(t^{1/2})$ is slowly varying and $\alpha/2\in(0,1)$, by Theorem \ref{thm:taub} (i), we have the equivalence as $t\to+\infty$,
\[
\mu^2\big( (t,+\infty)\big) \sim \frac{L(t^{1/2})}{t^{\alpha/2}}\qquad\iff\qquad -\psi_{\mu^2}\left( -\frac1t\right)\sim \frac{\pi\alpha/2}{\sin(\pi\alpha/2)}\frac{L(t^{1/2})}{t^{\alpha/2}}.
\]
Finally, since $\mu\in\mathcal{M}_S$, we have
\[
-\psi_\mu\left( -\frac1{it}\right) = -\psi_{\mu^2}\left( -\frac1{t^2}\right),
\]
which proves the assertion.
\end{proof}

\section{Free multiplicative random walk}\label{sec:3}

Assume that $\A=\A^\ast\in\tcA$ and $\A\geq 0$. Let $\mu$ be the distribution of $\A$ and let $(\A_n)_n$ be a sequence of freely independent copies of $\A$. Define
\begin{align*}
\Pi_n^{\uparrow} &= \A_1^{1/2}\cdots\A_{n-1}^{1/2}\A_n \A_{n-1}^{1/2}\cdots\A_1^{1/2},
\end{align*}
For each $n\in\mathbb{N}$, we have $
\Pi_n^{\uparrow}\sim \mu^{\boxtimes n}$.
We will denote by $\Pi_n$ a generic element with distribution $\mu^{\boxtimes n}$.
We note that if $m_1(\mu)=1$, then $(\Pi_n^\uparrow)_{n\in\mathbb{N}}$ is a non-commutative martingale, \cite{Cu71}.

\subsection{Free multiplicative law of large numbers}\label{sec:LLN}
In the following section, we present several results that describe the asymptotic behavior of the distribution of $\Pi_n$.
\begin{thm}[]
    Suppose $\mu\in\mathcal{M}_+$ has a compact support.
    Let $L_n =\sup\left(\mathrm{supp}(\mu^{\boxtimes n})\right)$ be the right endpoint of the support of $\Pi_n$.
    Then,
    \[
    \lim_{n\to+\infty} \frac{L_n}{n\, m_1(\mu)^n} = e \frac{\mathrm{Var}(\mu)}{m_1(\mu)^2}.
    \]
\end{thm}
\begin{proof}
 By \cite[Theorem 1]{Kargin08}, the claim holds for the case $m_1(\mu)=1$. The general case can be deduced by applying this result to the dilated measure defined by $\tilde{\mu}(B) =  \mu(m_1(\mu) B)$ for any Borel set $B$. Clearly, $\tilde{\mu}$ has a unit mean, $\mathrm{Var}(\tilde{\mu}) = \mathrm{Var}(\mu)/m_1(\mu)^2$ and $L_n = m_1(\mu)^n \tilde{L}_n$, where $\tilde L_n =\sup\left(\mathrm{supp}(\tilde\mu^{\boxtimes n})\right)$.
\end{proof}
\begin{remark}\label{rem:m<1}
    If $\mu\in\mathcal{M}_+$ has a compact support and $m_1(\mu)<1$, then $L_n\to 0$ and therefore
    \[
    \lim_{n\to+\infty}\| \Pi_n \| = 0.
    \]
    Moreover, in this case, we see that the series
    $\sum_{n=1}^{+\infty} \Pi_n$
    converges in norm. Let $\X_n = \sum_{k=1}^n  \Pi_k$; one has $\| \Pi_k \|\sim c \,k\,m_1(\mu)^k$. Therefore, for $n>m\geq N$,
    \[
    \| \X_n-\X_m\| \leq \sum_{k=N}^{+\infty} \| \Pi_k \| <\eps
    \]
    for sufficiently large $N$.
\end{remark}
The following law of large numbers for free multiplicative convolutions points out a fundamental difference from the usual LLN, where one has almost sure convergence to a constant. Here, the limiting distribution is nontrivial unless $\mu$ is Dirac. 
\begin{thm}[Theorem 2 in \cite{HaM13}]
Let $\mu\in\mathcal{M}_+$ with $\delta=\mu(\{0\})$. Then,
\[
\mu_{\Pi_n^{1/n}} \stackrel{w}{\longrightarrow} \nu,
\]
where $\nu(\{0\}) = \delta$ and for $t\in(\delta,1)$,
\[
\nu\left( \left[0, \frac{1}{S_\mu(t-1)}\right] \right) = t.
\]
\end{thm}

\begin{remark}
    The support of $\nu$ is the closure of the interval
\begin{align*}
\left( \frac{1}{m_{-1}(\mu)} , m_1(\mu) \right).
\end{align*}
Since $S_\mu$ is analytic, $\nu$ has no atoms except possibly at $0$.
\end{remark}

\begin{Corollary}
If $\mu\in\mathcal{M}_+$ and $m_1(\mu)=1$, then
\[
\mu_{\Pi_n}=\mu^{\boxtimes n}\stackrel{w}{\longrightarrow}\delta_0.
\]
\end{Corollary}
Indeed, by the previous remark, we have $\nu\big( [0,1)\big)=1$.
Therefore, for any $s,t\in(0,1)$ and $n\geq \log(t)/\log(s)$, we obtain
\[
\mu_{\Pi_n}\big((t,+\infty)\big) = \mu_{\Pi_n^{1/n}}\big((t^{1/n},+\infty)\big)\leq\mu_{\Pi_n^{1/n}}\big((s,+\infty)\big) \to \nu\big((s,+\infty)\big) = \nu\big((s,1)\big).
\]
By letting $s\to1-$, the result follows.

\subsection{Moments}

In this subsection we study the asymptotic behavior of free moments of the $n$-fold free multiplicative convolution of a measure $\mu$ which has itself $p$-th moment finite. The next subsection studies the analogous problem for fractional moments. The main result of this subsection is the following.
\begin{thm}\label{thm:28} Assume that $\mu\in\mathcal{M}_+$ is non-Dirac. For any $p\in\mathbb{N}$ such that $m_p(\mu)<+\infty$, we have
    \[
  \lim_{n\to+\infty} n^{1-p} \frac{m_p(\mu^{\boxtimes n})}{m_1(\mu)^{n p}} = \frac{(\eta\,  p)^{p-1}}{p!},
    \]
    where $\eta = \mathrm{Var}(\mu)/m_1(\mu)^2$.
\end{thm}

In order to prove the above result, we will use a combinatorial description of the free multiplicative convolution. For a detailed discussion, we refer to \cite{NS06}. Additionally, we note that when $\mu$ has compact support, the asymptotics bounds for $m_p(\mu^{\boxtimes n})$ can be derived from the results of \cite{AV12}.

 A partition $\pi$ of $[n]:=\{1,\ldots,n\}$ is a set of subsets of $[n]$ of the form $\pi=\{V_1,\ldots,V_k\}$, where $V_i\cap V_j=\emptyset$ for $i\neq j$ and $V_1\cup\ldots \cup V_k=[n]$. The subsets $V_1,\ldots,V_k$ are called blocks of $\pi$ and we write $V_i\in \pi$. Let $|\pi|$ denote the number of blocks of $\pi$ and $|V_i|$ the size of the $i$-th block.

We say that $\pi$ is a non-crossing partition of $[n]$ if for $i_1, i_3\in V$ and $i_2, i_4\in V'$ with  $V,V'\in \pi$ condition $i_1<i_2<i_3<i_4$ implies $V=V'$. The set of non-crossing partitions on $[n]$ is denoted by $\NC(n)$. For $\pi,\rho\in \NC(n)$, we write $\pi\leq \rho$ if every block of $\pi$ is contained in some block of $\rho$. The partial order $\leq$ is called the reversed refinement order and it defines a lattice structure on $\NC(n)$. We denote by $0_n$ and $1_n$ respectively the minimal and maximal elements of $\NC(n)$ with respect to $\leq$.

For $\pi\in \NC(n)$, we define its Kreweras complement $\Kr(\pi)\in \NC(n)$ as follows:consider set $\{1,\overline{1},2,
\overline{2},\ldots,n,\overline{n}\}$, place $\pi$ on $\{1,\ldots,n\}$, and then $\Kr(\pi)$ is the maximal partition on $\{\overline{1},\ldots,\overline{n}\}$ such that $\pi\cup \Kr(\pi)$ is a non-crossing partition on $\{1,\overline{1},\ldots,n,\overline{n}\}$.

Let $\kappa_p$ be the free cumulant of order $p\in\mathbb{N}$.   For $\pi\in \NC(n)$, we denote
\begin{align*}
    \kappa_{\pi}(\A_1,\ldots,\A_n)&=\prod_{V\in\pi}  \kappa_{|V|}(\A_1,\ldots,\A_n|V),\\
    \tau_{\pi}(\A_1,\ldots,\A_n)&=\prod_{V\in\pi} \tau_{|V|}(\A_1,\ldots,\A_n|V).
\end{align*}
where for $V=\{i_1,\ldots,i_k\}$ with $i_1<i_2<\ldots<i_k$, we have $\kappa_{|V|}(\A_1,\ldots,\A_n|V)=\kappa_k(\A_{i_1},\ldots,\A_{i_k})$, and $\tau_{|V|}(\A_1,\ldots,\A_n|V)=\tau(\A_{i_1}\ldots\A_{i_k})$. If $\A_1=\ldots=\A_n=\A$  will use a short hand notation $\kappa_\pi(\A)$ instead of $\kappa_{\pi}(\A,\ldots,\A)$.

We will use the following two formulas, see Theorem 14.4 in \cite{NS06}. Let $\{\A_1,\ldots,\A_n\}$ and $\{\B_1,\ldots,\B_n\}$ be free families. Then,
\begin{align} \label{eqn:cumprod}
\begin{split}
\kappa_n(\A_1\B_1,\ldots,\A_n\B_n)=&\sum_{\pi\in \NC(n)}\kappa_\pi(\A_1,\ldots,\A_n)\kappa_{\Kr(\pi)}(\B_1,\ldots,\B_n),
\\
\tau(\A_1\B_1\ldots\A_n\B_n)=&\sum_{\pi\in \NC(n)}\kappa_\pi(\A_1,\ldots,\A_n)\tau_{\Kr(\pi)}(\B_1,\ldots,\B_n).
\end{split}
\end{align}

We will need a combinatorial formula, which must have been noticed before, however we did not find a good reference and we present it with an easy proof.

\begin{lemma}\label{lem:formula}
       \begin{align*}
        \sum_{r=1}^{p-1} \frac{r^{r-1}}{r!}\frac{(p-r)^{p-r}}{(p-r)!}=(p-1)\frac{p^{p-1}}{p!},\qquad p\in\mathbb{N}.
    \end{align*}
    \end{lemma}
\begin{proof}
Recall Abel's binomial formula
    \begin{align*}
        \sum_{r=0}^{p} \binom{p}{r} (w+r)^{r-1}(z+w+p)^{p-r}=w^{-1}(z+w+p)^p
    \end{align*}
which holds for $w\neq 0$. Set $z=0$. Separating the terms for $r=0$ and $r=p$, we obtain
    \begin{align*}
        \sum_{r=1}^{p-1} \binom{p}{r} (w+r)^{r-1}(p-r)^{p-r}=\frac{(w+p)^{p}-(p)^{p}}{w}-(w+p)^{p-1}.
    \end{align*}
    Observe that the left-hand side above is a polynomial in $w$. Letting $w\to0$, we derive the assertion. Alternatively this formula can be proved by a combinatorial argument which goes by using Cayley's formula, saying that $p^{p-1}$ represents the number of rooted labeled trees on $p$ vertices.
\end{proof}

The following result is easy to establish, e.g., by comparing the coefficients of $z$ in the expansions of $S_{\mu^{\boxtimes n}}(z)$ and $S_\mu^n(z)$ (recall \eqref{eq:S-2-terms}).
\begin{Corollary}\label{lem:secondCumulant}
Let $\mu\in\mathcal{M}_+$ with $m_2(\mu)<+\infty$. Then,
\[
\mathrm{Var}(\mu^{\boxtimes n}) = n\, \mathrm{Var}(\mu) \,m_1(\mu)^{2(n-1)},\qquad n\in\mathbb{N}.
\]
\end{Corollary}

\begin{proposition}\label{prop:Kamil}
    Let $\A_1,\A_2\ldots$ be freely independent copies of $\A$ and assume $\A\geq 0$. For any $p\in\mathbb{N}$ such that $\tau(\A^p)<+\infty$, we have \begin{align}\label{eqn:cumulantAsymp}
    \frac{\kappa_p(\Pi_n)}{n^{p-1} \tau(\A)^{p\, n}} = \frac{(\eta\,  p)^{p-1}}{p!}+O\left(\frac{1}{n}\right),
    \end{align}
    where $\eta=\mathrm{Var}(\A)/\tau(\A)^2$.
\end{proposition}
\begin{proof}
 For $p=1$, the assertion holds trivially, while for $p=2$ it follows directly from Corollary \ref{lem:secondCumulant}. In both cases the $O(1/n)$ term is equal to zero. Without loss of generality, we assume that $\tau(\A)=1$, implying $\eta=\mathrm{Var}(\A)=\kappa_2(\A)$.
 
 Fix $m\geq 3$ and assume that $\tau(\A^m)< +\infty$.   Suppose that \eqref{eqn:cumulantAsymp} holds for all $p=1,\ldots,m-1$; we aim to prove it for $p=m$. 
 
By traciality, $\kappa_m(\Pi_k)=\kappa_m(\A_k\Pi_{k-1})$ for all $k\in\mathbb{N}$. Applying formula \eqref{eqn:cumprod} and using the fact that all $\A_i$ have the same distribution as $\A$, we obtain for any $k\geq 2$,
\begin{align}\label{eqn:proofCumulantAsymptotics}
        \kappa_m(\Pi_k)-\kappa_m(\Pi_{k-1})=\sum_{\substack{\pi\in \NC(m)\\ |\Kr(\pi)|=2}}\kappa_{\pi}(\A)\kappa_{\Kr(\pi)}(\Pi_{k-1})+\sum_{\substack{\pi\in \NC(m)\\ |\Kr(\pi)|>2}}\kappa_{\pi}(\A)\kappa_{\Kr(\pi)}(\Pi_{k-1}).
    \end{align}
   By the induction hypothesis, for any $p<m$, there exists $C_p>0$ such that for all $k\geq 2$,
    $|\kappa_p(\Pi_{k-1})|\leq k^{p-1} C_p$.
    Consequently, for any $\pi\in \NC(m)$ with $|\Kr(\pi)|>2$, we have
\begin{align*}
    \left|\kappa_{\Kr(\pi)}(\Pi_{k-1})\right|\leq k^{m-|\Kr(\pi)|} \prod_{V\in \Kr(\pi)} C_{|V|} = O(k^{m-3}).
    \end{align*}
and therefore the second sum in \eqref{eqn:proofCumulantAsymptotics} is of order $O(k^{m-3})$.

When \(|\Kr(\pi)| = 2\), the partition \(\pi \in \NC(m)\) contains one block that is a pair and \(m-2\) singletons.
The value of $\kappa_{\Kr(\pi)}(\Pi_{k-1})$ depends on the distance between the two elements in the pair. Specifically, if $i<j$ and $\{i,j\}$ is the unique pair of $\pi$ with $j-i=r$, we have  $\kappa_{\Kr(\pi)}(\Pi_{k-1})=\kappa_{r}(\Pi_{k-1})\kappa_{m-r}(\Pi_{k-1})$. There are $m-r$ such pairs $\{i,j\}$ for $r\in\{1,\ldots,m-1\}$. Therefore,
\begin{align}\label{eq:firstsum}
        \sum_{\substack{\pi\in \NC(m)\\ |\Kr(\pi)|=2}}\kappa_{\pi}(\A)\kappa_{\Kr(\pi)}(\Pi_{k-1})=\sum_{r=1}^{m-1} (m-r) \kappa_{2}(\A) \kappa_{r}(\Pi_{k-1})\kappa_{m-r}(\Pi_{k-1}).
    \end{align}
By the induction hypothesis, for $r=1,\ldots,m-1$,
    \begin{align*}
        \kappa_{r}(\Pi_{k-1})&=(k-1)^{r-1}\frac{(\eta\, r)^{r-1}}{r!}+O(k^{r-2}),\\
        \kappa_{m-r}(\Pi_{k-1})&=(k-1)^{m-r-1}\frac{(\eta \,(m-r))^{m-r-1}}{(m-r)!}+O(k^{m-r-2}).
    \end{align*}
Substituting these into \eqref{eq:firstsum} gives
    \begin{align*}
 (k-1)^{m-2}\eta^{m-1}
        \sum_{r=1}^{m-1} \frac{r^{r-1}}{r!}\frac{(m-r)^{m-r}}{(m-r)!}+O(k^{m-3}).
    \end{align*}
By Lemma \ref{lem:formula}, this simplifies to
    \begin{align*}
\sum_{\substack{\pi\in \NC(m)\\ |\Kr(\pi)|=2}}\kappa_{\pi}(\A)\kappa_{\Kr(\pi)}(\Pi_{k-1})=(k-1)^{m-2}(m-1)\frac{(\eta\, m)^{m-1}}{m!}+O(k^{m-3}).
    \end{align*}
  Summing \eqref{eqn:proofCumulantAsymptotics} over $k=2,
    \ldots,n$, we obtain
\begin{align*}\label{eq:twosums}
        \kappa_{m}(\Pi_n)-\kappa_{m}(\A)& =\sum_{k=2}^n \left( (k-1)^{m-2}(m-1)\frac{(\eta\, m)^{m-1}}{m!}+O(k^{m-3}) \right)\\
        &=\frac{(\eta\, m)^{m-1}}{m!} (m-1) \sum_{k=2}^{n}(k-1)^{m-2}+O(n^{m-2}).
    \end{align*}
Using the asymptotic formula  $(m-1) \sum_{k=2}^{n}(k-1)^{m-2} = n^{m-1} + O(n^{m-2})$, the result follows.
\end{proof}

\begin{proof}[Proof of Theorem \ref{thm:28}]
    The assertion follows easily from Proposition \ref{prop:Kamil}. As before, without loss of generality, we assume that $m_1(\mu)=1$. Let $\A_1,\A_2\ldots$ be freely independent copies of $\A\sim \mu$. We have
    \begin{align*}
        m_p(\mu^{\boxtimes n}) & = \tau(\Pi_n^p) =\sum_{\pi\in \NC(p)}\kappa_{\pi}(\Pi_n)=\sum_{\pi\in \NC(p)} \prod_{V\in \pi} \kappa_{|V|}(\Pi_n)\\&=\sum_{\pi\in \NC(p)} \prod_{V\in \pi} \left(\frac{\left(\eta\, |V|\right)^{|V|-1}}{|V|!} n^{|V|-1}+O(n^{|V|-2})\right),
    \end{align*}
    where the last equality follows from \eqref{eqn:cumulantAsymp}.
    For each fixed $p$, this is a sum over a finite set. The dominating term in $n$ in this sum is the one corresponding to $\pi=1_p$, and is given by $(\eta\, p)^{p-1}/p! \,n^{p-1}$. All other terms are of smaller order. This completes the proof.
\end{proof}

\subsection{Fractional moments}
In this section, we find the asymptotics of $m_\gamma(\mu^{\boxtimes n})=\tau(\Pi_n^\gamma)$ for $\gamma\in(0,1)$ and $\mu\in\mathcal{M}_+$. These asymptotics differ drastically from the case of classical multiplicative convolution $\circledast$, where we have $m_\gamma(\mu^{\circledast n}) = m_\gamma(\mu)^n$. Note that if $m_1(\mu)=1$, and $\mu$ is non-Dirac, then by the strict Jensen inequality, we have $m_\gamma(\mu)<1$ for all $\gamma\in(0,1)$.

The main result in this section is the following.
\begin{thm}\label{lem:1}
Assume that $\mu\in\mathcal{M}_+$ is non-Dirac and $m_1(\mu)=1$.
\begin{enumerate}
        \item[(i)] For any $\gamma\in(0,1)$,
\[
\lim_{n\to+\infty} n^{1-\gamma }m_\gamma(\mu^{\boxtimes n})  =  \frac{(\mathrm{Var}(\mu) \gamma)^{\gamma-1}}{\Gamma(1+\gamma)}.
\]
\item[(ii)] For $\gamma_n=1/\log(n)$,
\begin{align}\label{eq:limit}
        \lim_{n\to+\infty} \frac{n}{\log(n)} m_{\gamma_n}(\mu^{\boxtimes n})  = \frac{e}{\mathrm{Var}(\mu)},
\end{align}
If $m_2(\mu)=+\infty$, then the limits in (i) and (ii) equal $0$.
\item[(iii)] If $\mu\big((t,+\infty)\big)\sim c\, t^{-\alpha}$ for $\alpha\in(1,2)$, then $\mathrm{Var}(\mu)=+\infty$ and, for any $\gamma\in(0,1)$,
\[
\lim_{n\to+\infty} n^{\frac{1-\gamma}{\alpha-1} }m_\gamma(\mu^{\boxtimes n})  = \frac{\sin(\pi\gamma)}{\pi\gamma} \Gamma\left(\frac{1-\gamma}{\alpha-1}\right) \frac{1}{\alpha-1}\left(\frac{-\sin(\pi\alpha)}{\pi\alpha \,c}\right)^{\frac{1-\gamma}{\alpha-1}}.
\]
\end{enumerate}
\end{thm}
\begin{remark}
    Observe that the asymptotic from $(i)$ in the above theorem is very similar to the analogous asymptotic found for integer moments in Theorem \ref{thm:28}
\end{remark}
The proof of the above theorem will be based on a couple of lemmas. The following result generalizes \cite[Lemma 10]{HaM13}, which assumed that $\mu(\{0\})=0$.
\begin{lemma}\label{lem:fract_mom}
Let $\mu\in\mathcal{M}_+$ with $\delta=\mu(\{0\})<1$. If $m_\gamma(\mu)<+\infty$ for $\gamma\in(0,1)$, then
    \[
\int_{[0,+\infty)} x^\gamma \mu(\dd x) = \frac{\sin(\pi\gamma)}{\pi\gamma} \int_0^{1-\delta} \left(\frac{1-t}{t} \frac{1}{S_\mu(-t)}\right)^\gamma\dd t.
\]
\end{lemma}
\begin{proof}
    We have
    \begin{align*}
\int_0^{+\infty} y^{-\gamma} \psi'_\mu(-y)\dd y &= \int_{[0,+\infty)}\int_0^{+\infty} \frac{y^{-\gamma} x }{(1+y x)^2} \dd y\,\mu(\dd x)   \\
&= \int_{[0,+\infty)} x^\gamma \int_0^{+\infty} \frac{1}{u^\gamma(1+u)^2}\dd u\,\mu(\dd x) \\
&=\frac{\pi\gamma} {\sin(\pi\gamma)}\int_{[0,+\infty)} x^\gamma\mu(\dd x),
    \end{align*}
    where in the first equality we have used Tonelli's theorem, and in the second equality we made the substitution $u=yx$. The third equality follows from the fact that $\int_0^{+\infty} u^{-\gamma}(1+u)^{-2}\dd u = \pi\gamma/\sin(\pi\gamma)$ for $\gamma\in(-1,1)$.
   
    Next, in the initial integral, we substitute $y=-\chi_\mu(-t)$. Since $\chi_\mu\colon (\delta-1,0)\to (-\infty,0)$, we obtain
    \begin{align*}
\int_0^{+\infty} y^{-\gamma} \psi'_\mu(-y)\dd y = \int_0^{1-\delta} (-\chi_\mu(-t))^{-\gamma} \dd t.
\end{align*}
The assertion follows from the definition of the $S$-transform of $\mu$.
\end{proof}

\begin{lemma}\label{lem:CS}
For $\mu\in\mathcal{M}_+$ and $y>0$, define
$\kappa(y) = \int_{[0,+\infty)} (1+y x)^{-1}\mu(\dd x)$ and $\mu_y\in\mathcal{M}_+$ by
\begin{align}\label{eq:nuy}
\mu_y(\dd x) = \frac{1}{\kappa(y)}  \frac{1}{1+y x}\mu(\dd x).
\end{align}
\begin{enumerate}
    \item[(i)] If $\mu$ is non-Dirac, then the function $(0,+\infty)\ni y\mapsto m_1(\mu_y)$ is strictly decreasing.
    \item[(ii)] If $m_1(\mu)<+\infty$, then
    \[
    m_1(\mu_y) = m_1(\mu)-\frac{y}{1+y\, m_1(\mu)} \int_{[0,+\infty)}(x-m_1(\mu))^2\mu_y(\dd x).
    \]
    \begin{enumerate}
        \item If $m_2(\mu)<+\infty$, then
\[        \lim_{y\to 0+} \frac{m_1(\mu_y)- m_1(\mu)}{y} = -\mathrm{Var}(\mu).
        \]
        \item If $m_2(\mu)=+\infty$, then
        \[
        \lim_{y\to 0+} \frac{m_1(\mu_y)- m_1(\mu)}{y} = -\infty.
        \]
        \item If $\mu\big((t,+\infty)\big)\sim c\, t^{-\alpha}$ for $\alpha\in(1,2)$ and $c>0$, then as $y\to 0+$,
        \[
         m_1(\mu_y) = m_1(\mu)
        + \frac{\pi\alpha}{\sin(\pi\alpha)}c\, y^{\alpha-1}(1+o(1)).
        \]
    \end{enumerate}
\end{enumerate}
\end{lemma}
\begin{remark}
Note that we have $\kappa(y) = 1+\psi_\mu(-y)$.

Moreover, the measures defined in \eqref{eq:nuy} constitute a Cauchy-Stieltjes kernel family in the sense of \cite{CS11}. Point (i) of Lemma \ref{lem:CS} follows from \cite[Section 2]{CS11}; we include its proof for completeness.
\end{remark}
\begin{proof}[Proof of Lemma~\ref{lem:CS}]
    We have
    \[
    m_1(\mu_y) = \frac{ 1 }{\kappa(y)}\int_{[0,+\infty)} \frac{x}{1+y x}\mu(\dd x) = \frac{1-\kappa(y)}{y\,\kappa(y)}.
    \]
    Clearly, the function $(0,+\infty)\ni y\mapsto  m_1(\nu_y)$ is differentiable, and we find that
    \[
    \frac{\dd}{\dd y}m_1(\mu_y) = \frac{\kappa(y)^2-\kappa(y)-y\,\kappa'(y)}{y^2 \kappa(y)^2}.
    \]
Since $\mu$ is non-degenerate, we apply the strict Jensen inequality to obtain
    \[
    \kappa(y)^2-\kappa(y)-y\, \kappa'(y) = \left(\int_{[0,+\infty)} \frac{1}{(1+y x)}\mu(\dd x)\right)^2-\int_{[0,+\infty)} \frac{1}{(1+y x)^2}\mu(\dd x) <0,
    \]
    which proves (i).

If $m_1(\mu)<+\infty$, we have
\[
m_1(\mu_y)-m_1(\mu) = \frac{1}{\kappa(y)} \int_{[0,+\infty)} \frac{x-m_1(\mu)}{1+y x}\mu(\dd x).
\]
Moreover,
\[
\frac{x-m_1(\mu)}{1+x y} = \frac{x-m_1(\mu)}{1+y\, m_1(\mu)}-\frac{y}{1+y\, m_1(\mu)}\frac{(x-m_1(\mu))^2}{1+x y}.
\]
Points (a) and (b) are immediate consequences of the formula stated in (ii).

For (c), $y\to 0+$, we have
\[
m_1(\mu)-m_1(\mu_y) \sim \frac{y}{\kappa(y)} \int_{[0,+\infty)}(x-m_1(\mu))^2\mu(\dd x) \sim \int_0^\infty \frac{\dd U(x)}{y^{-1}+x},
\]
where $U(x) = \int_{[0,x]} (t-m_1(\mu))^2 \mu(\dd t)$. Let $U_2(x) = \int_{[0,x]} t^2 \mu(\dd t)$. As in the proof of Theorem \ref{thm:taub}, we find that as $t\to+\infty$,
\[
\mu\big((t,+\infty)\big)\sim c\, t^{-\alpha}\qquad\iff\qquad  U_2(t)\sim c \frac{\alpha}{2-\alpha} t^{2-\alpha}.
\]
Since $U_2(t)\sim U(t)$, we again apply the argument from the proof of Theorem \ref{thm:taub} to obtain, as $y\to0+$,
\[
\int_0^{+\infty} \frac{\dd U(x)}{y^{-1}+x} \sim - \frac{\pi\alpha}{\sin(\pi\alpha)}c\, y^{\alpha-1}.
\]
\end{proof}

Now, we are ready to prove Theorem \ref{lem:1}.
\begin{proof}[Proof of Theorem \ref{lem:1}]
(i) Let $\delta = \mu(\{0\})$. By \cite[Theorem 4.1]{Bel03}, we have $\mu_{\Pi_n}(\{0\})=\delta$.
Since $S_{\Pi_n}=S_\mu^n$, by Lemma \ref{lem:fract_mom},
we obtain for $\gamma\in(0,1)$,
        \begin{align}
m_\gamma(\mu^{\boxtimes n})
 &= \frac{\sin(\pi\gamma)}{\pi\gamma} \int_0^{1-\delta}  \left(\frac{1-t}{t}\frac{1}{S_{\mu}(-t)^n}\right)^{\gamma}\dd t\nonumber\\
  &= \frac{\sin(\pi\gamma)}{\pi\gamma} \int_0^{1-\delta}  \left(\frac{t}{1-t}\frac{1}{(-\chi_\mu(-t))}\right)^{(n-1)\gamma} (-\chi_{\mu}(-t))^{-\gamma}\dd t\nonumber\\
  &= \frac{\sin(\pi\gamma)}{\pi\gamma} \int_{[0,+\infty)}\int_0^{+\infty}  \left(\frac{-\psi_\mu(-y)\frac1y}{1+\psi_\mu(-y)}\right)^{(n-1)\gamma} y^{-\gamma}\frac{x}{(1+x y)^2}\dd y \,\mu(\dd x)
  ,\label{eq:psiint}
\end{align}
 where we have substituted $(0,1-\delta)\ni t\mapsto-\chi_\mu(-t)=y\in(0,+\infty)$ and used the fact that  $\psi_\mu'(-y) = \int_{[0,+\infty)} \frac{x}{(1+x y)^2}\mu(\dd x)$.
Let
\begin{align*}
m(y)=\frac{-\psi_\mu(-y)\frac1y}{1+\psi_\mu(-y)}.
\end{align*}
We have $m(y)=m_1(\mu_y)$, where $\mu_y$ is defined by \eqref{eq:nuy}. Consequently, the results of Lemma \ref{lem:CS} are applicable.

To prove parts (i) and (ii), we consider a more general situation: if $\gamma_n\to \gamma_\ast\in[0,1)$ with $\gamma_nn\to+\infty$ as $n\to+\infty$,  we will evaluate the limit $\lim_{n\to+\infty} \gamma_n n^{1-\gamma_n}\tau(\Pi_n^{\gamma_n})$. We consider \eqref{eq:psiint} with $\gamma=\gamma_n$.

    For $s>0$, we substitute $y=s/(\gamma_n n)$ in \eqref{eq:psiint}. This leads us to analyze the pointwise limit of the integrand function multiplied by $\gamma_n n^{1-\gamma_n}$:
    \begin{align*}
    I_n(s,x) &= \gamma_n n^{1-\gamma_n} \frac{\sin(\pi\gamma_n)}{\pi\gamma_n} m\left(\frac{s}{\gamma_n n}\right)^{(n-1)\gamma_n} \left(\frac{s}{\gamma_n n}\right)^{-\gamma_n} \frac{x}{\left(1+\frac{s}{\gamma_n n} x\right)^2}\frac{1}{\gamma_n n} \\
    &=  \frac{\sin(\pi\gamma_n)}{\pi\gamma_n^{1-\gamma_n}}  m\left(\frac{s}{\gamma_n n}\right)^{(n-1)\gamma_n}s^{-\gamma_n} \frac{x}{\left(1+\frac{s}{\gamma_n n} x\right)^2}.
    \end{align*}
     If $m_2(\mu)<+\infty$, then by Lemma \ref{lem:CS} (ii), we obtain
    \[
    m(y) = 1 - \mathrm{Var}(\mu) y (1+o(1)),\qquad y\to 0+.
    \]
  If $m_2(\mu)=+\infty$, then again by Lemma \ref{lem:CS} (ii)
    \[
    \lim_{y\to 0+}\frac{m(y)-1}{y}=-\infty.
    \]
  If $m_2(\mu)<+\infty$,  we obtain, for any $s>0$ and $x\geq 0$,
\begin{align*}
\lim_{n\to+\infty} I_n(s,x)&= d_{\gamma_\ast}x\, s^{-\gamma_\ast} \lim_{n\to+\infty}  \left(1-\mathrm{Var}(\mu)\frac{s}{\gamma_n n}(1+o(1))\right)^{(n-1)\gamma_n} \\
&=d_{\gamma_\ast}x\, s^{-\gamma_\ast} e^{-\mathrm{Var}(\mu) s}=:I_\infty(s,x),
\end{align*}
where $d_\gamma = \sin(\pi\gamma)/(\pi\gamma^{1-\gamma})$ if $\gamma\in(0,1)$, and $d_0=1$. If $m_2(\mu)=+\infty$, then the above pointwise limit is $0$.

Fix $c\in(0, \mathrm{Var}(\mu))$. Then there exists $\varepsilon>0$ such that $m(y)\leq 1- c\,y$ for all $y\in(0,\eps]$.
Using the inequality $1-x\leq e^{-x}$ for $x\geq 0$, we obtain for $s\leq\eps \gamma_n n$,
\[
m\left(\frac{s}{\gamma_n n}\right)^{(n-1)\gamma_n} \leq \left(1-c\frac{s}{\gamma_n n}\right)^{(n-1)\gamma_n} \leq  e^{-c_1 s}
\]
for some $c_1>0$. Since $\lim_{n\to+\infty}\gamma_n=\gamma_\ast<1$, we have $\gamma_n\leq 1-\tilde\delta$ for some $\tilde\delta\in(0,1)$ for large $n$. Therefore, for large $n$, we obtain
\[
I_n(s,x) I_{(0,\eps \gamma_n n]}(s) \leq c_2 e^{-c_1 s} ( s^{-(1-\tilde\delta)}I_{(0,1]}(s)+I_{(1,+\infty)}(s) ) x \in L^1( \dd s\times\mu(\dd x))
\]
for some positive constants $c_1$ and $c_2$. The Lebesgue dominated convergence theorem implies that
\[
\lim_{n\to+\infty} \int_{[0,+\infty)}\int_0^{\eps \gamma_n n} I_n(s,x)\dd y\,\mu(\dd x) = \int_{[0,+\infty)}\int_0^{+\infty} I_\infty(s,x)\dd y\,\mu(\dd x).
\]
Since $y\mapsto m(y)$ is strictly decreasing and $m(\eps)<m(0)=1$ for any $\eps>0$, we have for $s>\eps \gamma_n n$,
\[
m\left(\frac{s}{\gamma_n n}\right)^{(n-1)\gamma} \leq m(\eps)^{(n-1)\gamma}\leq e^{-c_3 \gamma_n n}
\]
for some positive constant $c_3$.
Therefore, for $s>\eps \gamma_n n$, using the inequality $x(1+\frac{s}{\gamma_n n} x)^{-2}\leq \gamma_n n/s$, we arrive at
\[
I_n(s,x)I_{[\eps \gamma_n n,+\infty)}(s) \leq c_4 e^{-c_3 n}  s^{-1-\gamma} \gamma_n n.
\]
for positive $c_3$ and $c_4$.
Thus,
\[
\lim_{n\to+\infty} \int_{[0,+\infty)}\int_{\eps \gamma_n n}^{+\infty} I_n(s,x)\dd y\,\mu(\dd x)=0.
\]
Finally, we obtain in the case $m_2(\mu)<+\infty$,
\begin{align*}
\lim_{n\to+\infty} \gamma_n n^{1-\gamma_n} m_{\gamma_n}(\mu^{\boxtimes n}) &= d_{\gamma^\ast} \int_0^{+\infty} s^{-\gamma_\ast} e^{-\mathrm{Var}(\mu) s}\dd s\\
&=d_{\gamma^\ast} \Gamma(1-\gamma_\ast)\mathrm{Var}(\mu)^{\gamma_\ast-1}.
\end{align*}
The same limit equals $0$ if $m_2(\mu)=+\infty$.
If $\gamma_\ast=0$, then the right-hand side above equals $\mathrm{Var}(\mu)^{-1}$.
If $\gamma_n=1/\log(n)$, we have
 $n^{1-\gamma_n} \sim e^{-1} n$, thus recovering (ii).
 If $\gamma_n=\gamma\in(0,1)$ for all $n\in\mathbb{N}$, then
 \[
 \lim_{n\to+\infty} n^{1-\gamma }m_\gamma(\mu^{\boxtimes n})  =  \frac{1}{\gamma} \frac{\sin(\pi\gamma)}{\pi\gamma^{1-\gamma}} \Gamma(1-\gamma) \mathrm{Var}(\mu)^{\gamma-1}.
 \]
 By Euler's reflection formula, we obtain
\[
\frac{\sin(\pi\gamma)}{\pi\gamma}\Gamma(1-\gamma)  =  \frac{1}{\Gamma(1+\gamma)},
\]
which proves (i).

For (iii), we substitute $y=s/n^{1/(\alpha-1)}$ in \eqref{eq:psiint} and analyze the integrand function multiplied by $n^{\frac{1-\gamma}{\alpha-1}}$, i.e.,
\[
I_n(s,x) = \frac{\sin(\pi\gamma)}{\pi\gamma} m\left( \frac{s}{n^{1/(\alpha-1)}}\right)^{(n-1)\gamma} s^{-\gamma} \frac{x}{\left(1+\frac{s}{n^{1/(\alpha-1)}} x\right)^2}.
\]
Since, by Lemma \ref{lem:CS}, we have
\[
m\left( \frac{s}{n^{1/(\alpha-1)}}\right) = 1 + \frac{\pi\alpha}{\sin(\pi\alpha)} c \frac{s^{\alpha-1}}{n}(1+o(1)),
\]
we obtain for any $s>0$ and $x\geq0$,
\[
\lim_{n\to+\infty} I_n(s,x) =  \frac{\sin(\pi\gamma)}{\pi\gamma}x\,s^{-\gamma} \exp\left(  \frac{\pi\alpha}{\sin(\pi\alpha)} c s^{\alpha-1} \right)=I_\infty(s,x).
\]
Bounding as before, we have
\begin{align*}
I_n(s,x)I_{(0,\eps n^{1/(\alpha-1)}]}(s) &\leq c_2 s^{-\gamma}  e^{-c_1 s^{\alpha-1}}x
\intertext{and}
I_n(s,x)I_{(\eps n^{1/(\alpha-1)},+\infty)}(s) &\leq c_4 e^{-c_3 n} s^{-1-\gamma} n^{1/(1-\alpha)}
\end{align*}
for positive constants $c_i$, $i=1,\ldots,4$. Thus, by Lebesgue's dominated convergence theorem, we deduce that
\begin{align*}
\lim_{n\to+\infty} n^{\frac{1-\gamma}{\alpha-1} }m_\gamma(\mu^{\boxtimes n}) &=
\int_{[0,+\infty)}\int_0^{+\infty}  I_\infty(s,x) \dd s\,\mu(\dd x) \\
&=  \frac{\sin(\pi\gamma)}{\pi\gamma} \Gamma\left(\frac{1-\gamma}{\alpha-1}\right) \frac{1}{\alpha-1}\left(\frac{-\sin(\pi\alpha)}{\pi\alpha \,c}\right)^{\frac{1-\gamma}{\alpha-1}}.
\end{align*}
\end{proof}

\subsection{\texorpdfstring{Series convergence for $\tau(\A)=1$}{Series convergence for tau(A)=1}}
Recall that we denote
        \begin{align*}
                \Pi_n^{\uparrow} = \A_1^{1/2}\cdots\A_{n-1}^{1/2}\A_n \A_{n-1}^{1/2}\cdots\A_1^{1/2}
        \end{align*}
and that $\Pi_n$ is a generic element with distribution $\mu^{\boxtimes n}$.
Note that if $m_1(\mu)=1$, then $(\Pi_n^\uparrow)_{n\in\mathbb{N}}$ is a non-commutative martingale. We exploit this fact to establish bilaterall almost uniform convergence of $(\Pi_n^\uparrow)_{n\in\mathbb{N}}$.
The following theorem is the main result of this section.
\begin{thm}\label{thm:1}
Assume that $\A\geq0$ is non-Dirac and $\tau(\A)=1$. Then, the series
\[
\sum_{n=0}^{+\infty} \Pi_n^\uparrow\mbox{ converges  b.a.u.}
\]
If additionally $\mu_\A\big( (t,+\infty)\big)\sim c\,t^{-\alpha}$ for $c>0$ and $\alpha\in(1,2)$, then 
    \[
\sum_{n=0}^{+\infty} \Pi_n\mbox{ converges  a.u.}
\]
\end{thm}
Surprisingly, the proof of the latter result above is significantly simpler than the proof of b.a.u. convergence.

We start with a series of lemmas. The first one is the  analog of the maximal inequality for martingales.
\begin{lemma}[Proposition 5 in \cite{Cu71}]\label{lem:mart}
        Assume that $(X_n)_{n\geq 0}$ is a nonnegative martingale, and let $M$ be a positive number. There exists a projection $q$ such that
        \[
        \tau(1-q)\leq M^{-1}\tau(X_0)\quad\mbox{and}\quad        \|q\, X_n\, q\|\leq  M \mbox{ for all n}.
        \]
\end{lemma}

The proof of the following lemma is based on the proof of \cite[Proposition 6]{Cu71}.
\begin{lemma}\label{lem:2}
        Let $\mu\in\mathcal{M}_+$. Assume that $m_1(\mu)=1$, $m_2(\mu)<+\infty$ and that $\mu$ is non-Dirac.
        Let $(n_k)_k$ be an arbitrary monotone increasing positive integer sequence, and let $(h_k)_k$ be a sequence of nonnegative numbers such that
        \[
        \sum_{k=1}^{+\infty} e^{-h_k} <+\infty.
        \]
        For any $\eps>0$, there exists a projection $p$ and $N\in\mathbb{N}$  such that $\tau(1-p)\leq\eps$ and
        \begin{align}\label{eq:pPp}
                \| p \,\Pi_n^\uparrow\, p\|\leq \left(\frac{\log(n_k)}{n_k}\right)^{\log(n_k)} n_k^{h_k},\qquad  n_k\leq n<n_{k+1},\quad k\geq N.
        \end{align}
\end{lemma}
\begin{proof}
        Denote $f_k = \left(\frac{\log(n_k)}{n_k}\right)^{\log(n_k)} n_k^{h_k}$ and define $p_k=e_{[0,f_k]}(\Pi_{n_k}^\uparrow)$.
        For $k\geq 1$, let $X^{(k)}_n:= p_{k} \,\Pi_{n}^\uparrow\, p_{k}$ for $n\geq n_k$. 
        Note that for each $k$, $(X^{(k)}_n)_{n\geq n_k}$ is a nonnegative martingale. Therefore, by Lemma \ref{lem:mart},
        for each $k$ and any $M_k>0$ there exists a projection $q_k$ such that
        \[
        \tau(1-q_k)\leq M_k^{-1} \tau(X_{n_k}^{(k)}) \quad\mbox{and}\quad\|q_k X_n^{(k)} q_k\|\leq M_k\mbox{ for }n\geq n_k.
        \]
        Let $M_k = f_k$ and define
        $p=\bigwedge_{k\geq N} (q_k\wedge p_{k})$, where $N$ is yet to be determined.
       
        For all $n\geq n_k$ and $k\geq N$, we have
        \[
        \| p\,\Pi_{n}^\uparrow\,p \| \leq \| q_k p_{k} \Pi_{n}^\uparrow p_{k} q_k \|=\|q_k X_n^{(k)} q_k\| \leq f_k.
        \]     
        Moreover,
        \begin{align*}
                \tau(1-p) &\leq \sum_{k=N}^{+\infty} \tau(1-q_k)+\sum_{k=N}^{+\infty} \tau(1-p_{k}) \\
                & \leq \sum_{k=N}^{+\infty}  f_k^{-1} \tau(\Pi_{n_k}^{\uparrow} p_{k})+\sum_{k=N}^{+\infty} \mu_{\Pi_{n_k}}\big((f_k,+\infty)\big)
        \end{align*}
        By Markov's inequality, we have
        \[
        \mu_{\Pi_n}\big((f,+\infty)\big) \leq \frac{\tau(\Pi_n^{1/\log(n)})}{ f^{1/\log(n)}}.
        \]
        Similarly, since $t\, I_{[0,f]}(t) \leq t\left(f/t\right)^{1-1/\log(n)}=t^{1/\log(n)}f^{1-1/\log(n)}$ for $s,t>0$, it follows that
        \[
        f^{-1}\tau\left(\Pi_{n} e_{[0,f]}(\Pi_n)\right) = f^{-1} \int_{[0,+\infty)} t\, I_{[0,f]}(t)\,\mu_{\Pi_n}(\dd t) \leq \frac{\tau(\Pi_n^{1/\log(n)})}{ f^{1/\log(n)}}.
        \]
        By \eqref{eq:limit}, we have
        $\tau(\Pi_n^{1/\log(n)})\leq c  \frac{\log(n)}{n}$ for $n>1$, where $c$ is a positive constant. 
        Therefore, we can conclude that
        \[
        \tau(1-p) \leq  2 \sum_{k=N}^{+\infty} c \frac{\log(n_k)}{n_k} f_k^{-1/\log(n_k)} = 2 c \sum_{k=N}^{+\infty} e^{-h_k}.
        \]
        Clearly, one can adjust $N$ so that $\tau(1-p)\leq \eps$.
\end{proof}

Now, we are ready to prove Theorem \ref{thm:1}.
\begin{proof}[Proof of Theorem \ref{thm:1}]
    Let $\X_n=\sum_{k=1}^n \Pi_k^\uparrow$.
    Let $n_k = k^2$ and $h_k=\frac32\log(k)$. Then, we have $\sum_{k=1}^{+\infty} e^{-h_k}=\sum_{k=1}^{+\infty} k^{-3/2}<+\infty$.  By applying Lemma \ref{lem:2}, there exists a projection $p$ and $N\in\mathbb{N}$ such that the inequality \eqref{eq:pPp} holds.
   
Let $K\in\mathbb{N}$ satisfy $K\geq N$. For $n>m\geq n_K$, we obtain
    \begin{align*}
    \|p(\X_n-\X_m)p\|&\leq \sum_{k=m+1}^n \| p \,\Pi_k^\uparrow\, p \| \leq \sum_{k=K}^{+\infty} \sum_{l=n_k}^{n_{k+1}-1}   \| p \,\Pi_l^\uparrow\, p \| \\
    & \leq \sum_{k=K}^{+\infty} (n_{k+1}-n_k) \left(\frac{\log(n_k)}{n_k}\right)^{\log(n_k)} n_k^{h_k} \\
    &= \sum_{k=K}^{+\infty} (2k+1) \left(\frac{2\log(k)}{k^2}\right)^{2\log(k)} k^{3\log(k)}  \\
    &=\sum_{k=K}^{+\infty} (2k+1) k^{2 \log\log(k)} \left(\frac{4}{k}\right)^{\log(k)}.
    \end{align*}
    It is straightforward to see that this series converges. Therefore, for any $\varepsilon'>0$, we can choose $K$ such that $\|p(\X_n-\X_m)p\|\leq \varepsilon'$. This implies that $(\X_n)_n$ is Cauchy bilaterally almost uniformly, which proves the first part of the assertion.

For the second part, we will not rely on martingales.
  Let $\X_n=\sum_{k=1}^n \Pi_k$.
   For $\beta>0$ define $p_k=e_{[0, k^{-\beta}]}(\Pi_k)$ and $p=\bigwedge_{k\geq N} p_k$.
Fix $\eps,\eps'>0$. Given $N\in\mathbb{N}$, we find that for $n>m>N$,
    \[
  \| (\X_n-\X_m)p\| \leq \sum_{k=N}^\infty \| \Pi_k p_k\| \leq \sum_{k=N}^\infty \frac{1}{k^\beta} < \eps
    \]
  for sufficiently large $N$, provided $\beta>1$. Moreover, for $\gamma\in(0,1)$, we have
    \[
    \tau(1-p)\leq \sum_{k=N}^{+\infty} \tau(1-p_k) = \sum_{k=N}^{+\infty} \mu_{\Pi_k}\big((k^{-\beta},+\infty)\big)\leq \sum_{k=N}^\infty k^{\beta\gamma}\tau(\Pi_k^\gamma).
    \]
    By Theorem \ref{lem:1} (iii), we know that $\tau(\Pi_k^\gamma)\leq C k^{-\frac{1-\gamma}{\alpha-1}}$ for some positive constant $C$. Setting $c:=\frac{2-\alpha}{1+\beta(\alpha-1)}\in(0,1)$, if $\gamma\in(0,c)$, then
    $\frac{1-\gamma}{\alpha-1}-\beta\gamma>1$. Thus, by increasing $N$ if necessary, we obtain $\tau(1-p)\leq \eps'$.
\end{proof}

\section{Fixed point equations}\label{sec:4}

In this section, we focus on solutions to the affine fixed-point equations of the form
\begin{align}\label{eq:aff0}
\X\stackrel{d}{=}\A\X\A^\ast+\B,\qquad (\A,\B)\mbox{ and }\X\mbox{ are *-free,}
\end{align}
where $\B=\B^\ast$.
For a given pair $(\A,\B)\in\tcA^2$, our goal is to establish the existence and uniqueness of the solution and to investigate the properties of this solution.
It is important to note that it is the distribution of $\X\in\tcA$ that serves as a solution to the above problem, rather than $\X$ itself.
However, for simplicity, we may occasionally refer to $\X$ as the solution.

We begin by examining two illustrative examples.

\subsection{Examples}\label{sec:ex1}
\subsubsection{Free Beta prime distribution}
Yoshida \cite{Yoshida} defined the free-beta prime distribution $f\mathcal{B}'_{a,b}$ as the free multiplicative convolution $\mu_a\boxtimes \mu_b^{-1}$, $a>0, b>1$, where $\mu_b^{-1}$ is the pushforward measure of $\mu_b$ by the mapping $x\mapsto x^{-1}$, and $\mu_\lambda$ is the Marchenko-Pastur distribution defined by
 \[
 \mu_\lambda(\dd x)=\frac{\sqrt{(\lambda_+-x)(x-\lambda_-)}}{2\pi x}\dd x+\max\{1-\lambda,0\}\delta_0(\dd x),
 \]
 where $\lambda_\pm=(1\pm\sqrt{\lambda})^2$, $\lambda>0$. Yoshida obtained the following expression for the free-beta prime distribution:
\[
f\mathcal{B}'_{a,b}(\dd x) = f_{a,b}(x)\dd x+\max\{1-a,0\}\delta_0(\dd x),
\]
where
\begin{align}\label{eq:fab}
f_{a,b}(x)=\frac{(b-1)\sqrt{(\gamma_+-x)(x-\gamma_-)}}{2\pi x(1+x)}I_{[\gamma_-,\gamma_+]}(x)
\end{align}
and
\[
\gamma_\pm = \left(\frac{\sqrt{ab}\pm\sqrt{a+b-1}}{b-1}\right)^2.
\]
We observe that $f_{a,b}$ has a point-wise limit as $b\to1+$, which allows us to extend the definition of $f\mathcal{B}'_{a,b}$ to the case $b=1$. We define
\[
f\mathcal{B}'_{a,1}(\dd x) = \frac{\sqrt{4a x-(a-1)^2}}{2\pi x(1+x)}I_{[(a-1)^2/(4 a),+\infty)}(x)\dd x+\max\{1-a,0\}\delta_0(\dd x).
\]
The case of $b=1$ was not covered in \cite{Yoshida}. Note that, in contrast to the case $b>1$, for $b=1$, the support of free beta-prime is unbounded.

The $S$-transform of $f\mathcal{B}'_{a,b}$ is $S_{a,b}(z)=S_{\mu_a}(z) S_{\mu_b^{-1}}(z)=\frac{b-1-z}{a+z}$.

\begin{Corollary}\label{thm:exConv}
Let $a>0$, $b\geq1$ and assume that $\A\sim f\mathcal{B}'_{a,a+b}$. Then $\X\sim f\mathcal{B}'_{a,b}$ is the unique solution to
\begin{align}\label{eq:new}
        \X \stackrel{d}{=} \A^{1/2}\X\A^{1/2}+\A,\qquad \A\mbox{ and }\X\mbox{ are freely independent}.
\end{align}
\end{Corollary}
\begin{proof}
Eq. \eqref{eq:new} holds if and only if $S_\X = S_{\A} S_{\X+1}$. We will show that $\X$ solves \eqref{eq:new} for arbitrary $\A$ if and only if the following functional equation is satisfied:   
    \begin{align}\label{eq:funeqS}
    (1+S_\X(z))S_{\X}(z (1+S_\X(z))) = S_\A(z(1+S_\X(z))) S_\X(z).
    \end{align}
Starting with the identity
    \[
    \psi_{\X+1}\left( \frac{z}{1+z}\right) = z + (1+z) \psi_\X(z),
    \]
we can rearrange this to obtain
    \[
    \frac{\chi_\X(z)}{1+\chi_\X(z)} = \chi_{\X+1}\left(z+(1+z)\chi_\X(z)\right).
    \]
This leads to
    \[
    \frac{S_\X(z)}{1+S_\X(z)} =  S_{\X+1}(z(1+S_\X(z))).
    \]
The condition $S_\X = S_{\A} S_{\X+1}$ then implies the functional equation \eqref{eq:funeqS}.

Now, we can verify by a direct computation that if $S_\A(z) = \frac{a+b-1-z}{a+z}$, then $S_\X(z)=
\frac{b-1-z}{a+z}$ satisfies \eqref{eq:funeqS}. It will follow from other considerations in the paper,
that there are no other solutions to \eqref{eq:new}.
\end{proof}

\subsubsection{Inverse Marchenko-Pastur distribution}\label{sec:MY}

For $\lambda\in\R$, we define the free Generalized Inverse Gaussian (free GIG)  distribution by
        \begin{align*}
        f{GIG}_{\lambda}(\dd x) = \frac{\sqrt{(x-a)(b-x)}}{2\pi}\left(\frac{1}{x}+\frac{1}{\sqrt{a b}\,x^2}\right) I_{[a,b]}(x) \dd x,
        \end{align*}
where $(a,b)$, with $0<a<b$, is the unique solution to the system of equations:
\begin{align*}
                1-\lambda +\sqrt{ab}-\frac{a+b}{2 a b}=0 =
                1+\lambda +\frac{1}{\sqrt{ab}}-\frac{a+b}{2}.
\end{align*}
Typically, the free GIG distribution is defined with three parameters $(\lambda,\alpha,\beta)$, see \cite{Feral0, Hasebe}. For our considerations below, we need to restrict ourselves to the case $\alpha=\beta$. We simplify notation by setting $\alpha=\beta=1$.

In \cite[Theorem 3.3]{KSzMY}, it was shown that if $\X\sim fGIG_{-\lambda}$ and $\Y$ has the Marchenko-Pastur distribution $\mu_{\lambda}$, with $\X$ and $\Y$ being freely independent and $\lambda>1$, then
$\mathbb{U}=(\X+\Y)^{-1}$ and $\mathbb{V}=\X^{-1}-(\X+\Y)^{-1}$ are also freely independent. Moreover, it follows from consideration in \cite[page 383]{KSzMY} that $\mathbb{U}\stackrel{d}{=}\X$ and $\mathbb{V}\stackrel{d}{=}\Y$.

Furthermore, in \cite[Theorem 4.1]{Hasebe}, the authors characterized $fGIG_{-\lambda}$ through the continued fraction fixed-point equation, i.e.,
\[
\X\stackrel{d}{=}(\X+\Y)^{-1},
\]
where $\X$ and $\Y\sim\mu_\lambda$ are free and $\lambda>0$.

Using the Hua identity, we have $\mathbb{V}^{-1} = \X \Y^{-1}\X+\X$. Thus, the condition $\mathbb{V}\stackrel{d}{=}\Y$ implies that
\begin{align}\label{eq:MY}
\Y^{-1}\stackrel{d}{=} \X \Y^{-1}\X+\X,
\end{align}
which is an affine fixed-point equation.
The equality in distribution above was established only under the condition $\lambda>1$, ensuring that the support of $\mu_\lambda$ is bounded away from $0$. This guarantees that $\Y^{-1}$ is bounded. However, the inverse of $\Y$ is also well-defined in the case of $\lambda=1$, when its support is unbounded from above. Specifically, we have
\begin{align}\label{eq:IMP}
\mu_{\Y^{-1}}(\dd x) = \mu_1^{-1}(\dd x) = \frac{\sqrt{x-1/4}}{\pi x^2} I_{[1/4,+\infty)}(x)\dd x,
\end{align}
Since all operations involved in \eqref{eq:MY} are continuous in the topology of weak convergence of probability measures, we can take the limit as $\lambda\to1+$. This allows us to extend the validity of  \eqref{eq:MY} to the case $\lambda=1$.

\subsection{Existence and uniqueness of the solution}

Note that if the distribution of $\X$ satisfies \eqref{eq:aff0}, then the distribution of $\X^\ast$ also satisfies this equation. Thus, if there exists a unique solution, it must necessarily be self-adjoint. 

To simplify the form of the fixed point equation \eqref{eq:aff0}, we consider the polar decomposition of $\A$, expressed as $\A=U|\A|$, where $U$ is unitary. Substituting this into our equation yields,
\[
\X\stackrel{d}{=}\A\X\A^\ast+\B = U(|\A|\X|\A|+U^\ast\B U)U^\ast\stackrel{d}{=}|\A|\X|\A|+U^\ast\B U,
\]
where $(|\A|, U^\ast \B U)$ are *-free from $\X$.

Thus, without loss of generality, we may always assume that $\A=\A^\ast\geq0$. From now on, we will consider the following affine fixed-point equation:
\begin{align}\label{eq:affine}
\X\stackrel{d}{=}\A^{1/2}\X\A^{1/2}+\B,\qquad (\A,\B)\mbox{ and }\X\mbox{ are *-free,}
\end{align}
where $\A,\B\in\tcA$, $\A=\A^\ast\geq 0$ and $\B=\B^\ast\neq 0$.

\begin{thm}\label{thm:uniqe_free}
    If $\tau(\A)\leq 1$ and the series
    \begin{align}\label{eq:perp}
  \mathbb{S}=\sum_{n=1}^{+\infty} \A_1^{1/2}\cdots \A_{n-1}^{1/2}\B_n\A_{n-1}^{1/2}\cdots\A_1^{1/2}
    \end{align}
    is convergent in distribution, then $\mathbb{X}\stackrel{d}{=}\mathbb{S}$
    is the unique solution to \eqref{eq:affine}.
\end{thm}
\begin{proof}
It is immediate to check that $\mathbb{S}$ solves \eqref{eq:affine}. Suppose that $\nu$ and $\nu'$ are two solution to \eqref{eq:affine}. Take $\X_0\sim\nu$, $\X_0'\sim \nu'$, each free of $(\A_n,\B_n)_n$, and define two processes: for $n\geq 1$,
  \begin{align*}
\X_n = \A_n^{1/2} \X_{n-1}\A_n^{1/2}+\B_n,\qquad
\X_n' = \A_n^{1/2} \X_{n-1}'\A_n^{1/2}+\B_n.
  \end{align*}
By \eqref{eq:affine}, for each $n$, we have $\X_n\sim \nu$ and $\X_n'\sim\nu'$. However,
  \[
  \X_n-\X_{n}' = \A_n^{1/2}\cdots\A_1^{1/2}(\X_0-\X_0')\A_1^{1/2}\cdots\A_n^{1/2}.
  \]
Since under $\tau(\A)\leq 1$, $\Pi_n=\A_n^{1/2}\cdots\A_2^{1/2}\A_1\A_2^{1/2}\cdots\A_n^{1/2}$ converges to zero in distribution (recall results from Section \ref{sec:LLN}),
we see that $\X_n-\X_{n}'$ also converges to zero in distribution. Therefore, we conclude that $\nu=\nu'$. 
\end{proof}

\begin{remark}
    The condition $\tau(\A)\leq 1$ is not necessary for the existence of a solution to \eqref{eq:affine}. E.g. if $\A c+\B=c$ for some $c\in\R$, then clearly $c$ is a solution regardless of any moment condition on $(\A,\B)$. In fact, in this case, we have 
    \[
\sum_{n=1}^k \A_1^{1/2}\cdots \A_{n-1}^{1/2}\B_n\A_{n-1}^{1/2}\cdots\A_1^{1/2}
  =c(1- \Pi_k^\uparrow).  \]
\end{remark}

\begin{defin}
    We say that the model \eqref{eq:affine} is irreducible if for any $c\in\R$
    \[
   \A c+\B\neq c.
    \]
\end{defin}

By the representation of the solution in \eqref{eq:perp}, we deduce the following easy result.
\begin{Corollary}\
\begin{enumerate}
    \item[(i)] If $\B\geq 0$, then $\X\geq 0$.
    \item[(ii)] If $(\A,\B)\stackrel{d}{=}(\A,-\B)$, then $\X\stackrel{d}{=}-\X$.
\end{enumerate}
\end{Corollary}

Let $\log^+(a)=\max\{\log(a),0\}$.
\begin{thm}
The series \eqref{eq:perp} is convergent in the following cases:
\begin{enumerate}
    \item[(i)] almost uniformly if  $\A=\alpha 1$ with $\alpha\in(0,1)$ and $\B\in\tcA$ with $\tau(\log^+(|\B|))<+\infty$,
    \item[(ii)] in norm $\|\cdot\|$ if $\tau(\A)<1$ and $\A,\B\in\cA$,
    \item[(iii)] in $L^1(\tcA,\tau)$ if $(\A,\B)\in\tcA$ with $\tau(\A)<1$, $\B\geq 0$ and $\tau(\B)<+\infty$,
    \item[(iv)] bilaterally almost uniformly if $\A\in\tcA$ with $\tau(\A)=1$ and $\tau(\A^2)<+\infty$ and $\B\in\cA$,
    \item[(v)] almost uniformly if $\A\in\tcA$ with $\tau(\A)=1$ and $\mu_\A\big((t,+\infty)\big)\sim c\, t^{-\alpha}$ for $\alpha\in(1,2)$ and $c>0$, and $\B\in\cA$.
\end{enumerate}
\end{thm}
\begin{proof}
    Case (i) follows from \cite{Ber05}.
    Case (ii) was proved in Remark \ref{rem:m<1}.
    Case (iii) is trivial.
    Cases (iv) and (v) follow from Theorem \ref{thm:1}. Indeed, if $\B\in\mathcal{A}$, then $\|\B\|<+\infty$ and we can bound $\A^{1/2}\X\A^{1/2}+\B\leq \A^{1/2}\X\A^{1/2}+\|\B\|\cdot 1_{\cA}$.
\end{proof}

\begin{remark}
    If $\B=c\,\A$ and $\X$ solves \eqref{eq:affine}, then $\Y=c^{-1}\X+1$ solves the modified equation
    \[
    \Y\stackrel{d}{=}\A^{1/2}\Y\A^{1/2}+1.
    \]
    Thus, it follows that a unique solution to \eqref{eq:affine} also exists in the case where $(\A,\B)\in\tcA^2$, $0\leq \B\leq c\,\A$, $\tau(\A)\leq 1$.
\end{remark}

The following result was obtained in the proof of Theorem \ref{thm:exConv}.
\begin{lemma}\label{lem:Sperp}
    If $\B=\A$ and the $S$-transform of $\mu\in\mathcal{M}_+$ satisfies the following functional equation
    \[
    (1+S_\mu(z))S_{\mu}\left(z (1+S_\mu(z))\right) = S_\A\left(z(1+S_\mu(z))\right) S_\mu(z),\quad z\in(-\eps,0)
    \]
    for some $\eps>0$,
    then the solution $\X$ to \eqref{eq:affine} has the distribution $\mu$.
\end{lemma}

\subsection{Moments}\label{sec:moments}

As we noted earlier, when $\tau(\A) < 1$ and $\tau(\B) < +\infty$, the solution to \eqref{eq:affine} is given by a series that converges in $L^1(\tilde{\cA},\tau)$. The following lemma shows that, in this regime, the existence of higher moments of $\X$ is determined by the corresponding moments of $\A$ and $\B$. In contrast, once $\tau(\A) = 1$ and $\B > 0$, the solution $\X$ fails to possess moments—even if all moments of $\A$ and $\B$ exist.

\begin{lemma} Assume $\A,\B\geq 0$ and $\B\neq 0$. The following holds 
\begin{enumerate}
    \item[(i)]  If $\tau(\A)<1$, $\tau(\A^p)<+\infty$ and $\tau(\B^p)<+\infty$, then $\tau(\X^p)<+\infty$.
    \item[(ii)] If $\tau(\A)=1$, then $\tau(\X)=+\infty$.
\end{enumerate}  
\end{lemma}
\begin{proof}
As we observed earlier, the solution to  \eqref{eq:affine} has a series representation
\begin{align*}
  \X\stackrel{d}{=}\sum_{n=1}^{+\infty} \A_1^{1/2}\cdots \A_{n-1}^{1/2}\B_n\A_{n-1}^{1/2}\cdots\A_1^{1/2}.
    \end{align*}
For any $p\geq 1$, by Minkowski's inequality, we have
\begin{align*}
\|\X\|_p\leq  \sum_{n=1}^{+\infty} \|\A_1^{1/2}\cdots \A_{n-1}^{1/2}\B_n\A_{n-1}^{1/2}\cdots\A_1^{1/2}\|_p.
\end{align*}
Let $p\in \mathbb{N}$, by \eqref{eqn:cumprod},
\begin{align}\label{eq:pp}\begin{split}
    \|\A_1^{1/2}\cdots \A_{n-1}^{1/2}\B_n\A_{n-1}^{1/2}\cdots\A_1^{1/2}\|_p^p & =\tau\left(\left(\A_1^{1/2}\cdots \A_{n-1}^{1/2}\B_n\A_{n-1}^{1/2}\cdots\A_1^{1/2}\right)^p\right)\\
    &=\tau\left(\left(\B_n \Pi_{n-1}\right)^p\right)=\sum_{\pi\in \NC(p)} \kappa_\pi(\B_n)\tau_{\Kr(\pi)}\left(\Pi_{n-1}\right),
    \end{split}
\end{align}
where we used the fact that $\A_{n-1}^{1/2}\ldots \A_2^{1/2}\A_1 \A_2^{1/2}\ldots \A_{n-1}^{1/2}$ has the same distribution as $\Pi_{n-1}$ and that it is freely independent from $\B_n$.
By Theorem \ref{thm:28}, for each $k>0$, there exists $C_k$ such that for all $n\in\mathbb{N}$,
\[
\tau\left( \Pi_{n-1}^k \right)\leq C_k\,  n^{k-1}\tau(\A)^{nk}.
\]
Thus, in the sum on the right-hand side of \eqref{eq:pp}, the dominating term 
corresponds to $\pi=1_p$. Hence, there exists a constant $D_p$ such that for all $n$,
\begin{align*}
    \|\A_1^{1/2}\cdots \A_{n-1}^{1/2}\B_n\A_{n-1}^{1/2}\cdots\A_1^{1/2}\|_p^p\leq D_p\, n^{p-1} \tau(\A)^{np},
\end{align*}
which implies
\begin{align*}
    \|\X\|_p\leq  D_p^{1/p} \sum_{n=1}^{\infty}  n^{1-\frac{1}{p}} \tau(\A)^n.
\end{align*}
The series above is convergent if and only if $\tau(\A)<1$.

The proof of part (ii) is immediate: suppose that $\tau(\A)=1$. Then, applying $\tau$ to both sides of \eqref{eq:affine}, we obtain
$\tau(\X)=\tau(\X)+\tau(\B)$.
Since $\tau(\B)>0$, we conclude that $\tau(\X)$ cannot be finite.
\end{proof}

\begin{remark}
    By the main results of this paper (see Section \ref{sec:tails}), it follows that $\tau(\X^\gamma)<+\infty$ for $\gamma\in(0,1/2)$ provided $\tau(\B)<+\infty$.
\end{remark}

\section{\texorpdfstring{Subordination for $\A\X\A+\B$}{Subordination for AXA+B}}\label{sec:subordination}


In this section, we present the subordination results relevant to our analysis. Our general setting is that of a $W^*$-noncommutative probability space $(\mathcal A,\tau)$. When all operators $\mathbb{X}$, $\mathbb{A}$, and $\mathbb{B}$ belong to $\mathcal A$, the subordination formulas hold under the weaker assumption that $\tau$ is a faithful state continuous in the wo (weak operator) topology, rather than a trace. However, in the case of unbounded operators, the non-tracial setting introduces significant difficulties because the operators affiliated with an arbitrary von Neumann algebra may not form an algebra. Therefore, we assume that if the variables $\mathbb{X}$, $\mathbb{A}$, and $\mathbb{B}$ are unbounded, then $\tau$ is a trace.

The results here extend the result of \cite{LS19} in several direction. In \cite{LS19}, the distribution of $f(\A)\X f(\A) + \A$ was studied under the assumption that all operators were bounded. In this section, we do not require that
$\mathbb{X}\stackrel{d}{=}\mathbb{AXA}+\mathbb{B}$ so $\X$ in general is not a perpetuity.
For notational simplicity, we write $\mathbb{AXA}+\mathbb{B}$ rather than $\A^{1/2}\X\A^{1/2}+\mathbb{B}$ as in other sections.


Let $\mathbb{X,A,B}$ be self-adjoint operators, affiliated with $\mathcal A$, such that $\mathbb X$ is
*-free from $(\mathbb{A,B})$. Since the cases $\mathbb X\in\mathbb R\cdot
1_\mathcal A,\mathbb A\in\mathbb R\cdot1_\mathcal A$ are trivial, we exclude them from the beginning (from now on, we suppress the subscript and denote the unit of our von
Neumann algebra by $1\in\mathcal A$). We prove our subordination results in several steps, in contexts with varying degrees of generality.

\begin{lemma}\label{inv}
Assume that $\mathbb{A,B,X}$ are bounded in $(\mathcal A,\tau)$, where $\tau$ is a normal faithful tracial state, and that $\mathbb X$ is *-free from $(\mathbb{A,B})$
with respect to $\tau$ (here we do not need to assume that these variables are selfadjoint). Assume also that $\mathbb A$ is invertible in $\mathcal A$. Then there exists an
analytic function $f$ defined on a neighborhood of infinity such that
$$\tau\left(\left(z-\mathbb B-\mathbb A\mathbb X\mathbb A\right)^{-1}\right)=\tau\left(\left(\mathbb Af(z)\mathbb A+z-\mathbb B\right)^{-1}\right).$$
If in addition $\mathbb{A,B,X}$ are selfadjoint, then $f$ extends to an analytic self-map of $\mathbb C^+$. If $E_{(\mathbb{A,B})}$ is the $\tau$-preserving conditional
expectation onto the von Neumann algebra generated by $(\mathbb{A,B})$, then
$$
E_{(\mathbb{A,B})}\left[\left(z-\mathbb B-\mathbb A\mathbb X\mathbb A\right)^{-1}\right]=
\left(\mathbb Af(z)\mathbb A+z-\mathbb B\right)^{-1}.
$$
\end{lemma}

\begin{proof}
As $\mathbb X$ is *-free from $(\mathbb{A,B})$, $\begin{bmatrix} 1 & \mathbb A^{-2} \\ 0 & \mathbb A^{-1}(z-\mathbb B)\mathbb A^{-1} \end{bmatrix}$ and
$\begin{bmatrix} 0&0\\0& \mathbb X \end{bmatrix}$ are *-free over $M_2(\mathbb C)$ with respect to $\tau\otimes{\rm id}_{M_2(\mathbb C)}$. One has
\begin{eqnarray*}
\lefteqn{\left(\begin{bmatrix} 1 &\mathbb A^{-2}\\0&\mathbb A^{-1}(z-\mathbb B)\mathbb A^{-1} \end{bmatrix}+\begin{bmatrix} 0&0\\0& -\mathbb X \end{bmatrix}\right)^{-1}}\\
& = &\begin{bmatrix} 1 & -\mathbb A^{-2}\left(\mathbb A^{-1}(z-\mathbb B)\mathbb A^{-1}-\mathbb X\right)^{-1} \\
0 & \left(\mathbb A^{-1}(z-\mathbb B)\mathbb A^{-1}-\mathbb X\right)^{-1} \end{bmatrix}.
\end{eqnarray*}
It is easy to observe that, since all operators involved are bounded, all inverses above make sense if $|z|$ is sufficiently large. Voiculescu's result \cite{Vo00} (or \cite{BBs})
guarantees that
\begin{eqnarray*}
\lefteqn{\begin{bmatrix} 0 & \!E_{(\mathbb{A,B})}\!\left[\mathbb A^{-2}\left(\mathbb A^{-1}(z-\mathbb B)\mathbb A^{-1}\!-\mathbb X\right)^{\!-1}\right]\!
E_{(\mathbb{A,B})}\!\left[\!\left(\mathbb A^{-1}(z-\mathbb B)\mathbb A^{-1}\!-\mathbb X\right)^{\!-1}\right]^{\!-1}\!-\mathbb A^{-2} \\
0 & E_{(\mathbb{A,B})}\left[\left(\mathbb A^{-1}(z-\mathbb B)\mathbb A^{-1}-\mathbb X\right)^{-1}\right]^{-1}-\mathbb A^{-1}(z-\mathbb B)\mathbb A^{-1} \end{bmatrix}}\\
& = &\begin{bmatrix} 1 & -E_{(\mathbb{A,B})}\left[\mathbb A^{-2}\left(\mathbb A^{-1}(z-\mathbb B)\mathbb A^{-1}-\mathbb X\right)^{-1}\right] \\
0 & E_{(\mathbb{A,B})}\left[\left(\mathbb A^{-1}(z-\mathbb B)\mathbb A^{-1}-\mathbb X\right)^{-1}\right] \end{bmatrix}^{-1}-\begin{bmatrix} 1 &\mathbb A^{-2}\\0&\mathbb A^{-1}(z-\mathbb B)\mathbb A^{-1} \end{bmatrix}\\
& = & \begin{bmatrix} 0& \varpi(z) \\0& f(z) \end{bmatrix}\in M_2(\mathbb C),
\end{eqnarray*}
where $f,\varpi$ are $\mathbb C$-valued analytic functions on a neighborhood of infinity in $\mathbb C$.
In particular, $E_{(\mathbb{A,B})}\left[\left(\mathbb A^{-1}(z-\mathbb B)\mathbb A^{-1}-\mathbb X\right)^{-1}\right]=
\left(f(z)+\mathbb A^{-1}(z-\mathbb B)\mathbb A^{-1}\right)^{-1}$. Multiply left and right with $\mathbb A^{-1}$ to obtain
\begin{equation}
E_{(\mathbb{A,B})}\left[\left(z-\mathbb B-\mathbb A\mathbb X\mathbb A\right)^{-1}\right]=
\left(\mathbb Af(z)\mathbb A+z-\mathbb B\right)^{-1},\quad|z|\text{ large}.
\end{equation}
If, in addition, $\mathbb{A,X,B}$ are selfadjoint, then $\Im\left\{E_{(\mathbb{A,B})}\left[\left(z-\mathbb B-\mathbb A\mathbb X\mathbb A\right)^{-1}\right]^{-1}\right\}
\ge\Im z\cdot1,$ so that $\Im f(z)\ge0$ whenever $z\in\mathbb C^+$. When $f$ is not constant, one obtains immediately that $\Im f(z)>0$ for all $z\in\mathbb C^+$.
Applying $\tau$ yields the claimed relation.

While it is not relevant for our purposes, we nevertheless observe that direct arithmetic manipulations show that, under the hypotheses of our lemma, $\varpi(z)=0$ for all $z$.
\end{proof}

Let us further extend the above lemma to eliminate the requirement of invertibility of $\mathbb A$. There exists a sequence $\{s_n\}_{n\in\mathbb N}$ such that
$s_n\searrow0$ as $n\to +\infty$ and $s_n,\frac{-s_n}{2}$ are not eigenvalues of $\mathbb A$. Define $\mathbb A_n=k_n(\mathbb A)$, where
$$
k_n(t)=\left\{\begin{array}{cc}
t & \text{ if } t\le -s_n\text{ or }t\ge s_n\\
\frac{t-3s_n}{4} & \text{ if }-s_n<t<s_n
\end{array}\right.,
$$
and $k_n(\mathbb A)$ is understood in the sense of Borel functional calculus. Then $\mathbb A_n\to\mathbb A$ in the so (strong operator) topology, but in fact also in norm, as it
follows from the definition of $k_n$ and from the lack of eigenvalues among $s_n$ that $\|\mathbb A_n-\mathbb A\|\le 2s_n\to0$ as $n\to +\infty$. For each $n$, the above lemma
guarantees the existence of an analytic $f_n\colon\mathbb C^+\to\mathbb C^+$ such that
$$
E_{(\mathbb{A}_n,\mathbb B)}\left[\left(z-\mathbb B-\mathbb A_n\mathbb X\mathbb A_n\right)^{-1}\right]^{-1}-z+\mathbb B=
f_n(z)\mathbb A_n^2.
$$
Unless $\mathbb A=0$ (case which is trivial), it follows that there exists an $\epsilon>0$ such that $\tau(\mathbb A_n^2)>\epsilon>0$ for all sufficiently large $n$.
Since $\left\|E_{(\mathbb{A}_n,\mathbb B)}\left[\left(z-\mathbb B-\mathbb A_n\mathbb X\mathbb A_n\right)^{-1}\right]^{-1}\right\|$ stays bounded as $n\to +\infty$,
applying $\tau$ in the above-displayed relation shows that $\{f_n\colon n\in\mathbb N\}$ is a normal family, so there exists some limit point $f\colon\mathbb C^+
\to\mathbb C^+\cup\mathbb R$. Pick such a limit point. Applying $\tau$ to the equality $E_{(\mathbb{A}_n,\mathbb B)}\left[\left(z-\mathbb B-\mathbb A_n\mathbb X\mathbb A_n
\right)^{-1}\right]=\left(\mathbb A_nf_n(z)\mathbb A_n+z-\mathbb B\right)^{-1}$, we obtain $\tau\left(\left(z-\mathbb B-\mathbb A_n\mathbb X\mathbb A_n
\right)^{-1}\right)=\tau\Big((\mathbb A_nf_n(z)\mathbb A_n+z-\mathbb B)^{-1}\Big)$, and by taking limit as $n\to +\infty$, we find
$\tau\left(\left(z-\mathbb B-\mathbb A\mathbb X\mathbb A\right)^{-1}\right)=\tau\Big((\mathbb Af(z)\mathbb A+z-\mathbb B)^{-1}\Big)$. Finally,
since $k_n$ is injective, it follows that $\mathbb A$ and $\mathbb A_n$ generate the same von Neumann algebras (not $C^*$-algebras!), so that
$E_{(\mathbb{A}_n,\mathbb B)}=E_{(\mathbb{A},\mathbb B)}$. This implies $E_{(\mathbb{A},\mathbb B)}\left[\left(z-\mathbb B-\mathbb A_n\mathbb X\mathbb A_n
\right)^{-1}\right]=\left(\mathbb A_nf_n(z)\mathbb A_n+z-\mathbb B\right)^{-1}$. The right-hand side was shown to converge to
$\left(\mathbb Af(z)\mathbb A+z-\mathbb B\right)^{-1}$. Under our hypothesis of bounded variables,
\begin{eqnarray}
\lefteqn{\left(z-\mathbb B-\mathbb A_n\mathbb X\mathbb A_n\right)^{-1}-\left(z-\mathbb B-\mathbb A\mathbb X\mathbb A\right)^{-1}}\nonumber\\
& = & \left(z-\mathbb B-\mathbb A_n\mathbb X\mathbb A_n\right)^{-1}(\mathbb{A}_n\mathbb{XA}_n-\mathbb{AXA})\left(z-\mathbb B-\mathbb A\mathbb X\mathbb A\right)^{-1}
\nonumber\\
& = & \left(z-\mathbb B-\mathbb A_n\mathbb X\mathbb A_n\right)^{-1}\left[\mathbb{A}_n\mathbb{X}(\mathbb{A}_n-\mathbb{A})+(\mathbb A_n-\mathbb A)\mathbb{XA}\right]
\left(z-\mathbb B-\mathbb A\mathbb X\mathbb A\right)^{-1},\label{trii}
\end{eqnarray}
so that
\begin{eqnarray}
\lefteqn{\left\|\left(z-\mathbb B-\mathbb A_n\mathbb X\mathbb A_n\right)^{-1}-\left(z-\mathbb B-\mathbb A\mathbb X\mathbb A\right)^{-1}\right\|}\nonumber\\
& \leq & \frac{1}{|\Im z|^2}\left(\|\mathbb{A}_n\mathbb{X}\|+\|\mathbb{XA}\|\right)\|(\mathbb A_n-\mathbb A)\|<\frac{4}{|\Im z|^2}\|\mathbb X\|\|\mathbb A\|s_n\to0\text{ as }
n\to +\infty.\label{bax}
\end{eqnarray}
Since $E_{\mathbb{A,B}}$ is wo, hence also norm, continuous, we have $E_{(\mathbb{A},\mathbb B)}\left[\left(z-\mathbb B-\mathbb A\mathbb X\mathbb A\right)^{-1}\right]
=\left(\mathbb Af(z)\mathbb A+z-\mathbb B\right)^{-1}$, proving the uniqueness of the limit point $f$ as well.
\begin{remark}\label{rmkoldf}
If $\mathbb{A,B,X}\in(\mathcal A,\tau)$ are bounded selfadjoint random variables in a tracial noncommutative probability space such that $\mathbb X$ is *-free
from $(\mathbb A,\mathbb B)$, then there exists an analytic function $f\colon\mathbb C^+\to\mathbb C^+\cup\mathbb R$ such that
\begin{eqnarray}
\tau\left(\left(z-\mathbb B-\mathbb A\mathbb X\mathbb A\right)^{-1}\right) & = & \tau\Big((\mathbb Af(z)\mathbb A+z-\mathbb B)^{-1}\Big),\\
E_{(\mathbb{A},\mathbb B)}\left[\left(z-\mathbb B-\mathbb A\mathbb X\mathbb A\right)^{-1}\right]
&=&\left(\mathbb Af(z)\mathbb A+z-\mathbb B\right)^{-1},\quad z\in\mathbb C^\pm.
\end{eqnarray}
\end{remark}
Note that estimate \eqref{bax} holds for any selfadjoint, not necessarily bounded, $\mathbb B.$
This suggests that we can extend the above result further. Let $\mathbb B$ be an arbitrary selfadjoint random variable such that $\mathbb X$ is *-free from $(\mathbb{A,B})$
with respect to $\tau$ ($\mathbb{A,X}$ remain bounded for now). For any $n\in\mathbb N$, $n>2$, let $\chi_n\colon\mathbb R\to(-n-\frac\pi2,n+\frac\pi2)$ given by
$\chi_n(t)=t$ if $t\in[-n,n]$, $\chi_n(t)=n+\arctan\left(t-n\right)$ if $t>n$, and $\chi_n(t)=-n+\arctan\left(t+n\right)$ if $t<-n$. This function is injective, continuously
differentiable, and $\lim_{n\to +\infty}\chi_n(t)=t$ for all $t\in\mathbb R$ (uniformly on compact subsets of $\mathbb R$, in fact). Denote $\mathbb B_n=\chi_n(\mathbb B)$,
in the sense of continuous functional calculus. It is obvious that the $C^*$-algebras generated by all $\{\mathbb{A,B}_n\}$ coincide, and coincinde with the $C^*$-algebra
generated by $\mathbb A$ and $\mathbb B$, and that the same statement holds even more for the corresponding von Neumann algebra. Clearly
$$
t-\chi_n(t)=\left\{\begin{array}{ll}
t-n-\arctan(t-n) & \text{ if }t>n\\
0 & \text{ if } -n\le t\le n\\
t+n-\arctan(t+n) & \text{ if }t<-n.
\end{array}\right.
$$
Then $\mathbb B-\mathbb B_n$ is the functional calculus of $t-\chi_n(t)$ applied to $\mathbb B$. For
completing the argument, let us formulate an auxiliary result in a lemma:

\begin{lemma}\label{rezlem}
Let $(\mathcal A,\tau)$ be a tracial $W^*$-probability space, acting on $L^2(\mathcal A,\tau)$,
$\mathbb{B,B}_n$ and $\chi_n$ be as above and $b\in\mathcal A,\Im b>c>0$.
Then $\lim_{n\to +\infty}(b-\mathbb B_n)^{-1}=(b-\mathbb B)^{-1}$ in the {\rm so} topology.
\end{lemma}
\begin{proof}
        Let $0\neq\xi\in L^2(\mathcal A,\tau)$. Let $p_n=\boldsymbol{1}_{[-n,n]}(\mathbb B)$, $q_n=1-p_n$. Clearly (by Borel functional calculus), $p_{n+1}\ge p_n$ for any $n\in
        \mathbb N,n\ge1$, so $q_{n+1}\le q_n$. Recall that $\mathbb B$ is selfadjoint. Thus, $p_n\to1$ as $n\to +\infty$ in the so topology (indeed, otherwise
        $1\gneqq\sup_{n\in\mathbb N}p_n=\sup_{n\in\mathbb N}\boldsymbol{1}_{[-n,n]}(\mathbb B)=\boldsymbol{1}_{\mathbb R}(\mathbb B)=1,$ a contradiction). By their definitions,
        $\mathbb B-\mathbb B_n=(\mathbb B-\mathbb B_n)q_n=q_n(\mathbb B-\mathbb B_n)$. For simplicity, let $b_{1,1}^{(n)}=q_nbq_n,b_{2,2}^{(n)}=p_nbp_n,b_{1,2}^{(n)}=
        q_nbp_n,b_{2,1}^{(n)}=p_nbq_n$, so that $b=\begin{bmatrix} b_{1,1}^{(n)} & b_{1,2}^{(n)} \\ b_{2,1}^{(n)} & b_{2,2}^{(n)} \end{bmatrix}$. In order to prove the
        lemma, we need to estimate the vector $\left\|\left(\left(b-\mathbb B_n\right)^{-1}-\left(b-\mathbb B\right)^{-1}\right)\xi\right\|_2$.
        Writing
        {\tiny{\begin{eqnarray*}
                                \lefteqn{\left(b-\mathbb B_n\right)^{-1}-\left(b-\mathbb B\right)^{-1}}\\
                                & = & \begin{bmatrix}b_{1,1}^{(n)}-q_n\mathbb B_n & b_{1,2}^{(n)} \\ b_{2,1}^{(n)} & b_{2,2}^{(n)}-p_n\mathbb B_n \end{bmatrix}^{-1}-
                                \begin{bmatrix}b_{1,1}^{(n)}-q_n\mathbb B & b_{1,2}^{(n)} \\ b_{2,1}^{(n)} & b_{2,2}^{(n)}-p_n\mathbb B \end{bmatrix}^{-1}\\
                                & = & \begin{bmatrix}b_{1,1}^{(n)}-q_n\mathbb B_n & b_{1,2}^{(n)} \\ b_{2,1}^{(n)} & b_{2,2}^{(n)}-p_n\mathbb B_n \end{bmatrix}^{-1}
                                \begin{bmatrix}q_n(\mathbb B-\mathbb B_n) & 0 \\ 0 & 0 \end{bmatrix}
                                \begin{bmatrix}b_{1,1}^{(n)}-q_n\mathbb B & b_{1,2}^{(n)} \\ b_{2,1}^{(n)} & b_{2,2}^{(n)}-p_n\mathbb B \end{bmatrix}^{-1}\\
                                & = & \begin{bmatrix}[b_{11}^{(n)}\!-q_n\mathbb B_n\!-\!b_{12}^{(n)}(b_{22}^{(n)}\!-p_n\mathbb B_n)^{-1}\!b_{21}^{(n)}]^{-1} & \!\!\!\!\!
                                        [b_{11}^{(n)}\!\!-\!q_n\mathbb B_n\!-\!b_{12}^{(n)}\!(b_{22}^{(n)}\!\!-\!p_n\mathbb B_n)^{-1}\!b_{21}^{(n)}]^{-1}\!b_{12}^{(n)}\!(p_n\mathbb B_n\!\!-\!b_{22}^{(n)})^{-1}\\
                                        (p_n\mathbb B_n\!\!-\!b_{22}^{(n)})^{-1}\!b_{21}^{(n)}\![b_{11}^{(n)}\!\!-\!q_n\mathbb B_n\!\!-\!b_{12}^{(n)}(b_{22}^{(n)}\!\!-\!p_n\mathbb B_n)^{-1}\!b_{21}^{(n)}]^{-1} &
                                        [b_{22}^{(n)}-p_n\mathbb B_n\!-b_{21}^{(n)}(b_{11}^{(n)}\!-q_n\mathbb B_n)^{-1}b_{12}^{(n)}]^{-1} \end{bmatrix}\\
                                & & \mbox{}\times\begin{bmatrix}q_n(\mathbb B-\mathbb B_n) & 0 \\ 0 & 0 \end{bmatrix}\\
                                & & \mbox{}\times\begin{bmatrix}[b_{11}^{(n)}\!\!-\!q_n\mathbb B\!-\!b_{12}^{(n)}\!(b_{22}^{(n)}\!\!-\!p_n\mathbb B)^{-1}\!b_{21}^{(n)}]^{-1} & \!\!\!\!\!
                                        [b_{11}^{(n)}\!-q_n\mathbb B\!-\!b_{12}^{(n)}\!(b_{22}^{(n)}\!\!-\!p_n\mathbb B)^{-1}\!b_{21}^{(n)}]^{-1}\!b_{12}^{(n)}\!(p_n\mathbb B\!-b_{22}^{(n)})^{-1} \\
                                        (p_n\mathbb B\!-b_{22}^{(n)})^{-1}\!b_{21}^{(n)}\![b_{11}^{(n)}\!\!-\!q_n\mathbb B\!-\!b_{12}^{(n)}(b_{22}^{(n)}\!\!-\!p_n\mathbb B)^{-1}\!b_{21}^{(n)}]^{-1} &
                                        [b_{22}^{(n)}\!\!-\!p_n\mathbb B\!-b_{21}^{(n)}\!(b_{11}^{(n)}\!\!-\!q_n\mathbb B)^{-1}\!b_{12}^{(n)}]^{-1} \end{bmatrix}\\
                                & = & \begin{bmatrix}[b_{11}^{(n)}\!-q_n\mathbb B_n\!-\!b_{12}^{(n)}(b_{22}^{(n)}\!-p_n\mathbb B_n)^{-1}\!b_{21}^{(n)}]^{-1}q_n(\mathbb B-\mathbb B_n) & 0 \\
                                        (p_n\mathbb B_n\!\!-\!b_{22}^{(n)})^{-1}\!b_{21}^{(n)}\![b_{11}^{(n)}\!\!-\!q_n\mathbb B_n\!\!-\!b_{12}^{(n)}(b_{22}^{(n)}\!\!-\!p_n\mathbb B_n)^{-1}\!b_{21}^{(n)}]^{-1}
                                        q_n(\mathbb B-\mathbb B_n) & 0 \end{bmatrix}\\
                                & & \mbox{}\times\begin{bmatrix}[b_{11}^{(n)}\!\!-\!q_n\mathbb B\!-\!b_{12}^{(n)}\!(b_{22}^{(n)}\!\!-\!p_n\mathbb B)^{-1}\!b_{21}^{(n)}]^{-1} & \!
                                        [b_{11}^{(n)}\!-q_n\mathbb B\!-\!b_{12}^{(n)}\!(b_{22}^{(n)}\!\!-\!p_n\mathbb B)^{-1}\!b_{21}^{(n)}]^{-1}\!b_{12}^{(n)}\!(p_n\mathbb B\!-b_{22}^{(n)})^{-1} \\
                                        0 & 0 \end{bmatrix}\!,\\
        \end{eqnarray*}}}
        we observe that the entries (in the order $(1,1),(1,2),(2,1),(2,2)$) of the above are
        \begin{align*}
                & [b_{11}^{(n)}\!-q_n\mathbb B_n\!-\!b_{12}^{(n)}(b_{22}^{(n)}\!-p_n\mathbb B_n\!)^{-1}\!b_{21}^{(n)}]^{-1}\!q_n(\mathbb B\!-\!\mathbb B_n)
                [b_{11}^{(n)}\!\!-\!q_n\mathbb B\!-\!b_{12}^{(n)}\!(b_{22}^{(n)}\!\!-\!p_n\mathbb B)^{-1}\!b_{21}^{(n)}]^{-1},\\
                & [b_{11}^{(n)}\!-q_n\mathbb B_n\!-\!b_{12}^{(n)}(b_{22}^{(n)}\!-p_n\mathbb B_n\!)^{-1}\!b_{21}^{(n)}]^{-1}\!q_n(\mathbb B\!-\!\mathbb B_n)
                [b_{11}^{(n)}\!-q_n\mathbb B\!-\!b_{12}^{(n)}\!(b_{22}^{(n)}\!\!-\!p_n\mathbb B)^{-1}\!b_{21}^{(n)}]^{-1}\\
                &\mbox{}\times\!b_{12}^{(n)}(p_n\mathbb B\!-b_{22}^{(n)})^{-1},\\
                & (p_n\mathbb B_n\!\!-\!b_{22}^{(n)})^{-1}\!b_{21}^{(n)}\![b_{11}^{(n)}\!\!-\!q_n\mathbb B_n\!\!-\!b_{12}^{(n)}(b_{22}^{(n)}\!\!-\!p_n\mathbb B_n)^{-1}\!
                b_{21}^{(n)}]^{-1}\!q_n(\mathbb B\!-\!\mathbb B_n)\\
                & \mbox{}\times[b_{11}^{(n)}\!\!-\!q_n\mathbb B\!-\!b_{12}^{(n)}\!(b_{22}^{(n)}\!\!-\!p_n\mathbb B)^{-1}\!b_{21}^{(n)}]^{-1},\\
                & (p_n\mathbb B_n\!\!-\!b_{22}^{(n)})^{-1}\!b_{21}^{(n)}\![b_{11}^{(n)}\!\!-\!q_n\mathbb B_n\!\!-\!b_{12}^{(n)}(b_{22}^{(n)}\!\!-\!p_n\mathbb B_n)^{-1}\!
                b_{21}^{(n)}]^{-1}\!q_n(\mathbb B\!-\!\mathbb B_n)\\
                &\mbox{}\times[b_{11}^{(n)}\!-q_n\mathbb B\!-\!b_{12}^{(n)}\!(b_{22}^{(n)}\!\!-\!p_n\mathbb B)^{-1}\!b_{21}^{(n)}]^{-1}\!b_{12}^{(n)}(p_n\mathbb B\!-b_{22}^{(n)})^{-1}\!.
        \end{align*}
        The first and the third apply to $q_n\xi$,  the second and fourth to $p_n\xi$. We consider each of them separately. For\footnote{We use here the fact that if $x_n\in\mathcal A$
                satisfies $x_n^*x_n\to0$ wo, then $x_n\to0$ so. Indeed, $0=\lim_n\langle x_n^*x_n\xi,\xi\rangle=\lim_n\|x_n\xi\|_2^2$. Also, if $\|x_n\|<M$ for all $n$ and so-$x_n\to x$,
                then $x_n^*\to x^*$, $x_n^*x_n\to x^*x$ in the wo topology. Indeed, $\langle(x_n^*-x^*)\xi,\eta\rangle=\langle\xi,(x_n-x)\eta\rangle$,
                $\langle(x_n^*x_n-x^*x)\xi,\eta\rangle=\langle x_n^*(x_n-x)\xi+(x_n^*-x^*)x\xi,\eta\rangle$.} $(1,1)$:
        {\small{
                        \begin{align*}
                                & \left([b_{11}^{(n)}\!-q_n\mathbb B_n\!-\!b_{12}^{(n)}(b_{22}^{(n)}\!-p_n\mathbb B_n\!)^{-1}\!b_{21}^{(n)}]^{-1}\!q_n(\mathbb B\!-\!\mathbb B_n)
                                [b_{11}^{(n)}\!\!-\!q_n\mathbb B\!-\!b_{12}^{(n)}\!(b_{22}^{(n)}\!\!-\!p_n\mathbb B)^{-1}\!b_{21}^{(n)}]^{-1}\right)^*\\
                                &\mbox{}\times[b_{11}^{(n)}\!-q_n\mathbb B_n\!-\!b_{12}^{(n)}(b_{22}^{(n)}\!-p_n\mathbb B_n\!)^{-1}\!b_{21}^{(n)}]^{-1}\!q_n(\mathbb B\!-\!\mathbb B_n)
                                [b_{11}^{(n)}\!\!-\!q_n\mathbb B\!-\!b_{12}^{(n)}\!(b_{22}^{(n)}\!\!-\!p_n\mathbb B)^{-1}\!b_{21}^{(n)}]^{-1}\\
                                & \leq\frac{1}{c^2}\!\left([b_{11}^{(n)}\!\!-\!q_n\mathbb B\!-\!b_{12}^{(n)}\!(b_{22}^{(n)}\!\!-\!p_n\mathbb B)^{-1}\!b_{21}^{(n)}]^{-1}\right)^*\!\!
                                (\mathbb B\!-\!\mathbb B_n)^2[b_{11}^{(n)}\!\!-\!q_n\mathbb B\!-\!b_{12}^{(n)}\!(b_{22}^{(n)}\!\!-\!p_n\mathbb B)^{-1}\!b_{21}^{(n)}]^{-1}\\
                                &\leq\frac{4}{c^2}\left(1+\frac{\|b\|}{c}\right)^4q_n\to0\text{ as }n\to +\infty\text{ in the so topology;}
                        \end{align*}
        }}we have used the facts that $\|(b-\mathbb B)^{-1}\|,\|(b-\mathbb B_n)^{-1}\|\leq\frac{1}{c}$, that $
        [b_{11}^{(n)}\!\!-\!q_n\mathbb B\!-\!b_{12}^{(n)}\!(b_{22}^{(n)}\!\!-\!p_n\mathbb B)^{-1}\!b_{21}^{(n)}]^{-1}=q_n
        [b_{11}^{(n)}\!\!-\!q_n\mathbb B\!-\!b_{12}^{(n)}\!(b_{22}^{(n)}\!\!-\!p_n\mathbb B)^{-1}\!b_{21}^{(n)}]^{-1}=
        [b_{11}^{(n)}\!\!-\!q_n\mathbb B\!-\!b_{12}^{(n)}\!(b_{22}^{(n)}\!\!-\!p_n\mathbb B)^{-1}\!b_{21}^{(n)}]^{-1}q_n$, and the same for $\mathbb B_n$,
        and $\|(\mathbb B-\mathbb B_n)[b_{11}^{(n)}\!\!-\!q_n\mathbb B\!-\!b_{12}^{(n)}\!(b_{22}^{(n)}\!\!-\!p_n\mathbb B)^{-1}\!b_{21}^{(n)}]^{-1}\|<\frac{2}{c}+1$. (The
        only maybe not immediately obvious statement is the last. For this, note that if $\mathbb K$ is selfadjoint, possibly unbounded, and $\Im\beta>c>0$, then
        $\mathbb K(\beta-\mathbb K)^{-1}=\beta(\beta-\mathbb K)^{-1}-1$, so $\|\mathbb K(\beta-\mathbb K)^{-1}\|\le1+\frac{\|\beta\|}{c}$; apply this to $\mathbb K=\mathbb Bq_n$,
        $\beta=b_{11}^{(n)}\!\!-\!b_{12}^{(n)}\!(b_{22}^{(n)}\!\!-\!p_n\mathbb B)^{-1}\!b_{21}^{(n)}$, and recall that $\mathbb B_n^2\le\mathbb B^2$. One has
        $\|b_{11}^{(n)}\!\!-\!b_{12}^{(n)}\!(b_{22}^{(n)}\!\!-\!p_n\mathbb B)^{-1}\!b_{21}^{(n)}\|\le\|b\|+\frac{\|b\|^2}{c}$.)
       
        For $(2,2)$, we see that $(2,2)=(p_n\mathbb B_n-b_{22}^{(n)})^{-1}b_{21}^{(n)}(1,1)b_{12}^{(n)}(p_n\mathbb B-b_{22}^{(n)})^{-1}$, so that $(2,2)^*(2,2)
        =(p_n\mathbb B-(b_{22}^{(n)})^*)^{-1}(b_{12}^{(n)})^*(1,1)^*(b_{21}^{(n)})^*(p_n\mathbb B_n-(b_{22}^{(n)})^*)^{-1}(p_n\mathbb B_n-b_{22}^{(n)})^{-1}
        b_{21}^{(n)}(1,1)\allowbreak b_{12}^{(n)}(p_n\mathbb B-b_{22}^{(n)})^{-1}\!\le\frac{\|b\|^2}{c^2}(p_n\mathbb B-(b_{22}^{(n)})^*)^{-1}(b_{12}^{(n)})^*(1,1)^*
        (1,1)b_{12}^{(n)}(p_n\mathbb B-b_{22}^{(n)})^{-1}\!\le\frac{4\|b\|^2}{c^4}\left(\!1+\frac{\|b\|}{c}\right)^4\allowbreak(p_n\mathbb B-(b_{22}^{(n)})^*)^{-1}(b_{12}^{(n)})^*
        b_{12}^{(n)}(p_n\mathbb B-b_{22}^{(n)})^{-1}$\!. We claim that $(p_n\mathbb B-(b_{22}^{(n)})^*)^{-1}(b_{12}^{(n)})^*
        b_{12}^{(n)}(p_n\mathbb B-b_{22}^{(n)})^{-1}\to0$ as $n\to +\infty$ in the so topology. Obviously, since $b_{12}^{(n)}=q_nbp_n$, if $\xi\in L^2(\mathcal A,\tau)$ is fixed,
        $b_{12}^{(n)}(p_n\mathbb B-b_{22}^{(n)})^{-1}\xi=q_nbp_n(p_n\mathbb B-b_{22}^{(n)})^{-1}\xi$; observe
        \begin{eqnarray*}
                \lefteqn{(b-\mathbb B)^{-1}-q_n-(b_{22}^{(n)}-p_n\mathbb B)^{-1}=(b-\mathbb B)^{-1}-\begin{bmatrix}
                                q_n & 0 \\
                                0 & (b_{22}^{(n)}-p_n\mathbb B)^{-1}
                \end{bmatrix}}\\
                & = & (b-\mathbb B)^{-1}-\begin{bmatrix}
                        q_n & 0 \\
                        0 & p_nbp_n-p_n\mathbb B
                \end{bmatrix}^{-1}\\
                & = & \begin{bmatrix}
                        q_n & 0 \\
                        0 & (p_nbp_n-p_n\mathbb B)^{-1}
                \end{bmatrix}\left(\begin{bmatrix}
                        q_n & 0 \\
                        0 & p_nbp_n-p_n\mathbb B
                \end{bmatrix}-b+\mathbb B\right)(b-\mathbb B)^{-1}\\
                & = & \begin{bmatrix}
                        q_n & 0 \\
                        0 & (p_nbp_n-p_n\mathbb B)^{-1}
                \end{bmatrix}\begin{bmatrix}
                        q_n+q_n\mathbb B-q_nbq_n & -q_nbp_n \\
                        -p_nbq_n & 0
                \end{bmatrix}(b-\mathbb B)^{-1}\\
                & = & \begin{bmatrix}
                        q_n+q_n\mathbb B-q_nbq_n & -q_nbp_n \\
                        -(p_nbp_n-p_n\mathbb B)^{-1}p_nbq_n & 0
                \end{bmatrix}(b-\mathbb B)^{-1}.
        \end{eqnarray*}
(Inverses $(b_{22}^{(n)}-p_n\mathbb B)^{-1}$ are considered in the reduced algebra.)
        Apply this to an arbitrary vector $\xi$, recalling that $p_n=1-q_n$. Entry by entry, one gets $(q_n-q_nbq_n)(b-\mathbb B)^{-1}\xi\to0$ as $n\to +\infty$
        because $q_n\to0$ in the so topology, $q_n\mathbb B(b-\mathbb B)^{-1}\xi=q_nb(b-\mathbb B)^{-1}\xi-q_n\xi\to0$  as $n\to +\infty$ for the same reason, then
        $q_nbp_n(b-\mathbb B)^{-1}\xi=q_nb(b-\mathbb B)^{-1}\xi-q_nbq_n(b-\mathbb B)^{-1}\xi\to0$, again for the same reason, and finally
        $\|(p_nbp_n-p_n\mathbb B)^{-1}p_nbq_n(b-\mathbb B)^{-1}\xi\|_2\leq\|(p_nbp_n-p_n\mathbb B)^{-1}p_nb\|\|q_n(b-\mathbb B)^{-1}\xi\|_2\le
        \frac{\|b\|}{c}\|q_n(b-\mathbb B)^{-1}\xi\|_2\to0$ as $n\to +\infty.$ Thus,
        \begin{eqnarray*}
                q_nb(p_n\mathbb B-b_{22}^{(n)})^{-1}\xi & = & q_nb((p_n\mathbb B-b_{22}^{(n)})^{-1}+q_n-(b-\mathbb B)^{-1}+(b-\mathbb B)^{-1}-q_n)\xi\\
                & = & q_nb\left((p_n\mathbb B-b_{22}^{(n)})^{-1}+q_n-(b-\mathbb B)^{-1}\right)\xi\\
                & & \mbox{}+q_nb((b-\mathbb B)^{-1}-q_n)\xi,
        \end{eqnarray*}
        quantity that converges to zero in norm. As $\|(p_n\mathbb B-b_{22}^{(n)})^{-1}(b_{12}^{(n)})^*\|<\frac{\|b\|}{c}$ uniformly in $n$, our claim is proved. The convergence to
        zero in the so topology of entries $(1,2)$ and $(2,1)$ follow from the above with the same methods.
\end{proof}

From the above lemma we can conclude that Remark
\ref{rmkoldf} holds with $\mathbb{X,A}\in\mathcal A$ and $\mathbb B\in\tilde{\mathcal A}$. Next, let
us consider the case of $\mathbb A\in\mathcal A$ invertible, $\mathbb B\in\mathcal A,\mathbb X\in
\tilde{\mathcal A}$. Applying Lemma \ref{rezlem} with $b=\mathbb A^{-1}(z-\mathbb B)
\mathbb A^{-1}$ and $\mathbb B$ replaced by $\mathbb X$ tells us that Remark \ref{rmkoldf}
holds in this case as well. Finally, we consider the case when $\mathbb A\in\mathcal A$, and
$\mathbb B,\mathbb X\in\tilde{\mathcal A}$. In this case, we are satisfied with wo convergence, since
$E_{(\mathbb{A,B})}$ is wo continuous. As seen above, $\mathbb B-\mathbb B_n\to0,\mathbb X-
\mathbb X_n\to0$ as $n\to +\infty$ in the so topology, in the sense that if $\xi$ belongs to the domains
of $\mathbb B$, resp  $\mathbb X$, then $\|(\mathbb B-\mathbb B_n)\xi\|_2\to0$, resp.
$\|(\mathbb X-\mathbb X_n)\xi\|_2\to0$, as $n\to +\infty$. Since $\mathbb A$ is bounded, it follows
that $\mathbb B_n+\mathbb{AX}_n\mathbb A\to\mathbb B+\mathbb{AXA}$ in the so topology as
well, again on the domain of $\mathbb B+\mathbb{AXA}$. On the other hand, $\{(z-\mathbb B_n-
\mathbb{AX}_n\mathbb A)^{-1}\}_{n\in\mathbb N}$ is a sequence of operators in $\mathcal A$ which
is bounded (in norm) by $\frac{1}{|\Im z|}$, independently of $n$. Since $\mathcal A$ is a von
Neumann algebra, this set is precompact, so it has limit points, for example in the wo topology. If
$\mathbf x=\mathbf x(z)$ is such a limit point, then, first, $\|\mathbf x\|\le\frac{1}{|\Im z|}$
and, second, it obviously depends analytically on $z$ (simply because for any wo continuous linear
functional $\varphi$, the map $z\mapsto\varphi(\mathbf x(z))$ is a limit of analytic maps in the topology
of uniform convergence on compacts). Obviously, since $(z-\mathbb B_n+\mathbb Af_n(z)
\mathbb A)^{-1}=E_{(\mathbb{A,B}_n)}\left[(z-\mathbb B_n-\mathbb{AX}_n\mathbb A)^{-1}
\right]=E_{(\mathbb{A,B})}\left[(z-\mathbb B_n-\mathbb{AX}_n\mathbb A)^{-1}\right]$, one has
$E_{(\mathbb{A,B})}\left[\mathbf x(z)\right]=(z-\mathbb B+\mathbb Af(z)\mathbb A)^{-1}$, with
$f$ being the limit point of $\{f_n\}$ along the subsequence corresponding
to $\mathbf x(z)$. If $\xi,\eta$ are nonzero vector in the (dense) domain of of $\mathbb B+
\mathbb{AXA}$, then $\langle\mathbf x(z)\xi,(\overline{z}-\mathbb B-\mathbb{AXA})
\eta\rangle=\lim_{k\to +\infty}\langle(z-\mathbb B-\mathbb{AXA})(z-\mathbb B_{n_k}-
\mathbb{AX}_{n_k}\mathbb A)^{-1}\xi,\eta\rangle=\lim_{k\to +\infty}\langle(z-\mathbb B_{n_k}-
\mathbb{AX}_{n_k}\mathbb{A})(z-\mathbb B_{n_k}-\mathbb{AX}_{n_k}\mathbb A)^{-1}\xi,\eta
\rangle+\lim_{k\to +\infty}\langle(\mathbb B_{n_k}+\mathbb{AX}_{n_k}\mathbb{A}-\mathbb B-
\mathbb{AXA})(z-\mathbb B_{n_k}-\mathbb{AX}_{n_k}\mathbb A)^{-1}\xi,\eta\rangle
=\langle\xi,\eta\rangle+\lim_{k\to +\infty}\langle(z-\mathbb B_{n_k}-\mathbb{AX}_{n_k}
\mathbb A)^{-1}\xi,(\mathbb B_{n_k}-\mathbb B+\mathbb{A}(\mathbb{X}_{n_k}-\mathbb{X})
\mathbb{A})\eta\rangle$. Since $\eta$ is in the domain of the operator on the last term, one has
\begin{eqnarray*}
\lefteqn{\left|\langle(z-\mathbb B_{n_k}-\mathbb{AX}_{n_k}\mathbb A)^{-1}\xi,(\mathbb B_{n_k}-
\mathbb B+\mathbb{A}(\mathbb{X}_{n_k}-\mathbb{X})\mathbb{A})\eta\rangle\right|}\\
& \le & \left\|(z-\mathbb B_{n_k}-\mathbb{AX}_{n_k}\mathbb A)^{-1}\xi\right\|_2\left(\left\|(\mathbb B_{n_k}-\mathbb B)\xi\right\|_2+\|\mathbb A\|\left\|(\mathbb{X}_{n_k}-
\mathbb{X})\mathbb{A}\eta\right\|_2\right)\\
& \le & \frac{1}{|\Im z|}\left(\left\|(\mathbb B_{n_k}-\mathbb B)\xi\right\|_2+\|\mathbb A\|
\left\|(\mathbb{X}_{n_k}-\mathbb{X})\mathbb{A}\eta\right\|_2\right)\to0\text{ as }k\to +\infty.
\end{eqnarray*}
It follows that $\langle\mathbf x(z)\xi,(\overline{z}-\mathbb B-\mathbb{AXA})
\eta\rangle=\langle\xi,\eta\rangle$ on a dense set of vectors $\xi,\eta$ in the Hilbert space on which
$\mathcal A$ acts. This tells us that the apriori arbitrary cluster point $\mathbf x(z)$ is the inverse of
$z-\mathbb B-\mathbb{AXA}$. It follows that $(z-\mathbb B_n-\mathbb{AX}_n\mathbb A)^{-1}$
converges in the weak operator topology to $(z-\mathbb B-\mathbb{AXA})^{-1}$. These two facts
together guarantee

\begin{proposition}\label{rmkBig}
If $\mathbb{A,B,X}$ are selfadjoint random variables in a tracial noncommutative probability state
$(\mathcal A,\tau)$ such that $\mathbb{A}$ is bounded and $\mathbb X$ is *-free from
$(\mathbb A,\mathbb B)$, then there exists an analytic function $f\colon\mathbb C^+\to\mathbb C^+
\cup\mathbb R$ such that
\begin{eqnarray}
\tau\left(\left(z-\mathbb B-\mathbb A\mathbb X\mathbb A\right)^{-1}\right)
& = & \tau\Big((\mathbb Af(z)\mathbb A+z-\mathbb B)^{-1}\Big),\label{ab}\\
E_{(\mathbb{A},\mathbb B)}\left[\left(z-\mathbb B-\mathbb A\mathbb X\mathbb A\right)^{-1}\right]
&=&\left(\mathbb Af(z)\mathbb A+z-\mathbb B\right)^{-1},\quad z\in\mathbb C^\pm.\label{AB}
\end{eqnarray}
\end{proposition}

We go to the next step of our description of the subordination phenomenon in our context. In this step,
we assume that
\begin{itemize}
\item $\mathbb X=\mathbb X^*,\mathbb A=\mathbb A^*,\mathbb B=\mathbb B^*\in\tilde{\mathcal A}$;
\item $\mathbb X$ is invertible in $\tilde{\mathcal A}$, meaning that there exists an $\mathbb X^{-1}
\in\tilde{\mathcal A}$ such that $\mathbb{XX}^{-1}\xi=\xi$ for all $\xi$ in the domain of
$\mathbb X^{-1}$ and $\mathbb X^{-1}\mathbb X\eta=\eta$ for all $\eta$ in the domain of
$\mathbb X$ - both dense subspaces;
\item $\tau(\mathbb A^2)+\tau(\mathbb X^2)<+\infty$.
\end{itemize}
(The hypothesis that $\mathbb X,\mathbb A\not\in\mathbb R\cdot1\subset\mathcal A$ is maintained
throughout the section.) Define $\mathbf X=\begin{bmatrix} 0 & 0 \\ 0 & -\mathbb X^{-1}
\end{bmatrix},\mathbf A=\begin{bmatrix} \mathbb B & -\mathbb A \\ -\mathbb A & 0\end{bmatrix}
\in(\tilde{\mathcal A}\otimes M_2(\mathbb C),\tau\otimes{\rm tr}_2)$. By hypothesis, they are free
over $M_2(\mathbb C)$ with respect to $\tau\otimes{\rm tr}_2$. According to \cite{Vo00} (or
\cite{BBs} if $\tau$ is not a trace - and all operators are bounded), there exist two analytic self-maps
of the upper half-plane of $M_2(\mathbb C)$, which we call here $\alpha$ and $\omega$, such that
\begin{eqnarray}
(\tau\otimes{\rm tr}_2)\left[(b-\mathbf{A-X})^{-1}\right]
& = & (\tau\otimes{\rm tr}_2)\left[(\alpha(b)-\mathbf{A})^{-1}\right]\\
& = & (\tau\otimes{\rm tr}_2)\left[(\omega(b)-\mathbf{X})^{-1}\right]\nonumber\\
& = & (\omega(b)+\alpha(b)-b)^{-1}.\nonumber
\end{eqnarray}
This holds on any domain containing elements $b$ with positive definite imaginary part, i.e. the upper
half-plane of $M_2(\mathbb C)$, on which $b-\mathbf{A-X}$ and $(\tau\otimes{\rm tr}_2)
\left[(b-\mathbf{A-X})^{-1}\right]$ are invertible. Moreover, if $E_{\bf A}$ denotes the conditional
expectation onto the von Neumann algebra generated by $M_2(\mathbb C)$ and {\bf A}, and
$E_{\bf X}$ denotes the conditional expectation onto the von Neumann algebra
generated by $M_2(\mathbb C)$ and {\bf X}, then
\begin{eqnarray}
E_{\bf A}\left[(b-\mathbf{A-X})^{-1}\right] & = & (\alpha(b)-\mathbf{A})^{-1},\text{ and}\\
E_{\bf X}\left[(b-\mathbf{A-X})^{-1}\right] & = & (\omega(b)-\mathbf{X})^{-1}.
\end{eqnarray}
It is known (see, for instance, \cite{Takesaki}) that if $\tau$ is a normal faithful trace, then $E_{\bf A}$
and $E_{\bf X}$ are also continuous in the strong operator (so) and weak operator (wo) topologies. Let
$e_{11}=\begin{bmatrix} 1 & 0 \\ 0 & 0 \end{bmatrix}$. By Schur's complement formula (or direct
verification),
\begin{eqnarray*}
(ze_{11}-\mathbf{A-X})^{-1}& = &\begin{bmatrix} z-\mathbb B & \mathbb A \\
\mathbb A & \mathbb X^{-1}\end{bmatrix}^{-1}\\
& = & \begin{bmatrix} (z-\mathbb B-\mathbb A\mathbb X\mathbb A)^{-1} & -(z-\mathbb B-\mathbb A\mathbb X\mathbb A)^{-1}\mathbb{AX} \\
-\mathbb{X A}(z-\mathbb B-\mathbb A\mathbb X\mathbb A)^{-1} & \mathbb X+\mathbb{XA}(z-\mathbb B-\mathbb A\mathbb X\mathbb A)^{-1}\mathbb{AX}\end{bmatrix}\\
& = & \begin{bmatrix} 1 \\ -\mathbb{XA}\end{bmatrix}
\begin{array}{cc} \begin{bmatrix}(z-\mathbb B-\mathbb{AXA})^{-1}\end{bmatrix}
& \!\!\!\begin{bmatrix} 1 & -\mathbb{AX} \end{bmatrix}\\
\  & \ 
\end{array}+\begin{bmatrix} 0 & 0 \\ 0 & \mathbb X \end{bmatrix}.
\end{eqnarray*}
Of course, applying $E_{\bf A}$ to the above is equivalent to applying $E_{(\mathbb{A,B})}$, the
conditional expectation onto the von Neumann algebra generated by $(\mathbb{A,B})$ inside
$\mathcal A$, entrywise, and the same for ${\bf X}$ and $\mathbb X$. Since $\mathbb X$ is
selfadjoint, it generates an Abelian von Neumann algebra. It is obvious that $ze_{11}-\mathbf{A-X}$
is indeed invertible for all $z\not\in\sigma(\mathbb B+\mathbb {AXA})$, and in particular for all $z\in
\mathbb C^\pm$. Moreover, the expression in the right-hand side is well-defined regardless of whether
$\mathbb X$ is invertible or not. However, in order to apply $\tau$, we need the entries of this inverse
to be in the domain of $\tau$. This will ensure that  $z\mapsto(\tau\otimes{\rm tr}_2)\left[(ze_{11}-
\mathbf{A-X})^{-1}\right],z\mapsto E_{\bf X}\left[(ze_{11}-\mathbf{A-X})^{-1}\right],z\mapsto
E_{\bf A}\left[(ze_{11}-\mathbf{A-X})^{-1}\right]$ are well-defined on all of
$\mathbb C\setminus\sigma(\mathbb B+\mathbb{AXA})$. Since $\tau(\mathbb X^2)$ is finite by
hypothesis, clearly $\tau(\mathbb X)\in\mathbb R$ is well-defined. The inequality
$\|(z-\mathbb B-\mathbb{AXA})^{-1}\|\le\frac{1}{|\Im z|},z\in\mathbb C^\pm,$ implies $\|\Re(z-
\mathbb B-\mathbb{AXA})^{-1}\|\le\frac{1}{|\Im z|}$, $\|\Im(z-\mathbb B-\mathbb{AXA})^{-1}\|\le
\frac{1}{|\Im z|}$, so that
$0\le-\Im\tau(\mathbb{XA}(z-\mathbb B-\mathbb A\mathbb X\mathbb A)^{-1}\mathbb{AX})
\le\frac{\tau(\mathbb{XAAX})}{|\Im z|}=\frac{\tau(\mathbb{X}^2)\tau(\mathbb{A}^2)}{|\Im z|}<
\infty$ by hypothesis (here we have used the freeness of $\mathbb X$ and
$\mathbb A$), and $|\Re\tau(\mathbb{XA}(z-\mathbb B-\mathbb A\mathbb X\mathbb A)^{-1}\mathbb{AX})|
=|\tau(\mathbb{XA}\Re(z-\mathbb B-\mathbb A\mathbb X\mathbb A)^{-1}\mathbb{AX})|<
\frac{\tau(\mathbb{X}^2)\tau(\mathbb{A}^2)}{|\Im z|}<+\infty$.
Thus, the $(2,2)$ entry is in the domain of $\tau$. The argument for the off-diagonal entries is
similar. Next, we argue that $(\tau\otimes{\rm tr}_2)
\left[(ze_{11}-\mathbf{A-X})^{-1}\right]$ is invertible in $M_2(\mathbb C)$. From its expression, it is clear that
$$
\Im(ze_{11}-\mathbf{A-X})^{-1}
=-\Im z\begin{bmatrix} 1 \\ -\mathbb{XA}\end{bmatrix}\begin{array}{cc} \begin{bmatrix}\left((\Im z)^2+(\Re z-
\mathbb B-\mathbb{AXA})^2\right)^{-1}\end{bmatrix} & \!\!\!\begin{bmatrix} 1 & -\mathbb{AX} \end{bmatrix}\\
\  & \ 
\end{array}.
$$
Note that $\displaystyle\sphericalangle\lim_{z\to +\infty}(\Im z)^2\left((\Im z)^2+(\Re z-\mathbb B-
\mathbb{AXA})^2\right)^{-1}=1\in\mathcal A$ in the so topology.
It is well-known (and easy to deduce) that rank$\Im(\tau\otimes{\rm id}_{M_2(\mathbb C)})
\left[(ze_{11}-\mathbf{A-X})^{-1}\right]$ is constant on $\mathbb C^\pm$, equal to
one or two. Clearly ${\rm rank}\Im(\tau\otimes{\rm id}_{M_2(\mathbb C)})
\left[(ze_{11}-\mathbf{A-X})^{-1}\right]\allowbreak={\rm rank}\Im z\Im(\tau\otimes
{\rm id}_{M_2(\mathbb C)})\left[(ze_{11}-\mathbf{A-X})^{-1}\right]$ for all $z\in\mathbb C^+$.
For fixed $\Re z$,
$$
\lim_{\Im z\to +\infty}\Im z\Im(\tau\otimes{\rm id}_{M_2(\mathbb C)})
\left[(ze_{11}-\mathbf{A-X})^{-1}\right]=\begin{bmatrix} -1 & \tau(\mathbb{AX}) \\
\tau(\mathbb{XA}) & -\tau(\mathbb{AXXA})
\end{bmatrix}
$$
Typically this matrix is invertible: since the variables $\mathbb A$ and $\mathbb X$ are free, it follows that
$\tau(\mathbb{AX})=\tau(\mathbb{XA})=\tau(\mathbb A)\tau(\mathbb X)$
and $\tau(\mathbb{AXXA})=\tau(\mathbb X^2)\tau(\mathbb A^2)$. Under the assumption that
$\mathbb X\neq0\neq\mathbb A$ and that at least one of $\mathbb{X,A}$
is not a multiple of the identity, it follows that
$\begin{bmatrix} -1 & \tau(\mathbb{AX}) \\ \tau(\mathbb{XA}) & -\tau(\mathbb{AXXA})
\end{bmatrix}<0$. Since the limit of rank one two-by-two matrices cannot be invertible, one concludes
that $\Im(\tau\otimes{\rm id}_{M_2(\mathbb C)})
\left[(ze_{11}-\mathbf{A-X})^{-1}\right]<0$, and hence $(\tau\otimes{\rm id}_{M_2(\mathbb C)})
\left[(ze_{11}-\mathbf{A-X})^{-1}\right]$ is invertible in $M_2(\mathbb C)$ for all
$z\in\mathbb C^\pm$.

While not yet relevant, we take the opportunity now to record also the fact that
\begin{eqnarray*}
\lefteqn{\Re(ze_{11}-\mathbf{A-X})^{-1}=\begin{bmatrix} 0 & 0 \\ 0 & \mathbb X \end{bmatrix}}\\
& & \mbox{}+\begin{bmatrix} 1 \\ -\mathbb{XA}\end{bmatrix}\begin{array}{cc} \begin{bmatrix}
(\Re z-\mathbb B-\mathbb{AXA})((\Im z)^2+(\Re z-\mathbb B-\mathbb{AXA})^2)^{-1}\end{bmatrix}
& \!\!\!\begin{bmatrix} 1 & -\mathbb{AX} \end{bmatrix}\\
\  & \ 
\end{array}.
\end{eqnarray*}
This tells us that
$$
\sphericalangle\lim_{z\to +\infty}zE_{\bullet}\left[(ze_{11}-\mathbf{A-X})^{-1}-\begin{bmatrix} 0 & 0
\\ 0 & \mathbb X \end{bmatrix}\right]=E_\bullet
\left[\begin{bmatrix} 1 \\ -\mathbb{XA}\end{bmatrix}\begin{array}{c}
\!\!\!\begin{bmatrix} 1 & -\mathbb{AX} \end{bmatrix}\\
 \ 
\end{array}\right],
$$
where $\bullet\in\{\bf X,\bf A\}$ (here we have used the wo-continuity of $E_\bullet$).

\begin{remark}\label{condiR}
Interestingly enough,
$$
\lim_{\Im z\to +\infty}\!\Im z\Im E_{\bf X}\!\left[\!(ze_{11}\!-\mathbf{A-X})^{-1}\!\right]=
\begin{bmatrix} -1 & E_\mathbb X[\mathbb{AX}] \\ E_\mathbb X[\mathbb{XA}] &
\!-E_\mathbb X[\mathbb{XA}^2\mathbb X]\end{bmatrix}=\begin{bmatrix} -1 & \mathbb X
\tau(\mathbb{A}) \\ \tau(\mathbb{A})\mathbb X & \!-\tau(\mathbb{A}^2)
\mathbb X^2\end{bmatrix},
$$
$$
\lim_{\Im z\to +\infty}\!\Im z\Im E_{\bf A}\!\left[\!(ze_{11}\!-\mathbf{A-X})^{-1}\!\right]=
\begin{bmatrix} -1 & E_\mathbb A[\mathbb{AX}]\\ E_\mathbb A[\mathbb{XA}] & \!
-E_\mathbb A[\mathbb{XA}^2\mathbb X]\end{bmatrix}
$$
$$
=\begin{bmatrix}-1 & \tau(\mathbb{X})\mathbb A \\ \tau(\mathbb{X})\mathbb A & -
\tau(\mathbb{X})^2\mathbb A^2-\tau((\mathbb X-\tau(\mathbb X))^2)\tau(\mathbb A^2)
\end{bmatrix}.
$$
Since $\mathbb X$ is invertible in $\tilde{\mathcal A}$, the first operator is invertible in
$M_2(\tilde{\mathcal A})$, while the second operator is less than or equal to zero (the
upper bound of its spectrum is zero if and only if $\mathbb A$ is unbounded), and its
inverse makes sense, again in $M_2(\tilde{\mathcal A})$.
It follows that for any rank one projection $0\neq p\in M_2(\mathbb C)\subset M_2(\mathcal A)$,
$\Im pE_{\bf A}\!\left[(ze_{11}-\mathbf{A-X})^{-1}\right]p\neq0,
\Im pE_{\bf X}\!\left[\!(ze_{11}\!-\mathbf{A-X})^{-1}\!\right]p\neq0$.
\end{remark}

We have shown that
$b\mapsto(\tau\otimes{\rm id}_{M_2(\mathbb C)})\left[(b-\mathbf{A-X})^{-1}\right]$, well-defined
for elements $b\in M_2(\mathbb C)$ with strictly positive or strictly negative imaginary part, extends
continuously to $(\mathbb C\setminus\sigma(\mathbb B+\mathbb{AXA})) e_{11}$ as an analytic
function of $z\in\mathbb C\setminus\sigma(\mathbb B+\mathbb{AXA})$ and that $\mathbb C^\pm\ni
z\mapsto(\tau\otimes{\rm id}_{M_2(\mathbb C)})\left[(ze_{11}-\mathbf{A-X})^{-1}\right]^{-1}$ is
well-defined and analytic. This guarantees that the two subordination functions $\alpha,\omega$ are
defined (and analytic) on $\mathbb C^\pm e_{11}$ as functions of $z\in\mathbb C^\pm$ (we remind
the reader that at this moment we work under the hypotheses that $\mathbb X$ is invertible in 
$\tilde{\mathcal A}$ and $\tau(\mathbb X^2),\tau(\mathbb A^2)<+\infty$; these hypotheses
will be dropped successively in Proposition \ref{sans} and Theorem \ref{fixp}).

Since we have shown that
$(\tau\otimes{\rm id}_{M_2(\mathbb C)})\left[(ze_{11}-\mathbf{A-X})^{-1}\right]$ is invertible, let
us write also a formula for its inverse for future reference. As we have shown above that
$\Im(\tau\otimes{\rm id}_{M_2(\mathbb C)})\left[(ze_{11}-\mathbf{A-X})^{-1}\right]<0$ in
$M_2(\mathbb C)$ for all $z\in\mathbb C^+$, its inverse must have strictly positive imaginary part on
$\mathbb C^+$. This allows us to use Schur's complement formula. In order to save space, we
denote
\begin{eqnarray}
R & = & (z-\mathbb B-\mathbb{AXA})^{-1},\text{ and}\\
\boldsymbol\rho & = & \left(\tau(\mathbb X+\mathbb{XA}R\mathbb{AX})-\frac{\tau(\mathbb{XA}R)
\tau(R\mathbb{AX})}{\tau(R)}\right)^{-1}.
\end{eqnarray}
Then
\begin{equation}\label{SchurG}
(\tau\otimes{\rm tr}_2)\left[(ze_{11}-\mathbf{A-X})^{-1}\right]^{-1}=
\begin{bmatrix}
\frac{1}{\tau(R)}+\frac{\tau(R\mathbb{AX})\boldsymbol\rho\tau(\mathbb{XA}R)}{\tau(R)^2} &
\frac{\tau(R\mathbb{AX})\boldsymbol\rho}{\tau(R)} \\
\frac{\boldsymbol\rho\tau(\mathbb{XA}R)}{\tau(R)} & \boldsymbol\rho
\end{bmatrix}.
\end{equation}
The Schur complement formula coupled with $\Im(\tau\otimes{\rm id}_{M_2(\mathbb C)})
\left[(ze_{11}-\mathbf{A-X})^{-1}\right]<0$ tells us that $\Im\boldsymbol\rho>0$ and $\Im
\left(\frac{1}{\tau(R)}+\frac{\tau(R\mathbb{AX})\boldsymbol\rho\tau(\mathbb{XA}R)}{\tau(R)^2}
\right)\ge\Im z>0$ on $\mathbb C^+$. Recalling that $R=R(z)$, and hence $\boldsymbol\rho=
\boldsymbol\rho(z)$ as well, we investigate the asymptotic behavior at infinity of $\boldsymbol{\rho}$.
If $\tau(\mathbb X)=0$ (recall that $\mathbb X\in L^2(\mathcal A,\tau)$, so $\tau(\mathbb X)$ is
well-defined and real), then
\begin{eqnarray*}
\sphericalangle\lim_{z\to +\infty}\frac{\boldsymbol\rho(z)}{z}
& = & \sphericalangle\lim_{z\to +\infty}\left(z\tau(\mathbb{XA}R(z)\mathbb{AX})-\frac{z
\tau(\mathbb{XA}R(z))z\tau(R(z)\mathbb{AX})}{z\tau(R(z))}\right)^{-1}\\
& = & \frac{1}{\tau(\mathbb X^2)\tau(\mathbb A^2)-\tau(\mathbb X)^2\tau(\mathbb A)^2}
=\frac{1}{\tau(\mathbb X^2)\tau(\mathbb A^2)}.
\end{eqnarray*}
If $\tau(\mathbb X)\neq0$, then
$$
\sphericalangle\lim_{z\to +\infty}\frac{\boldsymbol\rho}{z}=\sphericalangle\lim_{z\to +\infty}
\left(z\tau(\mathbb X)+z\tau(\mathbb{XA}R\mathbb{AX})-\frac{z\tau(\mathbb{XA}R)z\tau(R
\mathbb{AX})}{z\tau(R)}\right)^{-1}=\frac{1}{+\infty}=0.
$$
It follows that
\begin{eqnarray}
\lefteqn{\sphericalangle\lim_{z\to +\infty}\frac{(\tau\otimes{\rm tr}_2)\left[(ze_{11}-\mathbf{A-X})^{-1}
\right]^{-1}}{z}}\\
& = & \sphericalangle\lim_{z\to +\infty}\frac{1}{z}\begin{bmatrix} \frac{1}{\tau(R)}+\frac{\tau(R
\mathbb{AX})\boldsymbol\rho\tau(\mathbb{XA}R)}{\tau(R)^2} & \frac{\tau(R
\mathbb{AX})\boldsymbol\rho}{\tau(R)} \\  \frac{\boldsymbol\rho\tau(\mathbb{XA}R)}{\tau(R)} &
\boldsymbol\rho \end{bmatrix}\nonumber\\
& = & \sphericalangle\lim_{z\to +\infty}\begin{bmatrix} \frac{1}{z\tau(R)}+\frac{z\tau(R\mathbb{AX})z
\tau(\mathbb{XA}R)}{z^2\tau(R)^2}\frac{\boldsymbol\rho}{z} & \frac{z\tau(R
\mathbb{AX})}{z\tau(R)}\frac{\boldsymbol\rho}{z} \\  \frac{\boldsymbol\rho}{z}\frac{z
\tau(\mathbb{XA}R)}{z\tau(R)} & \frac{\boldsymbol\rho}{z} \end{bmatrix}\nonumber\\
& = &\left\{
\begin{array}{ccc}
\begin{bmatrix} 1 & 0 \\ 0 & 0 \end{bmatrix} & \text{if} & \tau(\mathbb X)\neq0,\\
\ & \ & \ \\
\begin{bmatrix} 1 & 0 \\ 0 & \frac{1}{\tau(\mathbb A^2)\tau(\mathbb X^2)} \end{bmatrix} & \text{if}
& \tau(\mathbb X)=0.
\end{array}\right.\nonumber
\end{eqnarray}
From the relation
\begin{equation}\label{R}
(\tau\otimes{\rm tr}_2)\left[(ze_{11}-\mathbf{A-X})^{-1}\right]^{-1}+ze_{11}
=\alpha(ze_{11})+\omega(ze_{11})
\end{equation}
together with
$\Im\alpha(ze_{11})\ge\Im ze_{11},\Im\omega(ze_{11})\ge\Im ze_{11}$, we deduce that the $(1,1)$
entries of the matrices $\alpha(ze_{11}),\omega(ze_{11})$ are Nevanlinna maps whose
Julia-Carath\'eodory derivative at infinity equals one.

Let us perform some brute force computations in order to determine more explicit relations between
$\alpha,\omega,$ and the entries of $(\tau\otimes{\rm tr}_2)\left[(ze_{11}-\mathbf{A-X})^{-1}\right]$.
For simplicity, we denote the entries of the two-by-two matrices $\alpha(ze_{1,1})$ and
$\omega(ze_{1,1})$ just by $\alpha_{ij}$ and $\omega_{ij}$, respectively, suppressing the variable $z$
(for example, $\alpha(ze_{1,1})_{1,2}$ is denoted $\alpha_{12}$):
\begin{eqnarray*}
\lefteqn{E_{\bf X}\left[(ze_{1,1}-\mathbf{A-X})^{-1}\right]}\\
& \!\!= & \!\!\begin{bmatrix} E_{\mathbb X}\left[(z-\mathbb B-\mathbb{AXA})^{-1}\right] &
- E_{\mathbb X}\left[(z-\mathbb B-\mathbb{AXA})^{-1}\mathbb{AX}\right] \\
-E_{\mathbb X}\left[\mathbb{XA}(z-\mathbb B-\mathbb{AXA})^{-1}\right] &  E_{\mathbb X}
\left[\mathbb X+\mathbb{XA}(z-\mathbb B-\mathbb{AXA})^{-1}\mathbb{AX}\right]\end{bmatrix}\\
& \!\!= &\!\!
\begin{bmatrix}\omega_{11}&\omega_{12}\\\omega_{21}&\omega_{22}+\mathbb X^{-1}
\end{bmatrix}^{-1}\\
& \!\!= & \!\!
\begin{bmatrix} R_\omega & -R_\omega\omega_{12}(\omega_{22}+\mathbb X^{-1})^{-1} \\
-(\omega_{22}\!+\!\mathbb X^{-1})^{-1}\!\omega_{21}R_\omega
 & \!(\omega_{22}\!+\!\mathbb X^{-1})^{-1}\!+\!(\omega_{22}\!+\!\mathbb X^{-1})^{-1}\!
\omega_{21}R_\omega\omega_{12}(\omega_{22}\!+\!\mathbb X^{-1})^{-1}\end{bmatrix},
\end{eqnarray*}
where $R_\omega=(\omega_{11}-\omega_{12}(\omega_{22}+\mathbb X^{-1})^{-1}\omega_{21})^{-1}$.
Its expression\footnote{It is clear that $\omega_{22}+\mathbb X^{-1}$ is invertible: if
$\Im \omega_{22}(ze_{11})>0$ for some, and hence all, $z\in\mathbb C^+$, then that is obvious.
If that is not the case, then $\omega_{22}$ is a real constant.  Since $\Im\omega(ze_{11})\ge\Im
ze_{11}\ge0$, having $\omega_{22}\in\mathbb R$ would force
$\omega_{12}(ze_{11})=\overline{\omega_{21}(ze_{11})}$ for all $z\in\mathbb C^+$. This, coupled
with the analyticity of $z\mapsto\omega(ze_{11})$, forces $\omega_{12}=\overline{\omega_{21}}$
to be a constant as well. If this constant were zero, then we'd get the obvious
contradiction $0=\Im[(\omega(ze_{1,1})-\mathbf X)^{-1}]_{2,2}=
\Im E_{\mathbb X}\left[\mathbb X+\mathbb{XA}(z-\mathbb B-\mathbb{AXA})^{-1}\mathbb{AX}
\right]\neq0$. If not, using the Schur complement formula again, this time with respect to $\omega_{11}$,
the $(2,2)$ entry is $(\omega_{22}+\mathbb X^{-1}-\omega_{21}\omega_{11}(z)^{-1}\omega_{12})^{-1}$.
Since $\omega_{11}(iy)/iy\to1$ as $y\to+\infty$, it follows that (by letting $z\to +\infty$
nontangentially and using Remark \ref{condiR}) $|\omega_{12}|^2(\omega_{22}+\mathbb X^{-1})^{-2}=
\tau(\mathbb A^2)\mathbb X^2$, which yields $\omega_{22}=0$, and thus, by the invertibility
hypothesis on $\mathbb X$, that $\omega_{22}+\mathbb X^{-1}=\mathbb X^{-1}$ is invertible as
an unbounded operator in $\tilde{\mathcal A}$, as needed.} can be arranged more conveniently:
\begin{eqnarray}
R_\omega&=&(\omega_{11}-\omega_{12}(\omega_{22}+\mathbb X^{-1})^{-1}\omega_{21})^{-1}\\
& = & (\omega_{11}-\omega_{12}(\omega_{22}\mathbb X+1)^{-1}\omega_{21}\mathbb X)^{-1}\\
& = & \left((\omega_{22}\mathbb X+1)\omega_{11}-\omega_{12}\omega_{21}\mathbb X\right)^{-1}
(\omega_{22}\mathbb X+1)\\
& = &\left(\omega_{11}+\det\omega\mathbb X\right)^{-1}+\left(\omega_{11}+\det\omega\mathbb X
\right)^{-1}\omega_{22}\mathbb X\\
& = & \frac{\omega_{22}}{\det\omega}+\left(1-\frac{\omega_{11}\omega_{22}}{\det\omega}\right)
\left(\omega_{11}+\det\omega\mathbb X\right)^{-1}\\
& = & \frac{\omega_{22}}{\det\omega}-\frac{\omega_{12}\omega_{21}}{\det\omega}
\left(\omega_{11}+\det\omega\mathbb X\right)^{-1}\\
& = & \frac{\omega_{22}}{\det\omega}-\frac{\omega_{12}\omega_{21}}{(\det\omega)^2}
\left(\frac{\omega_{11}}{\det\omega}+\mathbb X\right)^{-1}.
\end{eqnarray}
In the first four equalities above, no assumption is necessary. In the fifth and later, one makes the
implicit assumption that $\det\omega\neq0$. However, that must happen: if it were that $\det\omega=
\det\omega(ze_{1,1})=0$, then $\omega_{11}\omega_{22}=\omega_{12}\omega_{21}$. Replacing in
the expression of $R_\omega$, we obtain
$$
R_\omega=(\omega_{11}-\omega_{11}(\omega_{22}+\mathbb X^{-1})^{-1}
\omega_{22})^{-1}=\frac{1}{\omega_{11}}\left(1-\left(1+(\omega_{22}\mathbb X)^{-1}\right)^{-1}\right)^{-1}
=\frac{1+\omega_{22}\mathbb X}{\omega_{11}}.
$$
The $(2,2)$ entry of $E_{\bf X}\left[(ze_{1,1}-\mathbf{A-X})^{-1}\right]$ is then
\begin{eqnarray*}
\lefteqn{E_{\mathbb X}\!
\left[\mathbb X+\mathbb{XA}(z-\mathbb B-\mathbb{AXA})^{-1}\mathbb{AX}\right]}\\
& = & (\omega_{22}\!+\!\mathbb X^{-1})^{-1}\!+\!(\omega_{22}\!+\!\mathbb X^{-1})^{-1}\!
\omega_{21}R_\omega\omega_{12}(\omega_{22}\!+\!\mathbb X^{-1})^{-1}\\
& = & (\omega_{22}\!+\!\mathbb X^{-1})^{-1}\!+\!(\omega_{22}\!+\!\mathbb X^{-1})^{-1}
\omega_{11}\frac{1+\omega_{22}\mathbb X}{\omega_{11}}\omega_{22}(1\!+\!\mathbb X\omega_{22})^{-1}\mathbb X\\
& = & (\omega_{22}\!+\!\mathbb X^{-1})^{-1}\left(1+\omega_{22}\mathbb X\right)=\mathbb X,
\end{eqnarray*}
forcing $E_{\mathbb X}\left[\mathbb{XA}(z-\mathbb B-\mathbb{AXA})^{-1}
\mathbb{AX}\right]=0$. If this happens for one arbitrary $z\in\mathbb C^+$, it happens for all $z\in
\mathbb C^+$. Multiplying by $z$ and taking limit at infinity along the imaginary axis yields
\begin{align*}
0=zE_{\mathbb X}&\left[\mathbb{XA}(z-\mathbb B-\mathbb{AXA})^{-1}\mathbb{AX}\right]=\\
&\lim_{z\to +\infty} zE_{\mathbb X}\left[\mathbb{XA}(z-\mathbb B-\mathbb{AXA})^{-1}\mathbb{AX}
\right]=E_{\mathbb X}\left[\mathbb{XA}\mathbb{AX}\right]=\tau(\mathbb A^2)
\mathbb X^2\neq0,
\end{align*}
a contradiction. So $\det\omega\neq0$. Of course,
\begin{eqnarray}
-R_\omega\omega_{12}(\omega_{22}+\mathbb X^{-1})^{-1} & = &-\left(\omega_{11}+\det\omega
\mathbb X\right)^{-1}\omega_{12}\mathbb X\\
& = & -\frac{\omega_{12}}{\det\omega}+\frac{\omega_{11}\omega_{12}}{\det\omega}
\left(\omega_{11}+\det\omega\mathbb X\right)^{-1}\\
& = & -\frac{\omega_{12}}{\det\omega}+\frac{\omega_{12}\omega_{11}}{(\det\omega)^2}
\left(\frac{\omega_{11}}{\det\omega}+\mathbb X\right)^{-1},
\end{eqnarray}
\begin{eqnarray}
-(\omega_{22}+\mathbb X^{-1})^{-1}\omega_{21}R_\omega
& = & -\frac{\omega_{21}}{\det\omega}+\frac{\omega_{11}\omega_{21}}{\det\omega}\left(\omega_{11}+\det\omega\mathbb X\right)^{-1}\\
& = & -\frac{\omega_{21}}{\det\omega}+\frac{\omega_{21}\omega_{11}}{(\det\omega)^2}\left(\frac{\omega_{11}}{\det\omega}+\mathbb X\right)^{-1},
\end{eqnarray}
\begin{eqnarray}
\lefteqn{\!(\omega_{22}\!+\!\mathbb X^{-1})^{-1}\!+\!(\omega_{22}\!+\!\mathbb X^{-1})^{-1}\!\omega_{21}
R_\omega\omega_{12}(\omega_{22}\!+\!\mathbb X^{-1})^{-1}}\quad\quad\quad\quad\\
&\quad\quad\quad\quad\quad = & \left(\omega_{22}+\mathbb X^{-1}-\frac{\omega_{12}\omega_{21}}{\omega_{11}}\right)^{-1}\\
&\quad\quad\quad\quad\quad = &  \left(1+\frac{\det\omega}{\omega_{11}}\mathbb X\right)^{-1}\mathbb X\\
&\quad\quad\quad\quad\quad = & \frac{\omega_{11}}{\det\omega}-\frac{\omega_{11}}{\det\omega}\left(1+\frac{\det\omega}{\omega_{11}}\mathbb X\right)^{-1}\\
&\quad\quad\quad\quad\quad = & \frac{\omega_{11}}{\det\omega}-\left(\frac{\omega_{11}}{\det\omega}\right)^2\left(\frac{\omega_{11}}{\det\omega}+\mathbb X\right)^{-1}.
\end{eqnarray}
Thus, in particular,
$$
\begin{bmatrix} \omega_{11} & \omega_{12} \\ \omega_{21} & \omega_{22}+\!\mathbb X^{-1} \end{bmatrix}^{-1}\!\!
=\omega(ze_{11})^{-1}\!-\begin{bmatrix} \omega_{12} \\ -\omega_{11} \end{bmatrix}
\begin{array}{cc}\begin{bmatrix}\frac{1}{(\det\omega)^{2}}\left(\frac{\omega_{11}}{\det\omega}+\mathbb X\right)^{-1}\end{bmatrix} & \!\!\begin{bmatrix} \omega_{21} & -\omega_{11} \end{bmatrix}\\
\ & \
\end{array}.
$$
Before going forward, note that since $\Im\omega(ze_{11})\ge0,$
one has $\Im(\omega(ze_{11})^{-1})\le0,$ so that $\Im\frac{\omega_{jj}(ze_{11})}{\det\omega(ze_{11})}\le0,j=1,2.$
Since $\omega(\overline{z}e_{11})=\omega(ze_{11})^*$, it follows that
$\omega_{12}(\overline{z}e_{11})=\left[\omega(ze_{11})^*\right]_{12}=\overline{\omega_{21}(ze_{11})}$, or
$\omega_{12}(ze_{11})=\overline{\omega_{21}(\overline{z}e_{11})}$. Finally (as before, the limits below involving operators are in the so topology),
\begin{eqnarray*}
\tau(\mathbb A^2)\mathbb X^2
& = & \sphericalangle\lim_{z\to +\infty}z\left(E_{\mathbb X}\left[\mathbb X+\mathbb{XA}(z-\mathbb B-\mathbb{AXA})^{-1}\mathbb{AX}\right]-\mathbb X\right)\\
& = & \sphericalangle\lim_{z\to +\infty}z\left(\left(1+\frac{\det\omega}{\omega_{11}}\mathbb X\right)^{-1}\mathbb X-\mathbb X\right)\\
& = & \sphericalangle\lim_{z\to +\infty}z\left(1+\frac{\det\omega}{\omega_{11}}\mathbb X\right)^{-1}\left(1-1-\frac{\det\omega}{\omega_{11}}\mathbb X\right)\mathbb X\\
& = & -\sphericalangle\lim_{z\to +\infty}z\left(1+\frac{\det\omega}{\omega_{11}}\mathbb X\right)^{-1}\frac{\det\omega}{\omega_{11}}\mathbb X^2,
\end{eqnarray*}
which tells us that
$$
\sphericalangle\lim_{z\to +\infty}z\left(\frac{\omega_{11}}{\det\omega}+\mathbb X\right)^{-1}=-\tau(\mathbb A^2)\cdot1,\quad
\sphericalangle\lim_{z\to +\infty}\frac{z\det\omega(ze_{11})}{\omega_{11}(ze_{11})}=-\tau(\mathbb A^2),
$$
so that (since $z\mapsto\omega_{11}(ze_{11})$ is a Nevanlinna map with derivative one at infinity)
\begin{equation}
\sphericalangle\lim_{z\to +\infty}\det\omega(ze_{11})=-\tau(\mathbb A^2).
\end{equation}
It follows that
\begin{eqnarray*}
-\tau(\mathbb A)\mathbb X
& = & \sphericalangle\lim_{z\to +\infty}z\left(-E_{\mathbb X}\left[(z-\mathbb B-\mathbb{AXA})^{-1}\mathbb{AX}\right]\right)\\
& = & \sphericalangle\lim_{z\to +\infty}z\frac{\omega_{12}}{\det\omega}\left(-1+\frac{\omega_{11}}{\det\omega}\left(\frac{\omega_{11}}{\det\omega}+\mathbb X\right)^{-1}\right)\\
& = & \sphericalangle\lim_{z\to +\infty}z\frac{\omega_{12}}{\det\omega}(-\mathbb X)\left(\frac{\omega_{11}}{\det\omega}+\mathbb X\right)^{-1},
\end{eqnarray*}
so that
\begin{equation}
\sphericalangle\lim_{z\to +\infty}\frac{\omega_{12}(ze_{11})}{\det\omega(ze_{11})}=-\frac{\tau(\mathbb A)}{\tau(\mathbb A^2)}
=\sphericalangle\lim_{z\to +\infty}\frac{\omega_{21}(ze_{11})}{\det\omega(ze_{11})}.
\end{equation}
In particular,
\begin{equation}
\sphericalangle\lim_{z\to +\infty}\omega_{12}(ze_{11})=\tau(\mathbb A).
\end{equation}
This  allows us to identify the behavior of $\omega_{22}$:
\begin{eqnarray*}
1& = & \sphericalangle\lim_{z\to +\infty}zE_{\mathbb X}\left[(z-\mathbb B-\mathbb{AXA})^{-1}\right]\\
& = & \sphericalangle\lim_{z\to +\infty}z\left(\frac{\omega_{22}}{\det\omega}-\frac{\omega_{12}\omega_{21}}{(\det\omega)^2}\left(\frac{\omega_{11}}{\det\omega}+\mathbb X\right)^{-1}\right)\\
& = & \sphericalangle\lim_{z\to +\infty}z\frac{\omega_{22}}{\det\omega}-\left(-\frac{\tau(\mathbb A)}{\tau(\mathbb A^2)}\right)^2(-\tau(\mathbb A^2))\\
& = & \frac{\tau(\mathbb A)^2}{\tau(\mathbb A^2)}+\sphericalangle\lim_{z\to +\infty}z\frac{\omega_{22}}{\det\omega},
\end{eqnarray*}
or
$$
\sphericalangle\lim_{z\to +\infty}z\frac{\omega_{22}(ze_{11})}{\det\omega(ze_{11})}=1-\frac{\tau(\mathbb A)^2}{\tau(\mathbb A^2)}
=\frac{\tau(\mathbb A^2)-\tau(\mathbb A)^2}{\tau(\mathbb A^2)}.
$$
This shows that
\begin{equation}
\sphericalangle\lim_{z\to +\infty}z\omega_{22}(ze_{11})=\tau(\mathbb A)^2-\tau(\mathbb A^2)<0,
\end{equation}
so that $z\mapsto\omega_{22}(ze_{11})$
is minus the Cauchy transform of a positive Borel measure on $\mathbb R$ of total mass $\tau(\mathbb A^2)-\tau(\mathbb A)^2$, the variance of $\mathbb A$.

A consequence of these calculations is
\begin{proposition}\label{sans}
Assume that $\mathbb{X,A,B}$ are selfadjoint elements affiliated with $\mathcal A$ such that neither of $\mathbb{X,A}$ is a multiple of the identity of $\mathcal A$ and $\tau(\mathbb
A^2),\tau(\mathbb X^2)<+\infty$. If $\mathbb X$ is *-free from $(\mathbb{A,B})$, then there exists an analytic map $\mathsf f\colon\mathbb C^+\to\mathbb C^+$ such that
\begin{equation}\label{sff}
E_\mathbb X\left[\mathbb A(z-\mathbb B-\mathbb{AXA})^{-1}\mathbb A\right]=\left(\mathsf f(z)-\mathbb X\right)^{-1},\quad z\in\mathbb C.
\end{equation}
Moreover, $\sphericalangle\lim_{z\to +\infty}\frac{\mathsf f(z)}{z}=\frac{1}{\tau(\mathbb A^2)}\in(0,+\infty)$. If $\tau$ is a normal faithful tracial state, then this
result holds with no other restrictions on the operators involved; otherwise, the result holds for bounded operators.
\end{proposition}
For invertible bounded $\mathbb A$ and bounded $\mathbb{B,X}$, \eqref{sff} follows from Voiculescu's \cite{FreeMarkov} applied (with the notations from Voiculescu's
article) to $a=\mathbb A^{-1}(z-\mathbb B)\mathbb A^{-1}$ and $c=-\mathbb X$.
\begin{proof}
If $\mathbb X$ is invertible, then one has $\mathsf f=-\frac{\omega_{11}}{\det\omega}$ according to the previous calculations.
For general $\mathbb X=\mathbb X^*\in L^2(\mathcal A,\tau)$, one approximates $\mathbb X$ with a sequence of invertible operators
$\mathbb X_n$, all being affiliated with the von Neumann algebra generated by $\mathbb X$. Such a sequence can be obtained as before, by, for instance, choosing
$\{s_n\}_{n\in\mathbb N}\subset(0,+\infty)$ such that $s_n\searrow0$ as $n\to +\infty$ and $s_n,-\frac{s_n}{2}$ are not eigenvalues of $\mathbb X$, and defining $\mathbb X_n=
k_n(\mathbb X)$ in the sense of Borel functional calculus with $k_n\colon\mathbb R\to\mathbb R\setminus[-s_n,s_n)$, $k_n(t)=t$ if $|t|\ge s_n$ and $k_n(t)=\frac{t-3s_n}{4}$ if
$|t|<s_n$. Then $\mathbb X-\mathbb X_n$ is a bounded positive operator, with $0\leq\mathbb X-\mathbb X_n<2s_n\cdot1\to0$ as $n\to +\infty$. Pick an arbitrary $M>0$ and let
$\mathbb A_M=K_M(\mathbb A)$, where $K_M(t)=t$ if $|t|\le M$, $K_M(t)=-M$ if $t<-M$, $K_M(t)=M$ if $t>M$. We can write
\begin{eqnarray*}
\lefteqn{\mathbb A_M(z-\mathbb B-\mathbb{AX}_n\mathbb{A})^{-1}\mathbb A_M-\mathbb A_M(z-\mathbb B-\mathbb{AX}\mathbb{A})^{-1}\mathbb A_M}\\
& = & \mathbb A_M\left[(z-\mathbb B-\mathbb{AX}_n\mathbb{A})^{-1}-(z-\mathbb B-\mathbb{AX}\mathbb{A})^{-1}\right]\mathbb A_M\\
& = & \mathbb A_M\left[(z-\mathbb B-\mathbb{AX}_n\mathbb{A})^{-1}(\mathbb{AX}\mathbb{A}-\mathbb{AX}_n\mathbb{A})(z-\mathbb B-\mathbb{AX}\mathbb{A})^{-1}\right]\mathbb A_M\\
& = & \mathbb A_M(z-\mathbb B-\mathbb{AX}_n\mathbb{A})^{-1}\mathbb{A}(\mathbb{X}-\mathbb{X}_n)\mathbb{A}(z-\mathbb B-\mathbb{AX}\mathbb{A})^{-1}\mathbb A_M.
\end{eqnarray*}
Applying $\tau$ to the above yields
\begin{eqnarray*}
\lefteqn{\tau\!\left(\!\mathbb A_M(z-\mathbb B-\mathbb{AX}_n\mathbb{A})^{-1}\mathbb{A}(\mathbb{X}-\mathbb{X}_n)\mathbb{A}(z-\mathbb B-\mathbb{AX}\mathbb{A})^{-1}
\mathbb A_M\!\right)}\\
&\!\!\!\! = &\!\!\! \tau\!\left(\![\mathbb{A}(\mathbb{X}-\mathbb{X}_n)\mathbb{A}]^\frac12(z-\mathbb B-\mathbb{AX}\mathbb{A})^{-1}\!
\mathbb A_M^2(z-\mathbb B-\mathbb{AX}_n\mathbb{A})^{-1}[\mathbb{A}(\mathbb{X}-\mathbb{X}_n)\mathbb{A}]^\frac12\!\right)\\
&\!\!\!\! = &\!\!\! \tau\!\left(\![\mathbb{A}(\mathbb{X}-\mathbb{X}_n)\mathbb{A}]^\frac12\Re\!\left(\!(z-\mathbb B-\mathbb{AX}\mathbb{A})^{-1}\!
\mathbb A_M^2(z-\mathbb B-\mathbb{AX}_n\mathbb{A})^{-1}\!\right)\![\mathbb{A}(\mathbb{X}-\mathbb{X}_n)\mathbb{A}]^\frac12\!\right)\\
&\!\!\!\! & \!\!\!\mbox{}+i\tau\!\left(\![\mathbb{A}(\mathbb{X}-\mathbb{X}_n)\mathbb{A}]^\frac12\Im\!\left(\!(z-\mathbb B-\mathbb{AX}\mathbb{A})^{-1}\!
\mathbb A_M^2(z-\mathbb B-\mathbb{AX}_n\mathbb{A})^{-1}\!\right)\![\mathbb{A}(\mathbb{X}-\mathbb{X}_n)\mathbb{A}]^\frac12\!\right)\!.
\end{eqnarray*}
Since $\left\|(z-\mathbb B-\mathbb{AX}\mathbb{A})^{-1}\mathbb A_M^2(z-\mathbb B-\mathbb{AX}_n\mathbb{A})^{-1}\right\|\le\frac{M^2}{|\Im z|^2}$,
the same holds for the real and imaginary part of this operator. Thus,
\begin{eqnarray}
\lefteqn{\left|\tau\!\left(\mathbb A_M(z-\mathbb B-\mathbb{AX}_n\mathbb{A})^{-1}\mathbb{A}(\mathbb{X}-\mathbb{X}_n)\mathbb{A}(z-\mathbb B-\mathbb{AX}\mathbb{A})^{-1}
\mathbb A_M\right)\right|}\label{em}\\
&\quad\quad\quad\leq & 2\frac{M^2}{|\Im z|^2}\tau\!\left(\mathbb{A}(\mathbb{X}-\mathbb{X}_n)\mathbb{A}\right)\leq 4\frac{M^2}{|\Im z|^2}s_n\tau\left(\mathbb A^2\right).
\nonumber
\end{eqnarray}
Now, $\mathbb A(\mathbb A-\mathbb A_M)$ is the continuous functional calculus of $\mathbb A$ with the function
$$
q_M\colon\mathbb R\ni t\mapsto\left\{\begin{array}{cc}
t(t+M) & \text{ if }t<-M\\
0 & \text{ if } |t|\le M\\
t(t-M) & \text{ if }t>M
\end{array}\right..
$$
Clearly, this is a continuous, nonnegative function, tending pointwise to zero as $M\to+\infty$. Thus, $\mathbb A(\mathbb A-\mathbb A_M)\ge0$ in $\mathcal A$.
Via the argument already used above, one may thus write $|\tau((\mathbb A-\mathbb A_M)(z-\mathbb B-\mathbb{AX}_n\mathbb{A})^{-1}\mathbb A)|=
|\tau(\mathbb A(\mathbb A-\mathbb A_M)(z-\mathbb B-\mathbb{AX}_n\mathbb{A})^{-1})|=
|\tau(\sqrt{\mathbb A(\mathbb A-\mathbb A_M)}(z-\mathbb B-\mathbb{AX}_n\mathbb{A})^{-1}\sqrt{\mathbb A(\mathbb A-\mathbb A_M)})|\le
\frac{2}{|\Im z|}\tau(\mathbb A(\mathbb A-\mathbb A_M))$. With $\mu_\mathbb A$ denoting the distribution of $\mathbb A$ with respect to $\tau$, one has
$$
\tau(\mathbb A(\mathbb A-\mathbb A_M))=\int_\mathbb Rq_M(t)\,{\rm d}\mu_\mathbb A(t).
$$
Since $\mathbb A\in L^2(\mathcal A,\tau)$ and $q_M(t)\le t^2$ for all $t\in\mathbb R$, dominated convergence tells us that $\tau(\mathbb A(\mathbb A-\mathbb A_M))\to0$
as $M\to+\infty$. Since $\mathbb A_M(\mathbb A-\mathbb A_M)$ is given by a function of the form
$$
\mathbb R\ni t\mapsto\left\{\begin{array}{cc}
-M(t+M) & \text{ if }t<-M\\
0 & \text{ if } |t|\le M\\
M(t-M) & \text{ if }t>M
\end{array}\right.,
$$
one also has that $\tau(\mathbb A_M(\mathbb A-\mathbb A_M))\to0$ as $M\to+\infty$. Writing $\tau(\mathbb A_M(z-\mathbb B-\mathbb{AX}_n\mathbb{A})^{-1}\mathbb A_M)
-\tau(\mathbb A(z-\mathbb B-\mathbb{AX}_n\mathbb{A})^{-1}\mathbb A)=\tau((\mathbb A_M-\mathbb A)(z-\mathbb B-\mathbb{AX}_n\mathbb{A})^{-1}\mathbb A_M)
+\tau(\mathbb A(z-\mathbb B-\mathbb{AX}_n\mathbb{A})^{-1}(\mathbb A_M-\mathbb A))$, we obtain
$$
\lim_{M\to+\infty}\tau(\mathbb A_M(z-\mathbb B-\mathbb{AX}_n\mathbb{A})^{-1}\mathbb A_M)=\tau(\mathbb A(z-\mathbb B-\mathbb{AX}_n\mathbb{A})^{-1}\mathbb A),
$$
uniformly in $n$. In particular,
$$
\lim_{M\to+\infty}\tau(\mathbb A_M(z-\mathbb B-\mathbb{AX}\mathbb{A})^{-1}\mathbb A_M)=\tau(\mathbb A(z-\mathbb B-\mathbb{AX}\mathbb{A})^{-1}\mathbb A)
$$
as well.

For a $z$ in a compact subset of $\mathbb C^+$, write
\begin{eqnarray}
\lefteqn{\left|\tau\left(\mathbb A(z-\mathbb B-\mathbb{AX}_n\mathbb{A})^{-1}\mathbb{A}\right)-\tau\left(\mathbb{A}(z-\mathbb B-\mathbb{AXA})^{-1}\mathbb A\right)\right|}
\nonumber\\
& \leq & \left|\tau\left(\mathbb A(z-\mathbb B-\mathbb{AX}_n\mathbb{A})^{-1}\mathbb{A}\right)-\tau\left(\mathbb A_M(z-\mathbb B-\mathbb{AX}_n\mathbb{A})^{-1}
\mathbb{A}_M\right)\right|\label{uno}\\
& & \mbox{}+\left|\tau\left(\mathbb A_M(z-\mathbb B-\mathbb{AX}_n\mathbb{A})^{-1}\mathbb{A}_M\right)-
\tau\left(\mathbb{A}_M(z-\mathbb B-\mathbb{AXA})^{-1}\mathbb A_M\right)\right|\label{altro}\\
& & \mbox{}+\left|\tau\left(\mathbb{A}_M(z-\mathbb B-\mathbb{AXA})^{-1}\mathbb A_M\right)-\tau\left(\mathbb{A}(z-\mathbb B-\mathbb{AXA})^{-1}\mathbb A\right)\right|.
\label{huh}
\end{eqnarray}
If an $\varepsilon>0$ is given then, by the dominated convergence argument above,
there exists an $M_\varepsilon>0$ such that for all $M\ge M_\varepsilon$,
expressions \eqref{uno} and \eqref{huh} are less than $\frac\varepsilon3$
for all $n\in\mathbb N$ and $z$ in the chosen compact set. For such a fixed $M\ge M_\varepsilon$,
we choose $n_0\in\mathbb N$ such that $\varepsilon>12\frac{M^2}{|\Im z|^2}s_n\tau\left(\mathbb A^2\right)$
for all $n\ge n_0$ and $z$ in the chosen compact set. Thanks to
\eqref{em}, we have proved that for any $\varepsilon>0$ there exists an $n_\varepsilon\in\mathbb N$ such that
$$
\left|\tau\left(\mathbb A(z-\mathbb B-\mathbb{AX}_n\mathbb{A})^{-1}\mathbb{A}\right)
-\tau\left(\mathbb{A}(z-\mathbb B-\mathbb{AXA})^{-1}\mathbb A\right)\right|<\varepsilon
$$
for all $n\ge n_\varepsilon$ and $z$ in the chosen compact set.

The existence of the function $\mathsf f_n$ such that the relation
$\tau\!\left(\mathbb A(z-\mathbb B-\mathbb{AX}_n\mathbb{A})^{-1}\mathbb{A}\right)\!=
\tau\left((\mathsf f_n(z)-\mathbb X_n)^{-1}\right)$ holds has been proven.
The convergence of the left hand side to $\tau\left(\mathbb{A}(z-\mathbb B-\mathbb{AXA})^{-1}\mathbb A\right)$
and the convergence of $w\mapsto\tau((w-\mathbb X_n)^{-1})$ right hand side to $g\colon
w\mapsto\tau\left((w-\mathbb X)^{-1}\right)$ as $n\to +\infty$ on compact subsets of $\mathbb C^+$
has been shown already. Then, by the arguments present, for instance,
in \cite{BV93}, there exists a truncated Stolz angle at infinity on which the
functions $g_n\colon w\mapsto\tau\left((w-\mathbb X_n)^{-1}\right)$ are all invertible, and their
inverses are defined on the image of possibly smaller truncated  Stolz angle at zero via
$z\mapsto\tau\left(\mathbb A(z-\mathbb B-\mathbb{AX}_n\mathbb{A})^{-1}\mathbb{A}\right)$
for all $n\in\mathbb N$. For all $z$ in this truncated Stolz angle, we find
$$
\mathsf f(z)=g^{\langle -1\rangle}\left(\tau\left(\mathbb{A}(z-\mathbb B-\mathbb{AXA})^{-1}\mathbb A\right)\right)
$$
$$
=\lim_{n\to +\infty}g^{\langle -1\rangle}_n
\left(\tau\!\left(\mathbb A(z-\mathbb B-\mathbb{AX}_n\mathbb{A})^{-1}\mathbb{A}\right)\right)=\lim_{n\to +\infty}{\sf f}_n(z).
$$
A normal families argument allows us to conclude that
$$
\lim_{n\to +\infty}{\sf f}_n(z)=\mathsf f(z)\text{ and }\tau\left(\mathbb{A}(z-\mathbb B-\mathbb{AXA})^{-1}\mathbb A\right)=
\tau\left((\mathsf f(z)-\mathbb X)^{-1}\right),
$$
holds for all $z\in\mathbb C^\pm$, the convergence being uniform on compact subsets of $\mathbb C^\pm$.
This is \eqref{sff} after the application of the trace $\tau$. Writing $E_{\mathbb X_n}
\left[\mathbb A(z-\mathbb B-\mathbb{AX}_n\mathbb{A})^{-1}\mathbb A\right]=\left(\mathsf f_n(z)-\mathbb X_n\right)^{-1}$,
we note that we have proved the (norm-) convergence, uniform on compact subsets of $\mathbb C^+$, of the right-hand side
to the operator $\left(\mathsf f(z)-\mathbb X\right)^{-1}$.
Since $\mathbb X_n=k_n(\mathbb X)$ and $k_n$ has a compositional inverse given by
$k_n^{\langle-1\rangle}\colon(-\infty,\frac{-s_n}{2})\cup[s_n,+\infty)\to\mathbb R$,
$k_n^{\langle-1\rangle}(s)=s$ if $|s|\ge s_n$, $k_n^{\langle-1\rangle}(s)=4s+3s_n$ if $s\in(-s_n,-\frac{s_n}{2})$,
it follows that the von Neumann algebras ({\em not} the $C^*$-algebras!) generated by $\mathbb X$ and $\mathbb X_n$
are isomorphic:  $\{\mathbb X_n\}''=\{\mathbb X\}''$. Thus, $E_{\mathbb X_n}=E_{\mathbb X}$. The above can then be
written as $E_{\mathbb X}\left[\mathbb A(z-\mathbb B-\mathbb{AX}_n\mathbb{A})^{-1}\mathbb A\right]=\left(\mathsf f_n(z)-
\mathbb X_n\right)^{-1}$. Thanks to the continuity of $E_\mathbb X$, this concludes the proof of our proposition.
\end{proof}

Let us keep in mind that the requirement $\mathbb X\in L^2(\mathcal A,\tau)$ came up in the proof of the above proposition only indirectly, through the existence of
$-\frac{\omega_{11}}{\det\omega}$.

We return to studying the matrix-valued subordination functions, and focus now on $\alpha$, with the previously introduced notations.
\begin{eqnarray*}
\lefteqn{E_{\bf A}\left[(ze_{1,1}-\mathbf{A-X})^{-1}\right]}\\
& \!\!= & \!\!\begin{bmatrix} E_{(\mathbb{A,B})}\left[(z-\mathbb B-\mathbb{AXA})^{-1}\right] &
- E_{(\mathbb{A,B})}\left[(z-\mathbb B-\mathbb{AXA})^{-1}\mathbb{AX}\right] \\
-E_{(\mathbb{A,B})}\left[\mathbb{XA}(z-\mathbb B-\mathbb{AXA})^{-1}\right] &  E_{(\mathbb{A,B})}\left[\mathbb X+
\mathbb{XA}(z-\mathbb B-\mathbb{AXA})^{-1}\mathbb{AX}\right]\end{bmatrix}\\
& \!\!= & \!\!\begin{bmatrix} E_{(\mathbb{A,B})}\left[(z-\mathbb B-\mathbb{AXA})^{-1}\right] &
- E_{(\mathbb{A,B})}\left[(z-\mathbb B-\mathbb{AXA})^{-1}\mathbb{AX}\right] \\
-E_{(\mathbb{A,B})}\left[\mathbb{XA}(z-\mathbb B-\mathbb{AXA})^{-1}\right] &  \tau(\mathbb X)+E_{(\mathbb{A,B})}
\left[\mathbb{XA}(z-\mathbb B-\mathbb{AXA})^{-1}\mathbb{AX}\right]\end{bmatrix}\\
& \!\!= &\!\! \begin{bmatrix} \alpha_{11}-\mathbb B & \alpha_{12}+\mathbb A \\
\alpha_{21}+\mathbb A & \alpha_{22} \end{bmatrix}^{-1}\\
& \!\!= & \!\!\begin{bmatrix} R_\alpha & -R_\alpha(\alpha_{12}+\mathbb A)\alpha_{22}^{-1} \\
-\alpha_{22}^{-1}\!(\alpha_{21}+\mathbb A)R_\alpha &
\alpha_{22}^{-1}+\alpha_{22}^{-1}(\alpha_{21}+\mathbb A)R_\alpha(\alpha_{12}+\mathbb A)\alpha_{22}^{-1}\end{bmatrix},
\end{eqnarray*}
where\footnote{We should probably mention here briefly why $\alpha_{22}\neq0$: otherwise, one has
$\begin{bmatrix} \alpha_{11}-\mathbb B & \alpha_{12}+\mathbb A \\ \alpha_{21}+\mathbb A & 0 \end{bmatrix}^{-1}=
\begin{bmatrix} 0 & (\alpha_{21}+\mathbb A)^{-1} \\ 
(\alpha_{12}+\mathbb A)^{-1} & (\alpha_{12}+\mathbb A)^{-1}(\mathbb B-\alpha_{11})(\alpha_{21}+\mathbb A)^{-1}
\end{bmatrix}$,
which tells us that $ E_{(\mathbb{A,B})}\left[(z-\mathbb B-\mathbb{AXA})^{-1}\right]=0,$ an obvious contradiction.}  again
$R_\alpha=\left(\alpha_{11}-\mathbb B-(\alpha_{12}+\mathbb A)\alpha_{22}^{-1}(\alpha_{21}+\mathbb A)\right)^{-1}$.
As for $\omega$, one has $\Im\alpha(z\allowbreak e_{11})\ge\Im ze_{11}$, so that $\Im \alpha_{11}(ze_{11})\ge\Im z$, $\Im
\alpha_{22}(ze_{11})\ge0$. As already proven,
$\sphericalangle\lim_{z\to +\infty}\frac{\alpha_{11}(ze_{11})}{z}=1$.

Under the more restrictive hypotheses of Proposition \ref{rmkBig} (no restriction on $\mathbb{B,X}$, but $\mathbb{A}$ bounded),
the $(1,1)$ entry tells us that
$$
(z-\mathbb B+\mathbb Af(z)\mathbb A)^{-1}=R_\alpha
=\left(\alpha_{11}-\mathbb B-(\alpha_{12}+\mathbb A)\alpha_{22}^{-1}(\alpha_{21}+\mathbb A)\right)^{-1},
$$
or
$$
z-\mathbb B+\mathbb Af(z)\mathbb A=
\alpha_{11}(z)-\mathbb B-(\alpha_{12}(z)+\mathbb A)\alpha_{22}(z)^{-1}(\alpha_{21}(z)+\mathbb A),\quad z\in\mathbb C^+.
$$
Expanding in the right hand side,
$$
z+f(z)\mathbb A^2=\alpha_{11}(z)-\frac{\alpha_{12}(z)\alpha_{21}(z)}{\alpha_{22}(z)}-\frac{\alpha_{12}(z)+\alpha_{21}(z)}{\alpha_{22}(z)}\mathbb A
-\frac{\mathbb A^2}{\alpha_{22}(z)}.
$$
If $\sigma(\mathbb A)$ contains three or more points, then, by applying the continuous functional calculus for bounded normal
operators in the above, for any $t\in\sigma(\mathbb A)\subseteq\mathbb R$, one has
$$
z-\left(\alpha_{11}(z)-\frac{\alpha_{12}(z)\alpha_{21}(z)}{\alpha_{22}(z)}\right)+
\frac{\alpha_{12}(z)+\alpha_{21}(z)}{\alpha_{22}(z)}t+\left(f(z)+\frac{1}{\alpha_{22}(z)}\right)t^2=0,
$$
for all $z\in\mathbb C^+$. This forces
\begin{eqnarray*}
\alpha_{11}(z)-\frac{\alpha_{12}(z)\alpha_{21}(z)}{\alpha_{22}(z)}& = & z,\\
\frac{\alpha_{12}(z)+\alpha_{21}(z)}{\alpha_{22}(z)}& = & 0,\\
f(z)=-\frac{1}{\alpha_{22}(z)}.
\end{eqnarray*}
If $\#\sigma(\mathbb A)=2$, i.e. $\mu_\mathbb A$ is a Bernoulli distribution, an approximation argument yields the same result
(recall that we have excluded $\#\sigma(\mathbb A)=1$ by hypothesis). Second relation simply means $\alpha_{12}=-\alpha_{21}$,
and the first together with the third tell us that $f(z)\det\alpha(ze_{11})=-z$; in particular, $\alpha(ze_{11})$ is confirmed again to be
invertible on $\mathbb C^\pm$. The second relation turns out to be superfluous:
consider the $(1,2)$ entry, multiplied to the right by $\mathbb A$.
\begin{eqnarray*}
\left(z-\mathbb B+f(z)\mathbb A^2\right)^{-1}(\alpha_{12}(z)+\mathbb A)f(z)\mathbb A
& = & -R_\alpha(\alpha_{12}(z)+\mathbb A)\alpha_{22}(z)^{-1}\mathbb A\\
& = & - E_{(\mathbb{A,B})}\left[(z-\mathbb B-\mathbb{AXA})^{-1}\mathbb{AX}\right]\mathbb A\\
& = & - E_{(\mathbb{A,B})}\left[(z-\mathbb B-\mathbb{AXA})^{-1}\mathbb{AX}\mathbb A\right]\\
& = & E_{(\mathbb{A,B})}\left[1-(z-\mathbb B-\mathbb{AXA})^{-1}(z-\mathbb B)\right]\\
& = & 1-\left(z-\mathbb B+f(z)\mathbb A^2\right)^{-1}(z-\mathbb B).
\end{eqnarray*}
We have only used Proposition \ref{rmkBig} and $E_{(\mathbb{A,B})}$ being a bimodule map. Multiply to the left in the above
by $z-\mathbb B+f(z)\mathbb A^2$ to obtain $\alpha_{12}(z)f(z)\mathbb A+f(z)\mathbb A^2=
z-\mathbb B+f(z)\mathbb A^2-z+\mathbb B$, which forces $\alpha_{12}(z)=\alpha_{21}(z)=0$ for all $z$.
Together with the first relation, this also imposes $\alpha_{11}(z)=z,z\in\mathbb C^+$. Thus,
\begin{equation}
\alpha(ze_{11})=\begin{bmatrix} z & 0 \\ 0 & -\frac{1}{f(z)} \end{bmatrix},\quad z\in\mathbb C^+,
\end{equation}
where $f$ is the map provided by Proposition \ref{rmkBig}. In particular, under the hypothesis of this proposition, we may record
the forms of all the entries of $E_{\bf A}\left[(ze_{1,1}-\mathbf{A-X})^{-1}\right]$:
\begin{eqnarray*}
E_{(\mathbb{A,B})}\left[(z-\mathbb B-\mathbb{AXA})^{-1}\right] & = & \left(z-\mathbb B+f(z)\mathbb A^2\right)^{-1},\\
E_{(\mathbb{A,B})}\left[(z-\mathbb B-\mathbb{AXA})^{-1}\mathbb{AX}\right] & = & -\left(z-\mathbb B+f(z)\mathbb A^2\right)^{-1}\mathbb Af(z),\\
E_{(\mathbb{A,B})}\left[\mathbb{XA}(z-\mathbb B-\mathbb{AXA})^{-1}\right] & = & - f(z)\mathbb A\left(z-\mathbb B+f(z)\mathbb A^2\right)^{-1}\\
E_{(\mathbb{A,B})}\left[\mathbb X+\mathbb{XA}(z-\mathbb B-\mathbb{AXA})^{-1}\mathbb{AX}\right]& = & f(z)\mathbb A\left(z-\mathbb B+f(z)\mathbb A^2\right)^{-1}\!\!
\mathbb Af(z)-f(z)\\
& = & \left(-\frac{1}{f(z)}-\mathbb A(z-\mathbb B)^{-1}\mathbb A\right)^{-1}.
\end{eqnarray*}
While the formulas above have been proven under the assumptions of Proposition \ref{rmkBig}, the existence of the analytic
functions $\alpha_{ij}$, $1\le i,j\le2$, on $\mathbb C^\pm$, has been shown under the assumption that
$\mathbb{A,B,X}$ are selfadjoint, $\tau(\mathbb A^2+\mathbb X^2)<+\infty$, and that $\mathbb X$ is free from
$(\mathbb{A,B})$ with respect to $\tau$, but also the more restrictive assumption that $\mathbb X$ is invertible.

Now we rewrite \eqref{R} to account for the entries of the matrices involved (recall that $\boldsymbol\rho=\boldsymbol\rho(z)=
\left(\tau(\mathbb X)+\tau(\mathbb{XA}R\mathbb{AX})-\frac{\tau(\mathbb{XA}R)\tau(R\mathbb{AX})}{\tau(R)}\right)^{-1}$
and the hypotheses are $\tau(\mathbb X^2)<+\infty$, $\mathbb X$ invertible in $\tilde{\mathcal A}$, and $\mathbb A$ bounded).
$$
\begin{bmatrix} \frac{1}{\tau(R)}+\frac{\tau(R\mathbb{AX})\boldsymbol\rho\tau(\mathbb{XA}R)}{\tau(R)^2} &
\frac{\tau(R\mathbb{AX})\boldsymbol\rho}{\tau(R)} \\
\frac{\boldsymbol\rho\tau(\mathbb{XA}R)}{\tau(R)} & \boldsymbol\rho \end{bmatrix}
+\begin{bmatrix} z & 0 \\ 0 & 0 \end{bmatrix}
=\begin{bmatrix}
z & 0 \\
0 & -\frac{1}{f(z)}
\end{bmatrix}
+\begin{bmatrix}
\omega_{11}(z) & \omega_{12}(z) \\
\omega_{21}(z) & \omega_{22}(z)
\end{bmatrix}.
$$
This provides the straightforward equalities
$$
\frac{\tau(R(z)\mathbb{AX})\boldsymbol\rho(z)}{\tau(R(z))}=\omega_{12}(z),\quad
\frac{\boldsymbol\rho(z)\tau(\mathbb{XA}R(z))}{\tau(R(z))}=\omega_{21}(z).
$$
Since $\tau(R(z)\mathbb{AX})=\tau\left(\left(\mathbb B-z-f(z)\mathbb A^2\right)^{-1}\mathbb Af(z)\right)
=\tau\left(f(z)\mathbb A\left(\mathbb B-z-f(z)\mathbb A^2\right)^{-1}\right)=\tau(\mathbb{XA}R(z))$ due to the
traciality of $\tau$, we obtain here also that $\omega_{12}(z)=\omega_{21}(z)$. Since $\omega_{12}(z)=
\overline{\omega_{21}(\overline{z})}$, we conclude also that $\overline{\omega_{12}(z)}=\omega_{12}(\overline{z})$.

More relations between entries can be established. It might be more convenient to write
\begin{eqnarray*}
\lefteqn{\begin{bmatrix}
\tau(R(z)) & -\tau(R(z)\mathbb{AX}) \\
-\tau(\mathbb{XA}R(z)) & \tau(\mathbb X+\mathbb{XA}R(z)\mathbb{AX})
\end{bmatrix}}\\
& = & \begin{bmatrix}
\omega_{11}(z) & \omega_{12}(z) \\
\omega_{12}(z) & \omega_{22}(z)-\frac{1}{f(z)}
\end{bmatrix}^{-1}\\
& = & \frac{1}{\omega_{11}(z)\omega_{22}(z)-\omega_{12}(z)^2-\frac{\omega_{11}(z)}{f(z)}}
\begin{bmatrix} \omega_{22}(z)-\frac{1}{f(z)} & -\omega_{12}(z) \\ -\omega_{12}(z) & \omega_{11}(z) \end{bmatrix}\\
& = & \frac{1}{\det\omega(z)-\frac{\omega_{11}(z)}{f(z)}}\begin{bmatrix} \omega_{22}(z)-\frac{1}{f(z)} & -\omega_{12}(z) \\
-\omega_{12}(z) & \omega_{11}(z) \end{bmatrix}.
\end{eqnarray*}
The lower right corner indicates a connection between $f$ and $\mathsf f$:
\begin{eqnarray*}
\mathsf f(z)^2\tau\left((\mathsf f(z)-\mathbb X)^{-1}\right)-\mathsf f(z)
& =&\frac{\omega_{11}(z)}{\det\omega(z)}-\left(\frac{\omega_{11}(z)}{\det\omega(z)}\right)^2\!
\tau\left(\left(\frac{\omega_{11}(z)}{\det\omega(z)}+\mathbb X\right)^{-1}\right)\\
& = & \tau(\mathbb X+\mathbb{XA}R(z)\mathbb{AX})\\
& = & \frac{\omega_{11}(z)}{\det\omega(z)-\frac{\omega_{11}(z)}{f(z)}}
=\frac{1}{\frac{\det\omega(z)}{\omega_{11}(z)}-\frac{1}{f(z)}}\\
& = & -\frac{1}{\frac{1}{\mathsf f(z)}+\frac{1}{f(z)}}\\
& = & f(z)^2\tau\left(\mathbb A(z-\mathbb B+f(z)\mathbb A^2)^{-1}\mathbb A\right)-f(z).
\end{eqnarray*}
(We have used here the relation through which we had found $\mathsf f$, namely
$\mathsf f(z)=-\frac{\omega_{11}(z)}{\det\omega(z)}$.)
This can also be written in a simpler form:
$$
\tau\left((\mathsf f(z)-\mathbb X)^{-1}\right)
=\frac{1}{\mathsf f(z)+f(z)}=\tau\left(\mathbb A(z-\mathbb B+f(z)\mathbb A^2)^{-1}\mathbb A\right).
$$
This allows us to obtain the pair $(f,\mathsf f)$ as a fixed point:

\begin{thm}\label{fixp}
Let $\mathbb{X,A,B}\in\tilde{\mathcal A}$ be three selfadjoint operators affiliated with the tracial $W^*$-probability
space $(\mathcal A,\tau)$, $\mathbb{A,X}\not\in\mathbb C\cdot1$.
Then there exists a pair of analytic self-maps of $\mathbb C^+$, denoted $(f,\mathsf f)$, such that
$$
\tau\left((\mathsf f(z)-\mathbb X)^{-1}\right)=\frac{1}{\mathsf f(z)+f(z)}
=\tau\left(\mathbb A(z-\mathbb B+f(z)\mathbb A^2)^{-1}\mathbb A\right)\quad z\in\mathbb C^+.
$$
In addition, the point $(f(z),\mathsf f(z))\in\mathbb C^+\times\mathbb C^+$ is the unique attracting fixed point of the map
$$
\mathscr F_z\colon\mathbb C^+\times\mathbb C^+\to\mathbb C^+\times\mathbb C^+,\quad
\mathscr F_z\begin{pmatrix} w_1\\ w_2\end{pmatrix}=
\begin{pmatrix}
\frac{1}{\tau\left((w_2-\mathbb X)^{-1}\right)}-w_2\\
\frac{1}{\tau\left(\mathbb A(z-\mathbb B+w_1\mathbb A^2)^{-1}\mathbb A\right)}-w_1
\end{pmatrix}.
$$
\end{thm}
Note that indeed if $\mathscr F_z\begin{pmatrix} w_1\\ w_2\end{pmatrix}=\begin{pmatrix} w_1\\ w_2\end{pmatrix}$, then
$w_1+w_2=\frac{1}{\tau\left(\mathbb A(z-\mathbb B+w_1\mathbb A^2)^{-1}\mathbb A\right)}
=\frac{1}{\tau\left((w_2-\mathbb X)^{-1}\right)}$. Remarkably, $\mathbb X$ being *-free
from $(\mathbb{A,B})$ is not part of the hypothesis. This is a statement about a pair of Nevanlinna maps.
The theorem has already been proved when $\mathbb X$ is invertible
as an affiliated operator, $\mathbb A$ is bounded, and $\tau(\mathbb X^2)<+\infty.$ Indeed, one can
simply make $\mathbb X$ *-free from $(\mathbb{A,B})$.

\begin{proof}
Let us first note that $\Im\frac{1}{\tau((w_2-\mathbb X)^{-1})}>\Im w_2$ for all $w_2\in\mathbb C^+$,
thanks to our hypothesis that $\mathbb X$ is not a multiple of
the identity. Second,  for any fixed $z\in\mathbb C^+$,
$w_1\mapsto\frac{1}{\tau\left(\mathbb A(z-\mathbb B+w_1\mathbb A^2)^{-1}\mathbb A\right)}$
is a well-defined Nevanlinna map sending $\mathbb C^+$ into itself, while strictly increasing the imaginary
part of the argument. Indeed, first, it is quite clear that this map is well-defined: one can write
$(z-\mathbb B+w_1\mathbb A^2)^{-1}=(\Re z-\mathbb B+\Re w_1\mathbb A^2+i(\Im z+\Im w_1\mathbb A^2))^{-1}=
(\Im z+\Im w_1\mathbb A^2)^{-1/2}
((\Im z+\Im w_1\mathbb A^2)^{-1/2}(\Re z-\mathbb B+\Re w_1\mathbb A^2)(\Im z+\Im w_1\mathbb A^2)^{-1/2}+i)^{-1}
(\Im z+\Im w_1\mathbb A^2)^{-1/2}$.
Applying the trace $\tau$ yields
\begin{eqnarray*}
\lefteqn{\tau\left(\mathbb A(z-\mathbb B+w_1\mathbb A^2)^{-1}\mathbb A\right)}\\
& = & \tau\left(\mathbb A(\Im z+\Im w_1\mathbb A^2)^{-1/2}\left((\Im z+\Im w_1\mathbb A^2)^{-1/2}
(\Re z-\mathbb B+\Re w_1\mathbb A^2)\frac{}{}\right.\right.\\
& & \left.\left.\mbox{}\times(\Im z+\Im w_1\mathbb A^2)^{-1/2}+i\right)^{-1}
(\Im z+\Im w_1\mathbb A^2)^{-1/2}\mathbb A\right)\\
& = & \tau\left((\Im z+\Im w_1\mathbb A^2)^{-1/2}\mathbb A^2(\Im z+\Im w_1\mathbb A^2)^{-1/2}\frac{}{}\right.\\
& & \left.\mbox{}\times\left((\Im z+\Im w_1\mathbb A^2)^{-1/2}(\Re z-\mathbb B+\Re w_1\mathbb A^2)
(\Im z+\Im w_1\mathbb A^2)^{-1/2}+i\right)^{-1}\right).
\end{eqnarray*}
The operator $(\Im z+\Im w_1\mathbb A^2)^{-1/2}\mathbb A^2(\Im z+\Im w_1\mathbb A^2)^{-1/2}$ is positive
and bounded (norm equal to $\frac{1}{|\Im w_1|}$, as it follows from continuous functional calculus). The operator
$(\Im z+\Im w_1\mathbb A^2)^{-1/2}(\Re z-\mathbb B+\Re w_1\mathbb A^2)(\Im z+\Im w_1\mathbb A^2)^{-1/2}$
is densely defined, selfadjoint ($(\mathcal A,\tau)$ is a tracial $W^*$-probability space, so the space of affiliated operators
is a star-algebra). Thus, the operator
$\left((\Im z+\Im w_1\mathbb A^2)^{-1/2}(\Re z-\mathbb B+\Re w_1\mathbb A^2)(\Im z+\Im w_1\mathbb A^2)^{-1/2}
+i\right)^{-1}$ is normal, bounded in norm by one, so $\tau\left(\mathbb A(z-\mathbb B+w_1\mathbb A^2)^{-1}\mathbb A\right)$
is well-defined, bounded in absolute value by $\frac{1}{|\Im w_1|}$. The analyticity of the correspondence in $z$ is obvious (the
derivative in $z$ equals $-\tau\left(\mathbb A(z-\mathbb B+w_1\mathbb A^2)^{-2}\mathbb A\right)$,
obviously well-defined) and the derivative in $w_1$ is
$-\tau\left(\mathbb A(z-\mathbb B+w_1\mathbb A^2)^{-1}\mathbb A^2(z-\mathbb B+w_1\mathbb A^2)^{-1}\mathbb A\right)$,
giving an upper bound of $\frac{1}{|\Im w_1|^2}$, so the correspondence in $w_1$ is analytic as well.
Let us determine the behavior at infinity of $w_1\mapsto\tau\left(\mathbb A(z-\mathbb B+w_1\mathbb A^2)^{-1}\mathbb A\right)$:
\begin{eqnarray*}
\lefteqn{\lim_{v\to+\infty}iv\tau\left(\mathbb A(z-\mathbb B+iv\mathbb A^2)^{-1}\mathbb A\right)}\\
& = & \lim_{v\to+\infty}\tau\left(\mathbb A\left(\frac{z-\mathbb B}{iv}+\mathbb A^2\right)^{-1}\mathbb A\right)\\
& = & \lim_{v\to+\infty}\tau\left(\mathbb A\left(\frac{\Re z-\mathbb B}{iv}+\frac{\Im z}{v}+\mathbb A^2\right)^{-1}\mathbb A\right)\\
& = & \lim_{v\to+\infty}\tau\left(\mathbb A\left(\frac{\Im z}{v}+\mathbb A^2\right)^{-\frac12}\right.\\
& & \left.\mbox{}\times\left(\left(\frac{\Im z}{v}+\mathbb A^2\right)^{-\frac12}
\frac{\Re z-\mathbb B}{iv}\left(\frac{\Im z}{v}+\mathbb A^2\right)^{-\frac12}+1\right)^{-1}
\left(\frac{\Im z}{v}+\mathbb A^2\right)^{-\frac12}\mathbb A\right)\\
& = & \lim_{v\to+\infty}\!\tau\!\left(\!\mathbb A^2\left(\frac{\Im z}{v}+\mathbb A^2\right)^{-1}\!\!\left(\!\left(\frac{\Im z}{v}+
\mathbb A^2\right)^{-\frac12}\!\frac{\Re z-\mathbb B}{iv}\left(\frac{\Im z}{v}+\mathbb A^2\right)^{-\frac12}\!+1\right)^{\!-1}\right)\!.
\end{eqnarray*}
It is well-known (and easy to check, for instance via Borel functional calculus) that
$\left\|\mathbb A^2\left(\frac{\Im z}{v}+\mathbb A^2\right)^{-1}\right\|\le1$ and
$\mathbb A^2\left(\frac{\Im z}{v}+\mathbb A^2\right)^{-1}\to1-P_{\ker(\mathbb A)}$ as
$v\to+\infty$, in the so topology, where $P_{\ker(\mathbb A)}$ is the orthogonal projection onto
the kernel of $\mathbb A$. Similarly, $\|(\Im z+v\mathbb A^2)^\frac{-1}{2}\|\le\frac{1}{\sqrt{\Im z}}$ and
$(\Im z+v\mathbb A^2)^\frac{-1}{2}\to\frac{1}{\sqrt{\Im z}}P_{\ker(\mathbb A)}$ as $v\to+\infty$ in the so topology.
Then $\left(\!\left(\frac{\Im z}{v}+\mathbb A^2\right)^{-\frac12}\!\frac{\Re z-\mathbb B}{iv}
\left(\frac{\Im z}{v}+\mathbb A^2\right)^{-\frac12}\!+1\right)^{\!-1}\to
i\big(P_{\ker(\mathbb A)}\frac{\Re z-\mathbb B}{\Im z}P_{\ker(\mathbb A)}+i\big)^{-1}$ as $v\to+\infty$ (easily seen to
happen in the wo topology). Both families stay bounded in norm, uniformly in $v$ (both by one).
Thus, their product converges as well in the wo topology, and to the product of the limits. It follows that
\begin{equation}
\lim_{v\to+\infty}iv\tau\left(\mathbb A(z-\mathbb B+iv\mathbb A^2)^{-1}\mathbb A\right)=1-\tau(P_{\ker(\mathbb A)}).
\end{equation}
It follows that $\Im\frac{1}{\tau\left(\mathbb A(z-\mathbb B+w_1\mathbb A^2)^{-1}\mathbb A\right)}>
\frac{\Im w_1}{1-\tau(P_{\ker(\mathbb A)})}\ge\Im w_1$, for all $w_1\in\mathbb C^+$ (obviously, we do
use the fact that $\mathbb A\neq0$). This allows us to conclude that $\mathscr F_z$ is well-defined, sending
$\mathbb C^+\times \mathbb C^+$ into itself.

There are several ways now to argue that $\mathscr F_z$ has a unique attracting fixed point inside
$\mathbb C^+\times \mathbb C^+$ (the analyticity of the dependence of this attracting fixed point on $z$ is
a trivial normal families argument, once its existence and uniqueness for an arbitrary $z\in\mathbb C^+$ has been established).
Consider the function
$$
g_z\colon w\mapsto \frac{1}{\tau\left(\left(\frac{1}{\tau\left(\mathbb A\left(z-\mathbb B+w\mathbb A^2\right)^{-1}\mathbb A\right)}-w-\mathbb X\right)^{-1}\right)}-
\frac{1}{\tau\left(\mathbb A(z-\mathbb B+w\mathbb A^2)^{-1}\mathbb A\right)}+w,
$$
for fixed $z\in\mathbb C^+.$ We have seen that it is well-defined, and (since $\mathbb{A,X}$ are not multiples of the identity) that
it maps $\mathbb C^+$ into itself. Moreover, since  $z\in\mathbb C^+$, it extends continuously (in fact smoothly) to $\mathbb R$,
and sends $\mathbb C^+\cup\mathbb R$ into $\mathbb C^+$. This, in particular, guarantees that its Denjoy-Wolff point
(exists and) is not in $\mathbb R$. We claim it cannot be $\infty$ either. To be clear, infinity being the Denjoy-Wolff
point of $g_z$ would mean
that $\lim_{v\to+\infty}g_z(iv)=\infty$ and $\lim_{v\to+\infty}\frac{g_z(iv)}{iv}$ exists and belongs to $[1,+\infty)$.
That is a straightforward, but long, argument: if the nontangential limit at infinity of this function either does not exist or is
finite (i.e. if $\lim_{v\to+\infty}g_z(iv)\neq\infty$), then we're done. As already
seen, if $\ker\mathbb A$ is nontrivial, then $\frac{1}{\tau\left(\mathbb A\left(z-\mathbb B+w\mathbb A^2\right)^{-1}\mathbb A\right)}-w$ tends to infinity as $w\to +\infty$
nontangentially and its Julia-Carath\'eodory derivative is $\frac{\tau(P_{\ker(\mathbb A)})}{1-\tau(P_{\ker(\mathbb A)})}$, so that the Julia-Carath\'eodory derivative
of $g_z$ at infinity is zero, which shows that infinity is not attracting for it. If
$w\mapsto\frac{1}{\tau\left(\mathbb A\left(z-\mathbb B+w\mathbb A^2\right)^{-1}\mathbb A\right)}-w$ does not have a nontangential limit at
infinity or it has a nontangential limit at infinity which is in $\mathbb C^+$, then the nontangential limit at infinity of $g_z$ either is in $\mathbb C^+$
or does not exist. In both cases, $g_z(w)$ does not have infinity as an attracting fixed point. Thus, the only possibility that remains to be discarded is that
 $w\mapsto\frac{1}{\tau\left(\mathbb A\left(z-\mathbb B+w\mathbb A^2\right)^{-1}\mathbb A\right)}-w$ has a finite real nontangential limit at infinity,
$$
\lim_{v\to +\infty}\frac{1}{\tau\left(\mathbb A\left(z-\mathbb B+iv\mathbb A^2\right)^{-1}\mathbb A\right)}-iv=\ell\in\mathbb R,\text{ that }
\ \sphericalangle\lim_{u\to\ell}\tau\left((\mathbb X-u)^{-1}\right)=0,
$$
and
$$
\lim_{v\to+\infty}\frac{1}{iv\tau\left(\left(\frac{1}{\tau\left(\mathbb A\left(z-\mathbb B+iv\mathbb A^2\right)^{-1}\mathbb A\right)}-iv-\!\mathbb X\right)^{-1}\right)}-
\frac{\frac{1}{\tau\left(\mathbb A(z-\mathbb B+iv\mathbb A^2)^{-1}\mathbb A\right)}-iv}{iv}\in[1,+\infty).
$$
Remarkably, only the first condition is needed in order to obtain a contradiction. Let us first see what one can conclude from $
\lim_{v\to +\infty}\frac{1}{\tau\left(\mathbb A\left(z-\mathbb B+iv\mathbb A^2\right)^{-1}\mathbb A\right)}-iv=\ell$:
\begin{eqnarray*}
\mathbb R\ni\ell&=&\lim_{v\to +\infty}\frac{1}{\tau\left(\mathbb A\left(z-\mathbb B+iv\mathbb A^2\right)^{-1}\mathbb A\right)}-iv\\
&=&\lim_{v\to +\infty}\frac{\tau\left((z-\mathbb B)\left(z-\mathbb B+iv\mathbb A^2\right)^{-1}\right)}{\tau\left(\left(z-\mathbb B+iv\mathbb A^2\right)^{-1}\mathbb A^2\right)}\\
&=&\lim_{v\to +\infty}
\frac{\tau\left((z-\mathbb B)\left(\frac{z-\mathbb B}{iv}+\mathbb A^2\right)^{-1}\right)}{\tau\left(\left(\frac{z-\mathbb B}{iv}+\mathbb A^2\right)^{-1}\mathbb A^2\right)}.
\end{eqnarray*}
The limit of the denominator exists and equals one (recall that we have already discarded the case $\ker(\mathbb A)\neq\{0\}$).
Accounting for traciality, one has
\begin{eqnarray*}
\lefteqn{\Im\tau\left((z-\mathbb B)\left(\frac{z-\mathbb B}{iv}+\mathbb A^2\right)^{-1}\right)=\Im\tau\left((x-\mathbb B+iy)\left(\frac{y}{v}+\mathbb A^2-i\frac{x-\mathbb B}{v}
\right)^{-1}\right)}\\
&=&\Im\tau\!\left(\!(x\!-\!\mathbb B+iy)\!\left(\!\mathbb A^2\!+\frac{y}{v}\!\right)^{\!-\frac12}\!\!\left[1\!-\!i\!\left(\!\mathbb A^2\!+\frac{y}{v}\right)^{\!-\frac12}\!
\frac{x\!-\!\mathbb B}{v}\!\left(\!\mathbb A^2\!+\frac{y}{v}\!\right)^{\!-\frac12}\right]^{-1}\!\!\!\left(\!\mathbb A^2\!+\frac{y}{v}\right)^{\!-\frac12}\!\right)\\
&=&\Im\tau\left(\!(x\!-\!\mathbb B+iy)\left[1\!+\!i\!\left(\!\mathbb A^2\!+\frac{y}{v}\right)^{\!-1}\!\frac{x\!-\!\mathbb B}{v}\!\right]
\left[\mathbb A^2\!+\frac{y}{v}\!+\!\frac{x\!-\!\mathbb B}{v}\!\left(\!\mathbb A^2\!+\frac{y}{v}\!\right)^{\!-1}\!\frac{x\!-\!\mathbb B}{v}\!\right]^{-1} \right)\\
&=&\tau\left(\left[(x-\mathbb B)\!\left(\!\mathbb A^2\!+\frac{y}{v}\!\right)^{\!-1}\!\frac{x\!-\!\mathbb B}{v}+y\right]
\left[\mathbb A^2\!+\frac{y}{v}\!+\!\frac{x\!-\!\mathbb B}{v}\!\left(\!\mathbb A^2\!+\frac{y}{v}\!\right)^{\!-1}\!\frac{x\!-\!\mathbb B}{v}\!\right]^{-1}\right).
\end{eqnarray*}
We have used $\left(\mathbb A^2+\frac{y}{v}\right)^{-1}\frac{x-\mathbb B}{v}\left[\mathbb A^2+\frac{y}{v}+\frac{x-\mathbb B}{v}\left(\mathbb A^2+\frac{y}{v}\right)^{-1}
\frac{x-\mathbb B}{v}\right]^{-1}$ being selfadjoint, and the fact that $\Im\tau((a+ib)(c+id))=\tau(ad+bc)$ for $a,b,c,d$ selfadjoint and $\tau$ tracial.
Clearly $(x-\mathbb B)\left(\mathbb A^2+\frac{y}{v}\right)^{-1}\frac{x-\mathbb B}{v}\ge0$, $y>0$, so the only possibility for
the above quantity to go to zero is if
\begin{equation}\label{el}
\lim_{v\to+\infty}\tau\left(\left[\mathbb A^2\!+\frac{y}{v}\!+\!\frac{x\!-\!\mathbb B}{v}\!\left(\!\mathbb A^2\!+\frac{y}{v}
\right)^{\!-1}\!\frac{x\!-\!\mathbb B}{v}\!\right]^{-1}\right)=0,
\end{equation}
$$
\lim_{v\to+\infty}\tau\left((x-\mathbb B)\!\left(\!\mathbb A^2\!+\frac{y}{v}\!\right)^{\!-1}\!\frac{x\!-\!\mathbb B}{v}
\left[\mathbb A^2\!+\frac{y}{v}\!+\!\frac{x\!-\!\mathbb B}{v}\!
\left(\!\mathbb A^2\!+\frac{y}{v}\right)^{\!-1}\!\frac{x\!-\!\mathbb B}{v}\!\right]^{-1}\right)=0.
$$
Clearly, $\frac{y}{v}\!+\!\frac{x\!-\!\mathbb B}{v}\!\left(\!\mathbb A^2\!+\frac{y}{v}\right)^{\!-1}\!\frac{x\!-\!\mathbb B}{v}=
\frac{y}{v}\!+\!(x\!-\!\mathbb B)\!\left(\!(v\mathbb A)^2\!+vy\right)^{\!-1}\!(x\!-\!\mathbb B)$, and
$v<v'\iff(v\mathbb A)^2+vy<(v'\mathbb A)^2+v'y\iff\left(\!(v\mathbb A)^2\!+vy\right)^{\!-1}>
\left(\!(v'\mathbb A)^2\!+v'y\right)^{\!-1}\implies(x-\mathbb B)\left(\!(v\mathbb A)^2\!+vy\right)^{\!-1}(x-\mathbb B)\ge
(x-\mathbb B)\left(\!(v'\mathbb A)^2\!+v'y\right)^{\!-1}(x-\mathbb B)\implies
\frac{x-\mathbb B}{v}\!\left(\!\mathbb A^2\!+\frac{y}{v}\right)^{\!-1}\!\frac{x-\mathbb B}{v}
+\frac{y}{v}>\frac{y}{v'}+\frac{x-\mathbb B}{v'}\!\left(\!\mathbb A^2+\frac{y}{v'}\right)^{\!-1}\!\frac{x-\mathbb B}{v'}\iff
\mathbb A^2\!+\frac{y}{v}+\frac{x\!-\!\mathbb B}{v}\!\left(\!\mathbb A^2\!+\frac{y}{v}\right)^{\!-1}\!\frac{x\!-\!\mathbb B}{v}>
\frac{x\!-\!\mathbb B}{v'}\!\left(\!\mathbb A^2\!+\frac{y}{v'}\right)^{\!-1}\!\frac{x\!-\!\mathbb B}{v'}+\mathbb A^2\!+\frac{y}{v'}
\iff\left[\mathbb A^2\!+\frac{y}{v}\!+\!\frac{x\!-\!\mathbb B}{v}\!\left(\!\mathbb A^2\!+\frac{y}{v}\right)^{\!-1}
\!\frac{x\!-\!\mathbb B}{v}\!\right]^{-1}\!\!<
\left[\mathbb A^2\!+\frac{y}{v'}\!+\!\frac{x\!-\!\mathbb B}{v'}\!\left(\!\mathbb A^2\!+\frac{y}{v'}\right)^{\!-1}\!\frac{x\!-\!\mathbb B}{v'}\!\right]^{-1}$.
Thus, $(0,+\infty)\ni v\mapsto \tau\left(\left[\mathbb A^2\!+\frac{y}{v}\!+\!\frac{x\!-\!\mathbb B}{v}\!\left(\!\mathbb A^2\!+\frac{y}{v}\right)^{\!-1}\!\frac{x\!-\!
\mathbb B}{v}\!\right]^{-1}\right)$ is increasing as a function from $(0,+\infty)$ to itself. But that excludes the possibility
that the limit \eqref{el} is zero, since $\tau$ is faithful and for $v,y>0$ fixed, the operator $\left[\mathbb A^2+\frac{y}{v}+
\frac{x-\mathbb B}{v}\left(\mathbb A^2+\frac{y}{v}\right)^{\!-1}\!\frac{x-\mathbb B}{v}\right]^{-1}$ is bounded and nonzero.
This forces the function
\begin{equation}\label{gez}
g_z\colon
w\mapsto\frac{1}{\tau\left(\left(\frac{1}{\tau\left(\mathbb A\left(z-\mathbb B+w\mathbb A^2\right)^{-1}\mathbb A\right)}
-w-\mathbb X\right)^{-1}\right)}-
\frac{1}{\tau\left(\mathbb A(z-\mathbb B+w\mathbb A^2)^{-1}\mathbb A\right)}+w,
\end{equation}
($z\in\mathbb C^+$ fixed) to have a unique attracting fixed point (the Denjoy-Wolff point) inside
$\mathbb C^+$. This fixed point is simply $f(z)$, and $\mathsf f(z)=\frac{1}{\tau\left(\mathbb A
(z-\mathbb B+f(z)\mathbb A^2)^{-1}\mathbb A\right)}-f(z)$.
A direct verification shows that $\mathscr F_z(f(z),{\sf f}(z))=(f(z),{\sf f}(z))$. The analytic dependence
on $z$ of $f(z)$ has been explained above via a normal families argument, but we should note that it
could also be obtained via the implicit function theorem because the derivative of a function at its
Denjoy-Wolff point in $\mathbb C^+$ is necessarily of absolute value strictly less than one.
\end{proof}

Let us state the main result (and the reason for) this section.
\begin{Corollary}\label{eq}
Assume that $(\mathcal A,\tau)$ is a tracial $W^*$-noncommutative probability space,
$\mathbb{X,A,B}$ are selfadjoint random variables (possibly unbounded) affiliated with $\mathcal A$,
and $\mathbb X$ is *-free from $(\mathbb{A,B})$ with respect to $\tau$. Assume in addition that
neither of $\mathbb{A,X}$ is a scalar multiple of the unit $1\in\mathcal A$.
Then there exist two analytic self-maps $\mathsf f,f\colon\mathbb C^+\to\mathbb C^+$ of the
complex upper half-plane such that
\begin{eqnarray}
f(z) & = & \frac{1}{\tau((\mathsf f(z)-\mathbb X)^{-1})}-\mathsf f(z),\label{49}\\
\mathsf f(z) & = & \frac{1}{\tau\left(\mathbb A(z-\mathbb B+f(z)\mathbb A^2)^{-1}\mathbb A\right)}-f(z),\label{50}\\
\tau\left((\mathsf f(z)-\mathbb X)^{-1}\right) & = & \tau\left(\mathbb A(z-\mathbb B-\mathbb{AXA})^{-1}\mathbb A\right)
\label{51}\\
& = & \tau\left(\mathbb A(z-\mathbb B+f(z)\mathbb A^2)^{-1}\mathbb A\right),\label{52}\\
\tau\left((z-\mathbb B+f(z)\mathbb A^2)^{-1}\right) & = & \tau\left((z-\mathbb B-\mathbb{AXA})^{-1}\right),\label{53}
\end{eqnarray}
for all $z\in\mathbb C^+$. Moreover, if $\tau(\mathbb A^2)<+\infty$, then
$$
\lim_{v\to+\infty}\frac{\mathsf f(iv)}{iv}=\frac{1}{\tau(\mathbb A^2)},\
\lim_{v\to +\infty}\frac{f(iv)}{iv}=0,\quad
\lim_{v\to +\infty}f(iv)=\lim_{y\to+\infty}\frac{1}{\tau\left((iy-\mathbb X)^{-1}\right)}-iy
=-\tau(\mathbb X).
$$
The last equality should be understood in the following sense: if
$\int_\mathbb Rt\,{\rm d}\mu_\mathbb X(t)=m\in\mathbb R$, then $\lim_{v\to+\infty}f(iv)=-m$. If
the integral is not well-defined, then the limit does not exist in $\mathbb R$ (however, it might
exist in $\mathbb C^+$, and then the last equality fails).
\end{Corollary}

\begin{proof}
The existence of the pair $(f,\mathsf f)$ satisfying \eqref{49} and \eqref{50} is Theorem \ref{fixp}. We have proved that
$\mathsf f$ satisfies \eqref{51} in Proposition \ref{sans} under the supplementary assumption that $\tau(\mathbb{A}^2+
\mathbb X^2)<+\infty$. Observe that we may assume without loss of generality that $\mathbb A\ge0$, since otherwise
(thanks to the traciality of $\tau)$ none of the above expressions changes\footnote{This might require a proof only for \eqref{53}:
if $\mathbb A=V\sqrt{\mathbb A^2}$ is the polar decomposition, then $V\sqrt{\mathbb A^2}=\mathbb A=\mathbb A^*=
\sqrt{\mathbb A^2}V^*$ (with $V$ unitary - we are in a tracial $W^*$-probability space), so $\tau\left((z-\mathbb B+
\mathbb{AXA})^{-1}\right)=\tau\left((z-\mathbb B+V\sqrt{\mathbb A^2}\mathbb{X}\sqrt{\mathbb A^2}V^*)^{-1}\right)=
\tau\left((z-V^*\mathbb BV+\sqrt{\mathbb A^2}\mathbb{X}\sqrt{\mathbb A^2})^{-1}\right)$. Replacing $\mathbb B$
with $V^*\mathbb BV$ in the left hand side of \eqref{53} changes nothing, since $\mathbb A^2=V\mathbb A^2V^*$.} when
replacing $\mathbb A$ with $\sqrt{\mathbb A^2}$. Thus, from now on, we assume $\mathbb A\ge0$.
From Proposition \ref{rmkBig}, by multiplying left and right in \eqref{AB} by $\mathbb A$, we also
have that \eqref{52}, \eqref{53} hold for bounded $\mathbb A$ and arbitrary selfadjoints $\mathbb{X,B}$. This
guarantees the existence of an $f$ under these hypotheses. We {\em define} $\mathsf f$ by \eqref{50} and obtain \eqref{49}
(thanks to an application of Theorem \ref{fixp}). Plugging \eqref{52} into \eqref{50} proves the corollary also under the hypotheses of
bounded $\mathbb A$ and arbitrary $\mathbb{B,X}$.

We extend the corollary to arbitrary $\mathbb A\ge0$ as well. To clarify, all that follows below is only necessary for
proving the parts of the corollary involving $\mathbb B+\mathbb{AXA}$, Equations \eqref{49}, \eqref{50}, and
\eqref{52} follow directly from Theorem \ref{fixp}. As before, we approximate $\mathbb A$ with
$\mathbb A_n=\Theta_n(\mathbb A)$, where $\Theta_n(t)=t$ if $0\le t\le n$ and $\Theta_n(t)=n$ if $t>n$
(we do not need $\chi_n$ from the proof of Proposition \ref{rmkBig} because we do not care to have $\{\mathbb A_n\}''
=\{\mathbb A\}''$ here). We apply the approximation in \eqref{51} and \eqref{52} first.
A bound of the form $\big\|\mathbb A_n(z-\mathbb B-\mathbb{A}_n\mathbb{XA}_n)^{-1}\mathbb A_n\big\|
\le\frac{n^2}{|\Im z|}$ is obvious. As noted above, we have proved that $f_n$ exists for $\mathbb A_n$. Since
$w\mapsto(z-\mathbb B+w\mathbb A_n^2)^{-1}$ is bounded by $1/\Im z$ uniformly in $n$ and
\begin{eqnarray*}
\lefteqn{(z-\mathbb B+w\mathbb{A}_m^2)^{-1}-(z-\mathbb B+w\mathbb{A}_n^2)^{-1}}\\
& = & w\big(z-\mathbb B+w\mathbb A^{2}_n\big)^{-1}\big(\mathbb A_n^{2}-\mathbb A_m^{2})
\big(z-\mathbb B+w\mathbb A_m^2\big)^{-1},
\end{eqnarray*}
for all $m,n$, we easily obtain that
\begin{eqnarray*}
\lefteqn{0\le\big(\overline{z}-\mathbb B+\overline{w}\mathbb A^2\big)^{-1}(\mathbb A_n^{2}-\mathbb A^{2})
\big|z-\mathbb B+w\mathbb A^{2}_n\big|^{-2}(\mathbb A_n^{2}-\mathbb A^{2})\big(z-\mathbb B+w\mathbb A^2\big)^{-1}}\\
& \le & \frac{1}{|\Im z|^2}\big(\overline{z}-\mathbb B+\overline{w}\mathbb A^2\big)^{-1}
(\mathbb A_n^{2}-\mathbb A^{2})^{2}\big(z-\mathbb B+w\mathbb A^2\big)^{-1}
\quad\quad\quad\quad\quad\quad\\
& = & \frac{1}{|\Im z|^2}\big(\overline{z}-\mathbb B+\overline{w}\mathbb A^2\big)^{-1}
q_n(\mathbb A_n^{2}-\mathbb A^{2})^{2}q_n\big(z-\mathbb B+w\mathbb A^2\big)^{-1}\\
& \le & \frac{1}{|\Im z|^2}\big(\overline{z}-\mathbb B+\overline{w}\mathbb A^2\big)^{-1}q_n\mathbb A^{4}q_n
\big(z-\mathbb B+w\mathbb A^2\big)^{-1},
\end{eqnarray*}
(recall the previously introduced $q_n=\boldsymbol{1}_{[n,+\infty)}(\mathbb A)$, so that $q_n\searrow0$ in the so topology) and
\begin{eqnarray*}
\lefteqn{\frac{1}{|\Im z|^2}\left\langle\big(\overline{z}-\mathbb B+\overline{w}\mathbb A^2\big)^{-1}q_n\mathbb A^{4}q_n
\big(z-\mathbb B+w\mathbb A^2\big)^{-1}\xi,\xi\right\rangle}\\
& \!= & \frac{1}{|\Im z|^2}\left\langle q_n\mathbb A^{2}q_n\big(z-\mathbb B+w\mathbb A^2\big)^{-1}\xi,q_n\mathbb A^{2}q_n
\big(z-\mathbb B+w\mathbb A^2\big)^{-1}\xi\right\rangle\\
& \!= & \frac{1}{|\Im z|^2}\left\langle q_n\mathbb A^{2}\big(z-\mathbb B+w\mathbb A^2\big)^{-1}\xi,
q_n\mathbb A^{2}\big(z-\mathbb B+w\mathbb A^2\big)^{-1}\xi\right\rangle\to0\text{ as }n\to +\infty,
\end{eqnarray*}
since $q_n\to0$ in the so topology and $\mathbb A^{2}\big(z-\mathbb B+w\mathbb A^2\big)^{-1}$ is a bounded operator
(recall that $\sigma(UV)\cup\{0\}=\sigma(VU)\cup\{0\}$ - one applies this to $U=\mathbb A,V=\mathbb A(z-\mathbb B)^{-1}$
in $\mathbb A^2(z-\mathbb B+w\mathbb A^2)^{-1}=\frac{1}{w}\big[1-(z-\mathbb B)(z-\mathbb B+w\mathbb A^2)^{-1}\big]
=\frac{1}{w}-\frac{1}{w}\big(1+w\mathbb A^2(z-\mathbb B)^{-1}\big)^{-1}=\frac{1}{w}-\frac{1}{w^2}\big(w^{-1}+
\mathbb A^2(z-\mathbb B)^{-1}\big)^{-1}$) and independent of $n$. This shows that $\big(z-\mathbb B+w\mathbb A^2_n
\big)^{-1}\to\big(z-\mathbb B+w\mathbb A^2\big)^{-1}$ in the so topology. On the other hand, functions $f_n$ provided by
Proposition \ref{rmkBig}, coinciding with those $f_n$'s provided by Theorem \ref{fixp}, form a normal family. Picking a convergent
subsequence $\{f_{n_k}\}_k$ makes $\big(z-\mathbb B+f_{n_k}(z)\mathbb A^2_{n_k}\big)^{-1}\to
\big(z-\mathbb B+f(z)\mathbb A^2\big)^{-1}$ in the so topology. By taking limits in \eqref{53} (which, as mentioned above, is
known for $\mathbb A_{n_k}$),
$$
\lim_{k\to +\infty}\tau\left((z-\mathbb B-\mathbb{A}_{n_k}\mathbb{XA}_{n_k})^{-1}\right)
=\tau\big((z-\mathbb B+f(z)\mathbb A^2)^{-1}\big).
$$
This does not tell us yet that $f$ coincides with the function $f$ provided by Theorem \ref{fixp} for the given $\mathbb{A,B,X}$.
To establish the uniqueness of the limit (i.e. the independence of the subsequence considered), and the fact that this limit
coincides with the $f$ associated with $\mathbb{A,B,X}$ from Theorem \ref{fixp}, let us estimate the difference
\begin{eqnarray*}
\lefteqn{\tau\left(\mathbb A(z-\mathbb B+w\mathbb{A}^2)^{-1}\mathbb A\right)-\tau\left(\mathbb A_n(z-\mathbb B+w
\mathbb{A}_n^2)^{-1}\mathbb A_n\right)}\\
& = & \tau\left((z-\mathbb B+w\mathbb{A}^2)^{-1}\mathbb A^2-(z-\mathbb B+w\mathbb{A}_n^2)^{-1}\mathbb A^2_n\right)\\
& = & \tau\left((z-\mathbb B+w\mathbb{A}^2)^{-1}\!\left[\mathbb A^2-\mathbb A_n^2\right]+\!
\left[(z-\mathbb B+w\mathbb{A}^2)^{-1}\!-(z-\mathbb B+w\mathbb{A}_n^2)^{-1}\right]\mathbb A^2_n\right).
\end{eqnarray*}
As we have already noted, $\mathbb A^2\ge\mathbb A_n^2$ and $\mathbb A^2-\mathbb A_n^2=
q_n(\mathbb A^2-\mathbb A_n^2)=(\mathbb A^2-\mathbb A_n^2)q_n$, so that
$(z-\mathbb B+w\mathbb{A}^2)^{-1}\left[\mathbb A^2-\mathbb A_n^2\right]^2(\overline{z}-\mathbb B+\overline{w}\mathbb{A}^2)^{-1}
\leq 4(z-\mathbb B+w\mathbb{A}^2)^{-1}\mathbb A^2q_n\mathbb A^2(\overline{z}-\mathbb B+\overline{w}\mathbb{A}^2)^{-1}$. Precisely
as above, $\|q_n\mathbb A^2(\overline{z}-\mathbb B+\overline{w}\mathbb{A}^2)^{-1}\xi\|_2\to0$ as $n\to +\infty$ for any $\xi\in
L^2(\mathcal A,\tau).$  Thus, $\left[\mathbb A^2-\mathbb A_n^2\right]\!(z-\mathbb B+w\mathbb{A}^2)^{-1}\!\to0$ as $n\to +\infty$ in
the so topology. Since $(z-\mathbb B+w\mathbb{A}^2)^{-1}\!-(z-\mathbb B+w\mathbb{A}_n^2)^{-1}\to0$ in the so topology,
$\tau\left(\left[(z-\mathbb B+w\mathbb{A}^2)^{-1}\!-(z-\mathbb B+w\mathbb{A}_n^2)^{-1}\right]\mathbb A^2_n\right)$ also converges to zero.

Indeed,  $(z-\mathbb B+w\mathbb{A}^2)^{-1}\!-(z-\mathbb B+w\mathbb{A}_n^2)^{-1}=w(z-\mathbb B+w\mathbb{A}^2)^{-1}
(\mathbb A_n^2-\mathbb A^2)(z-\mathbb B+w\mathbb{A}_n^2)^{-1}=-w(z-\mathbb B+w\mathbb{A}^2)^{-1}
q_n(\mathbb A^2-n^2)q_n(z-\mathbb B+w\mathbb{A}_n^2)^{-1}$. Observe also that if $\mathbb A$ (and hence $\mathbb A_n$) is
invertible, then $0\le q_n\mathbb A^{-j}=\mathbb A^{-j}q_n\le\frac{q_n}{n^j}$, $q_n\mathbb A_n^{-j}=\mathbb A_n^{-j}q_n
=\frac{q_n}{n^j}$. This tells us that one can make sense of $q_n\mathbb A^{-j}=\mathbb A^{-j}q_n$ for non-invertible $\mathbb A$
as well (with the same estimate), and that $q_n\mathbb A_n^{-j}=\mathbb A_n^{-j}q_n
=\frac{q_n}{n^j}$ holds as well, regardless of the invertibility of $\mathbb A$. Just in this paragraph, for saving space, let
$R=(z-\mathbb B+w\mathbb{A}^2)^{-1},R_n=(z-\mathbb B+w\mathbb{A}_n^2)^{-1}$. One writes $[\mathbb A_n(R-R_n)\mathbb A_n]^*
[\mathbb A_n(R-R_n)\mathbb A_n]=|w|^2[\mathbb A_nRq_n(\mathbb A^2-n^2)q_nR_n\mathbb A_n]^*[\mathbb A_nRq_n(\mathbb A^2-n^2)q_nR_n\mathbb A_n]
=|w|^2\mathbb A_nR^*_nq_n(\mathbb A^2-n^2)q_nR^*\mathbb A_n^2Rq_n(\mathbb A^2-n^2)R_n\mathbb A_n$. Under the trace, this is the same as
$|w|^2\mathbb A_nRq_n(\mathbb A^2-n^2)R_n\mathbb A_n^2R_n^*q_n(\mathbb A^2-n^2)R^*\mathbb A_n$, and
$|w|^2\mathbb A_nRq_n(\mathbb A^2-n^2)R_n\mathbb A_n^2R_n^*q_n\allowbreak(\mathbb A^2-n^2)R^*\mathbb A_n
\!=|w|^2\!\mathbb A_nRq_n\mathbb A_n^{-1}(\mathbb A^2-n^2)\mathbb A_nR_n\mathbb A_n^2R_n^*\mathbb A_n(\mathbb A^2-n^2)\mathbb A_n^{-1}q_nR^*\mathbb A_n
\!={|w|^2}\!\mathbb A_nRq_n(\mathbb A^2-n^2)\mathbb A_n^{-1}(\mathbb A_nR_n\mathbb A_n)(\mathbb A_nR_n\mathbb A_n)^*\mathbb A_n^{-1}(\mathbb A^2-n^2)
q_nR^*\mathbb A_n
\le|w|^2\left|\frac{1}{|w|}+\frac{1}{|\Im w|}\right|^2\!\mathbb A_nR(\mathbb A^2-n^2)q_n\mathbb A_n^{-2}q_n(\mathbb A^2-n^2)R^*\mathbb A_n
=\left|1+\frac{|w|}{|\Im w|}\right|^2\!\frac{\mathbb A_nR(\mathbb A^2-n^2)q_n(\mathbb A^2-n^2)R^*\mathbb A_n}{n^2}$. Again, under the trace
this is the same as $\left|1+\frac{|w|}{|\Im w|}\right|^2\!\frac{q_n(\mathbb A^2-n^2)R^*\mathbb A_n^2R(\mathbb A^2-n^2)q_n}{n^2}$, and
\begin{eqnarray*}
\left|1+\frac{|w|}{|\Im w|}\right|^2\!\frac{q_n(\mathbb A^2-n^2)R^*\mathbb A_n^2R(\mathbb A^2-n^2)q_n}{n^2}\!
& \le &\left|1+\frac{|w|}{|\Im w|}\right|^2\!\!q_n(\mathbb A^2-n^2)R^*R(\mathbb A^2-n^2)q_n\\
& \le & \left|1+\frac{|w|}{|\Im w|}\right|^2\!\|R(\mathbb A^2-n^2)q_n\|^2q_n\\
& = &  \left|1+\frac{|w|}{|\Im w|}\right|^2\!\|R(\mathbb A^2\!-n^2)q_n(\mathbb A^2\!-n^2)R^*\|q_n\\
& \le & 4\left|1+\frac{|w|}{|\Im w|}\right|^2\!\|R\mathbb A^2\|^2q_n,
\end{eqnarray*}
with the (very rough) majorization $0\le(\mathbb A^2\!-n^2)q_n(\mathbb A^2\!-n^2)\le4\mathbb A^4$, easily obtainable via
continuous functional calculus. We have used in the above the fact that $\mathbb A(z-\mathbb B+w\mathbb A^2)^{-1}\mathbb A=
w^{-1}-w^{-2}\big(w^{-1}+\mathbb A(z-\mathbb B)^{-1}\mathbb A\big)^{-1}$ for any selfadjoints $\mathbb{A,B}$ (bounded or not),
and $\big\|\big(w^{-1}+\mathbb A(z-\mathbb B)^{-1}\mathbb A\big)^{-1}\big\|\le|\Im(w^{-1})|=\frac{|\Im w|}{|w|^2}$.
Similarly, $R\mathbb A^2=(z-\mathbb B+w\mathbb{A}^2)^{-1}\mathbb A^2=\frac{1}{w}[1-(z-\mathbb B+w\mathbb A^2)^{-1}(z-\mathbb B)]=
\frac{1}{w}-\frac{1}{w^2}\big(\frac{1}{w}+(z-\mathbb B)^{-1}\mathbb A^2\big)^{-1}.$ Applying the known fact
that $\sigma(UV)\cup\{0\}=\sigma(VU)\cup\{0\}$ to $U=(z-\mathbb B)\mathbb A,V=\mathbb A$, guarantees that
$\|R\mathbb A^2\|<+\infty$, uniformly in $z\in\mathbb C^+,$ $w$ running in compact subsets of $\mathbb C^+$. Finally, we
have also used the definition of $\mathbb A_n,$ which tells us that $0\le\mathbb A_n\le\|\mathbb A_n\|\le n$. All this tells us that
\begin{eqnarray*}
\lefteqn{\tau\!\left(\mathbb A_n\!\left[(z-\mathbb B+w\mathbb{A}^2)^{-1}\!-(z-\mathbb B+w\mathbb{A}_n^2)^{-1}\right]^*\!\!\mathbb A^2_n\!
\left[(z-\mathbb B+w\mathbb{A}^2)^{-1}\!-(z-\mathbb B+w\mathbb{A}_n^2)^{-1}\right]\!\mathbb A_n\right)}\\
&\quad\quad\quad\quad\quad\quad\quad\quad =  & \tau\left([\mathbb A_n(R-R_n)\mathbb A_n]^*[\mathbb A_n(R-R_n)\mathbb A_n]\right)\\
&\quad\quad\quad\quad\quad\quad\quad\quad \le & \left|1+\frac{|w|}{|\Im w|}\right|^2\tau\left(\!\frac{\mathbb A_nR(\mathbb A^2-n^2)q_n(\mathbb A^2-n^2)R^*\mathbb A_n}{n^2}\right)\\
&\quad\quad\quad\quad\quad\quad\quad\quad = & \left|1+\frac{|w|}{|\Im w|}\right|^2\tau\left(\!\frac{q_n(\mathbb A^2-n^2)R^*\mathbb A_n^2R(\mathbb A^2-n^2)q_n}{n^2}\right)\\
&\quad\quad\quad\quad\quad\quad\quad\quad \le & 4\left|1+\frac{|w|}{|\Im w|}\right|^2\!\|R\mathbb A^2\|^2\tau(q_n)\to0\text{ as }n\to +\infty.\quad\quad\quad\quad\quad\quad\quad\quad\quad\quad\quad
\end{eqnarray*}
By the Cauchy-Schwarz-Buniakowski inequality for states, we conclude that
$\tau([(z-\mathbb B+w\mathbb{A}^2)^{-1}\!-(z-\mathbb B+w\mathbb{A}_n^2)^{-1}]\mathbb A^2_n)=
\tau\left(\mathbb A_n\left[(z-\mathbb B+w\mathbb{A}^2)^{-1}\!-(z-\mathbb B+w\mathbb{A}_n^2)^{-1}\right]\mathbb A_n\right)\to0$
as $n\to +\infty$, at a rate proportional to $\tau(q_n)$, uniformly in $z\in\mathbb C^+$ and $w$ in compact subsets of $\mathbb C^+$.
This completes the argument that
$$
\lim_{n\to +\infty}\tau\left(\mathbb A_n(z-\mathbb B+w\mathbb{A}_n^2)^{-1}\mathbb A_n\right)
=\tau\left(\mathbb A(z-\mathbb B+w\mathbb{A}^2)^{-1}\mathbb A\right).
$$
Thus, if $g_{z,n}$ is obtained from the function $g_z$ in the proof of Theorem \ref{fixp} (see
\eqref{gez}) by replacing $\mathbb A$ with $\mathbb A_n$, then we have proved that
$g_{z,n}\to g_z$ uniformly on compact subsets of $\mathbb C^+$.
As shown by Heins in \cite{Heins}, this implies that the Denjoy-Wolff points $f_n(z)$ of $g_{z,n}$
converge to the Denjoy-Wolff point $f(z)$ of $g_z$. Thanks to \eqref{49}, \eqref{50} (which hold for
arbitrary $\mathbb{A,B,X}$), this shows also that $\mathsf f_n(z)\to\mathsf f(z)$, meaning that the
subsequential limit described in the first part of the proof is independent of the specific subsequence
we choose, and nontrivial (that is, $f(\mathbb C^+)\subseteq\mathbb C^+$ is not a constant). Meaning, we have proved that
$\tau\big((z-\mathbb B-\mathbb{A}_n\mathbb{X A}_n)^{-1}\big)=\tau\big((z-\mathbb B+f_n(z)\mathbb A_n^2)^{-1}\big)
\to\tau\big((z-\mathbb B+f(z)\mathbb A^2)^{-1}\big)$, as claimed.
All we still need to argue is that this limit coincides with $\tau\big((z-\mathbb B-\mathbb A
\mathbb{XA})^{-1}\big)$. That is rather straightforward: for instance, the boundedness of $\|(z-\mathbb B-\mathbb A_n
\mathbb{XA}_n)^{-1}\|$ (by $\epsilon^{-1}$) on $z\in\mathbb C^++i\epsilon$ forces the existence of a wo-convergent subsequence: for each $z$,
$\lim_{k\to +\infty}(z-\mathbb B-\mathbb A_{n_k}\mathbb{XA}_{n_k})^{-1}=V(z)$. Pick\footnote{All this
can be found in the literature, but for the sake of completeness, we give an argument here.} a countable dense subset $\{z_m\colon m\in\mathbb N\}\subset
\mathbb C^+$ and, via a diagonal extraction process, find a possibly further subsequence such that the wo limit
$\lim_{k\to +\infty}(z_m-\mathbb B-\mathbb A_{n_k}\mathbb{XA}_{n_k})^{-1}=V(z_m)$ exists for all $m\in\mathbb N$.
 Any such wo limit depends wo-analytically on $z$: for any wo-continuous positive linear functional $\varphi$ on $\mathcal A$,
the map $z\mapsto\varphi\left((z-\mathbb B-\mathbb A_{n_k}\mathbb{XA}_{n_k})^{-1}\right)$ sends $\mathbb C^+$ into $\mathbb C^-$,
and $\lim_{k\to +\infty}\varphi\left((z_m-\mathbb B-\mathbb A_{n_k}\mathbb{XA}_{n_k})^{-1}\right)=\varphi(V(z_m))$ for all $m\in\mathbb N$.
Thus, thanks to Montel's theorem, there exists a unique analytic map $\mathsf g_\varphi(z)$ on $\mathbb C^+$ with values in
$\mathbb C^-$ such that $\mathsf g_\varphi(z_m)=\varphi(V(z_m))$ on the dense subset $\{z_m\colon m\in\mathbb N\}\subset\mathbb C^+$.
The linearity of the correspondence $\varphi\mapsto\mathsf g_\varphi$ is obvious, so, by the completeness of  a von Neumann algebra in
the wo topology, $z_m\mapsto V(z_m)$ extends to all of $\mathbb C^+$ by wo-continuity, with $\Im V(z)\le0,z\in\mathbb C^+$. 
We claim that $V(z)$ is the resolvent of $\mathbb{AXA}+\mathbb B\in\tilde{\mathcal A}$.
Indeed, note that $\lim_{k\to+\infty}(z-\mathbb B-\mathbb A_{n_k}\mathbb{XA}_{n_k})^{-1}=V(z)$ in the wo topology: if $\xi$ is given,
then
\begin{eqnarray*}
\lefteqn{\left\langle\left[(z-\mathbb B-\mathbb A_{n_k}\mathbb{XA}_{n_k})^{-1}-V(z)\right]\xi,\xi\right\rangle}\\
& = & \left\langle\left[(z-\mathbb B-\mathbb A_{n_k}\mathbb{XA}_{n_k})^{-1}-
(z_j-\mathbb B-\mathbb A_{n_k}\mathbb{XA}_{n_k})^{-1}\right]\xi,\xi\right\rangle\\
& & \mbox{}+\left\langle\left[(z_j-\mathbb B-\mathbb A_{n_k}\mathbb{XA}_{n_k})^{-1}-V(z_j)\right]\xi,\xi\right\rangle+\langle[V(z_j)-V(z)]\xi,\xi\rangle.
\end{eqnarray*}
Given $\varepsilon>0$, the continuity of $V$ in the weak operator topology (shown above) guarantees that there exits a $\delta_1>0$
(depending on $\varepsilon,z,$ and $\xi$)  such that $|\langle[V(z_j)-V(z)]\xi,\xi\rangle|<\frac{\varepsilon}{3}$ for
all $j\in\mathbb N$ such that $|z-z_j|<\delta_1$. For the resolvent, one has
$\|(z-\mathbb B-\mathbb A_{n_k}\mathbb{XA}_{n_k})^{-1}-(z_j-\mathbb B-\mathbb A_{n_k}\mathbb{XA}_{n_k})^{-1}\|=
|z_j-z|\|(z-\mathbb B-\mathbb A_{n_k}\mathbb{XA}_{n_k})^{-1}(z_j-\mathbb B-\mathbb A_{n_k}\mathbb{XA}_{n_k})^{-1}\|\le
\frac{|z_j-z|}{\Im z\Im z_j}$, providing an easy $\delta_2>0$, this one only depending on $\Im z$ and $\varepsilon$, not on $\xi$.
Fix a $z_j$ that works for these two terms. Since $(z_j-\mathbb B-\mathbb A_{n_k}\mathbb{XA}_{n_k})^{-1}\to V(z_j)$
as $k\to+\infty$, there exists a $k_{\varepsilon,j,\xi}\in\mathbb N$ depending on all of ${\varepsilon,j,\xi}$, such that
for all $k\ge k_{\varepsilon,j,\xi}$, one has $\left|\left\langle\left[(z_j-\mathbb B-\mathbb A_{n_k}\mathbb{XA}_{n_k})^{-1}-V(z_j)\right]\xi,\xi\right\rangle\right|<
\frac{\varepsilon}{3}$. Thus, for all $k\ge k_{\varepsilon,j,\xi}$, one has $\left|\left\langle\left[(z-\mathbb B-\mathbb A_{n_k}\mathbb{XA}_{n_k})^{-1}-V(z)\right]\xi,\xi\right\rangle
\right|<\varepsilon$. Having established this limit to hold, let us consider the product $(z-\mathbb B-\mathbb A_{n_k}\mathbb{XA}_{n_k})V(z)$.
By its definition, for any $\xi$ in the domain of $\mathbb A$, one has $\lim_{k\to+\infty}\|(\mathbb A-\mathbb A_{n_k})\xi\|_2=
\lim_{k\to+\infty}\|q_{n_k}(\mathbb A-{n_k})q_{n_k}\xi\|_2=0.$ Since all of $\mathbb{A,X,B}\in\tilde{\mathcal A}$,
there exists a dense subspace of the Hilbert space on which $\mathcal A$ acts such that all of
$\mathbb{A,X,B,AX,XA,AXA,AXA+B}$ have it as a common domain. For a $\xi$ in this dense subspace, one has
that $\mathbb{XA}\xi$ belongs to the domain of $\mathbb A$, so that
$\lim_{k\to+\infty}\|(\mathbb A_{n_k}-\mathbb A)\mathbb{XA}\xi\|_2=0$, and
$\|\mathbb A_{n_k}\mathbb X(\mathbb A_{n_k}-\mathbb A)\xi\|_2^2=
\langle\mathbb A_{n_k}\mathbb X(\mathbb A_{n_k}-\mathbb A)\xi,\mathbb A_{n_k}\mathbb X(\mathbb A_{n_k}-\mathbb A)\xi\rangle
=\langle\mathbb A_{n_k}^2[\mathbb X(\mathbb A_{n_k}-\mathbb A)\xi],[\mathbb X(\mathbb A_{n_k}-\mathbb A)\xi]\rangle
\leq\langle\mathbb A^2[\mathbb X(\mathbb A_{n_k}-\mathbb A)\xi],[\mathbb X(\mathbb A_{n_k}-\mathbb A)\xi]\rangle
=\|(\mathbb{AX})(\mathbb A_{n_k}-\mathbb A)\xi\|_2^2$. Since $(\mathbb A_{n_k}-\mathbb A)\xi$ belongs to the domain of
$\mathbb{AX}$ for all $k$ and $\lim_{k\to +\infty}\|(\mathbb A_{n_k}-\mathbb A)\xi\|_2=0$, we also have
$0=\lim_{k\to +\infty}\|(\mathbb{AX})(\mathbb A_{n_k}-\mathbb A)\xi\|_2\ge\lim_{k\to +\infty}\|(\mathbb{A}_{n_k}\mathbb{X})(\mathbb A_{n_k}-\mathbb A)\xi\|_2\ge0$. Thus,
$$
0\le\lim_{k\to +\infty}\|(\mathbb A_{n_k}\mathbb{XA}_{n_k}-\mathbb{AXA})\xi\|_2\leq
\lim_{k\to +\infty}\|(\mathbb{A}_{n_k}\mathbb{X})(\mathbb A_{n_k}-\mathbb A)\xi\|_2+\|(\mathbb A_{n_k}-\mathbb A)\mathbb{XA}\xi\|_2=0.
$$
We have
\begin{eqnarray*}
\lefteqn{\left\langle(z-\mathbb B-\mathbb A_{n_k}\mathbb{XA}_{n_k})V(z)\xi,\xi\right\rangle}\\
& = & \left\langle(z-\mathbb B-\mathbb A_{n_k}\mathbb{XA}_{n_k})\left[V(z)-(z-\mathbb B-\mathbb A_{n_k}\mathbb{XA}_{n_k})^{-1}\right]\xi,\xi\right\rangle+\langle\xi,\xi\rangle\\
& = & \left\langle\left[V(z)-(z-\mathbb B-\mathbb A_{n_k}\mathbb{XA}_{n_k})^{-1}\right]\xi,(z-\mathbb B)\xi\right\rangle\\
& & \mbox{}-\left\langle\left[V(z)-(z-\mathbb B-\mathbb A_{n_k}\mathbb{XA}_{n_k})^{-1}\right]\xi,\mathbb A_{n_k}\mathbb{XA}_{n_k}\xi\right\rangle+\langle\xi,\xi\rangle\\
& = & \left\langle\left[V(z)-(z-\mathbb B-\mathbb A_{n_k}\mathbb{XA}_{n_k})^{-1}\right]\xi,(z-\mathbb B)\xi\right\rangle\\
& & \mbox{}-\left\langle\left[V(z)-(z-\mathbb B-\mathbb A_{n_k}\mathbb{XA}_{n_k})^{-1}\right]\xi,[\mathbb A_{n_k}\mathbb{XA}_{n_k}-\mathbb{AXA}]\xi\right\rangle\\
& & \mbox{}-\left\langle\left[V(z)-(z-\mathbb B-\mathbb A_{n_k}\mathbb{XA}_{n_k})^{-1}\right]\xi,\mathbb{AXA}\xi\right\rangle+\langle\xi,\xi\rangle.
\end{eqnarray*}
Given that $\xi$ is in the domain of $\mathbb B$, the weak operator convergence of $(z-\mathbb B-
\mathbb A_{n_k}\mathbb{XA}_{n_k})^{-1}$ to $V(z)$ yields $\lim_{k\to +\infty}\left\langle\left[V(z)-(z-
\mathbb B-\mathbb A_{n_k}\mathbb{XA}_{n_k})^{-1}\right]\xi,(z-\mathbb B)\xi\right\rangle=0$.
By the same argument, $\lim_{k\to +\infty}\left\langle\left[V(z)-(z-\mathbb B-\mathbb A_{n_k}
\mathbb{XA}_{n_k})^{-1}\right]\xi,\mathbb{AXA}\xi\right\rangle=0$ as well.
Since $\|V(z)\|\le\frac{1}{\Im z},\|(z-\mathbb B-\mathbb A_{n_k}\mathbb{XA}_{n_k})^{-1}\|\le
\frac{1}{\Im z}$, one has by the Schwartz-Cauchy-Buniakowski inequality
$0\le|\left\langle\left[V(z)-(z-\mathbb B-\mathbb A_{n_k}\mathbb{XA}_{n_k})^{-1}\right]\xi,
[\mathbb A_{n_k}\mathbb{XA}_{n_k}-\mathbb{AXA}]\xi\right\rangle|\le\frac{2}{\Im z}\|\xi\|_2
\|[\mathbb A_{n_k}\mathbb{XA}_{n_k}-\mathbb{AXA}]\xi\|_2\to0$ as $k\to +\infty$, as proven above.
Thus,
$$
\lim_{k\to +\infty}\left\langle(z-\mathbb B-\mathbb A_{n_k}\mathbb{XA}_{n_k})V(z)\xi,\xi\right\rangle=\langle\xi,\xi\rangle.
$$
This tells us that $(z-\mathbb B-\mathbb A_{n_k}\mathbb{XA}_{n_k})V(z)\to1$ in the wo topology.
Since $(z-\mathbb B-\mathbb A_{n_k}\mathbb{XA}_{n_k})\to(z-\mathbb B-\mathbb A\mathbb{XA})$
in the so topology, one has $(z-\mathbb B-\mathbb A\mathbb{XA})V(z)=1$. The proof that $V(z)(z-
\mathbb B-\mathbb A\mathbb{XA})\xi=\xi$ for all $\xi$ in the domain of $\mathbb{AXA}+\mathbb B$
is straightforward. We conclude that $V(z)=(z-\mathbb B-\mathbb A\mathbb{XA})^{-1}$.

%
%

This concludes the proof of \eqref{53}, and hence of \eqref{51}--\eqref{52} as well, and hence of the existence part of the corollary.

It remains to study the asymptotics of $f,\mathsf f$ at infinity under the assumption that $\tau(\mathbb A^2)<+\infty$.
Proposition \ref{sans} shows that $\lim_{y\to+\infty}\frac{\mathsf f(iy)}{iy}=\frac{1}{\tau(\mathbb A^2)}$ under the supplementary assumption that
$\tau(\mathbb X^2)<+\infty$, an assumption that we shall easily remove next. Indeed, recall
that the hypothesys $\mathbb X\in L^2(\mathcal A,\tau)$ in Proposition \ref{sans} is used only
in the proof of the {\em existence} of $\mathsf f$, and the existence of {\sf f} has now been
established in Theorem \ref{fixp} without making this assumption. Thus, at this moment
we can argue directly from \eqref{51} (by multiplying with $z$ on both sides and letting
$z$ go to infinity nontangentially) that $\mathsf f(iv)/iv\to\tau(\mathbb A^2)^{-1}$. We can also
argue directly from Theorem \ref{fixp}: since we know, thanks to the Julia-Carath\'eodory Theorem that
$c:=\lim_{v\to+\infty}f(iv)/iv,k:=\lim_{v\to+\infty}\mathsf f(iv)/iv$ exist and belong to $[0,+\infty)$,
from the functional equations in Theorem \ref{fixp} we obtain by multiplying with the variable and
taking nontangential limit at infinity that, first, $k>0$, second, $c=0$, and, from these two,
$k=\frac{1}{\tau(\mathbb A^2)}$. These three items of information tell us that $\mathsf f(iv)$ tends
to infinity nontangentially as $v\to+\infty$, so that
$$
\lim_{v\to+\infty}f(iv)=\lim_{v\to+\infty}\frac{1}{\tau\left((\mathsf f(iv)-\mathbb X)^{-1}\right)}-
\mathsf f(iv)=\sphericalangle\lim_{w\to +\infty}\frac{1}{\tau\left((w-\mathbb X)^{-1}\right)}-w.
$$
A direct computation shows that if $\tau(|\mathbb X|)<+\infty$, then the right hand side equals
$-\tau(\mathbb X)$. If, say, the distribution of $\mathbb X$  with respect to $\tau$ is standard Cauchy,
then the right hand side equals $i\in\mathbb C$, so $f(iv)\to i$ as $v\to+\infty$,
and $\tau(\mathbb X)$ is not well-defined. This concludes our proof.
\end{proof}

We should mention that the pair $(\mathsf f,f)$ is uniquely detemined via Theorem \ref{fixp}
by the distributions of $\mathbb X$ and of $(\mathbb{A,B})$ with respect to the tracial state $\tau$.
Corollary \ref{eq} tells us then that the distribution of $\mathbb{AXA}+\mathbb B$ with respect
to $\tau$ is uniquely determined by $(\mathsf f,f)$, so that the distribution of $\mathbb{AXA}+
\mathbb B$ with $\mathbb X$ being *-free from $(\mathbb{A,B})$ is determined via a (reasonably
simple) fixed-point functional equation involving the Cauchy transform of the distribution of
$\mathbb X$ and a variation of the Cauchy transform of the affine expression $\mathbb B+
w\mathbb A^2$.

\begin{remark}\label{neg}
If all of $\mathbb{A,B,X}\in\tilde{\mathcal A}$ are positive, then of course $\mathbb B+\mathbb{AXA}
\ge0$ as well. Then all functions in Corollary \ref{eq} are well-defined, analytic, and real-valued on
$(-\infty,0)$. Indeed, this is immediately visible for $z\mapsto\tau\big((z-\mathbb B-
\mathbb{AXA})^{-1}\big)$ and for $z\mapsto\tau\big(\mathbb A(z-\mathbb B-
\mathbb{AXA})^{-1}\mathbb A\big)$. Moreover, both these functions send $(-\infty,0)$
into itself and are decreasing on $(-\infty,0)$. From \eqref{51}, it follows that $\mathsf f$
extends analytically to $(-\infty,0)$ and takes values in $\mathbb R\setminus\sigma(\mathbb X)$.
Combining \eqref{51} and \eqref{49}, it follows that $z\mapsto f(z)$ extends analytically to
$(-\infty,0)$, with real values as well. Since $f\colon\mathbb C^+\to\mathbb C^+$,
it follows (for instance thanks to the Julia-Carath\'eodory Theorem) that $f$ is strictly increasing
on $(-\infty,0)$. Corollary \ref{eq} together with Lindel\"{o}f's Theorem tells us that $\lim_{v
\to-\infty}f(v)=\lim_{w\to-\infty}\frac{1}{\tau((w-\mathbb X)^{-1})}-w\in[-\infty,0)$, meaning that
if $\tau(|\mathbb X|)<+\infty$, then $\lim_{v\to-\infty}f(v)=-\tau(\mathbb X)$ and if $\tau(|\mathbb X|)
=+\infty$, then $\lim_{v\to-\infty}f(v)=-\infty$. {Similarly as before, the corresponding statement of 
Corollary \ref{eq} together with Lindel\"{o}f's Theorem imply that $\lim_{v\to-\infty} \frac{f(v)}{v}=0$.}
\end{remark}

For the purposes of studying the moments of $\mathbb{AXA}+\mathbb B$ (those that
exist), it is convenient to use the moment generating function introduced in Section \ref{sec:tr},
and the function $\delta\colon\mathbb C^+\to\mathbb C^-$, $\delta(z)=-f(z),$ and
$h\colon\mathbb C^+\to\mathbb C^+,h(z)=\mathsf f(z)+f(z)$. In the next Lemma and Example we go back to the notation form the previous Sections and consider the element $\A^{1/2}\X\A^{1/2}+\B$.

\begin{lemma}\label{lem:subordination}

Fix a non-commutative probability space $(\tcA,\tau)$. Let $\X$, $\B$, $\A$ be self-adjoint, where $\A$
is positive. Assume that $\X$ and $(\A,\B)$ are free. Define a function $\delta\colon \mathbb{C}^+\cup
\R^-\to \mathbb{C}^-\cup\R$ as a unique analytic function satisfying a conditions
         \[
         \sphericalangle\lim_{|z| \to+\infty}\frac{\delta(z)}{z}=0.
         \]
    and
\[ 
\psi_{\X}\left(\frac{1}{h(z)+\delta(z)}\right)=\frac{\delta(z)}{h(z)},
\]
where
\[
\frac{1}{h(z)}=\tau\left(\A^{1/2}(z-\B-\delta(z)\A)^{-1}\A^{1/2}\right).
\]
Then
\begin{align}\label{eq:FSZ1}   
\psi_{\A^{1/2}\X\A^{1/2}+\B}(z^{-1})=\tau\left(\left(1-\frac{\delta(z)}{z}\A-\frac{1}{z}\B\right)^{-1}\right)-1.
\end{align}
 \end{lemma}
\begin{proof}[Proof of Lemma \ref{lem:subordination}]
The lemma is just a reformulation of the above remark and of Corollary \ref{eq} in terms of $h,\delta$
and $\psi_\mathbb X$, by simply making use of the relation $\psi_\mathbb X(z)=
\tau\left(z\mathbb X(1-z\mathbb X)^{-1}\right)=\tau\left((1-z\mathbb X)^{-1}\right)-1=
\frac{1}{z}\tau\left(\left(\frac1z-\mathbb X\right)^{-1}\right)-1$. Indeed, the first equation is a
reformulation of \eqref{49}, the second is precisely the second defining equality in Theorem
\ref{fixp}, and the third (namely \eqref{eq:FSZ1}) is \eqref{53}.
\end{proof}

Let us illustrate the behavior of function $\delta$ in the example of free Beta prime distribution.
\begin{example}
We can calculate function $\delta$ explicitly in our running example.
Recall that we have $\X\stackrel{d}{=}\A^{1/2}\X\A^{1/2}+\A$, where $\A$ and $\X$ are freely independent.
By \eqref{eq:FSZ1} we have
\[
\psi_\X(z^{-1}) = \psi_\A\left(\frac{1+\delta(z)}{z} \right).
\]
Thus,
\begin{align*}
        \frac{1+\delta(z)}{z} = \chi_{\A}(\psi_{\X}(z^{-1})),
\end{align*}
with $\chi_\A = \chi_{f\beta'_{a,a+b}}$ and $\psi_\X = \psi_{f\beta'_{a,b}}$, where by \cite[Proof of Proposition 3.2]{Yoshida}, we have
\begin{align*}
\psi_{f\beta'_{a,b}}(z) &= \frac{(b-1)-(1+a)z-\sqrt{((b-1)-(1+a)z)^2-4 az(z+1)}}{2(1+z)},\\
\chi_{f\beta'_{a,b}}(z) &=\frac{z(b-1-z)}{(1+z)(z+a)}.
\end{align*}
We obtain for $z<0$, $a>0$, $b\geq1$,
\[
\delta(z)=\frac{a \left(a-1+2b+(b-1)z + \sqrt{(z(b-1)-(1-a))^2 - 4 a b z }\right)}{2 b (a+b-1)}.
\]
If $b>1$, then $\lim_{z\to-\infty}\delta(z) = a/(b-1)$.
If $b=1$, then
\[
\delta(z)=\frac{a+1+\sqrt{(1-a)^2 - 4 a z }}{2}\quad\mbox{and}\quad \lim_{z\to-\infty} \frac{\delta(z)}{\sqrt{-z}} = \sqrt{a}.
\]
\end{example}

\begin{lemma}\label{lem:deltaProp} \
\begin{enumerate}
    \item[(i)] If $\A, \B, \X\geq 0$, then $\delta\colon (-\infty,0)\to (0,+\infty)$ is monotonic.
    \item[(ii)] If $(\A,\B)\stackrel{d}{=}(\A,-\B)$ and $\X\stackrel{d}{=}-\X$, then $\delta\colon i(-\infty,0)\to i(0,+\infty)$ is monotonic.
\end{enumerate}
\end{lemma}
\begin{proof}
Part (i) has been shown  as part of Remark \ref{neg}. Part (ii) follows from Theorem \ref{fixp}
by observing that $\Re\tau\left(\mathbb A(iy-\mathbb B+iv\mathbb A^2)^{-1}\mathbb A\right)=
\Re\tau\left((iy-\mathbb X)^{-1}\right)=0$ for all $y,v>0$.
\end{proof}

To conclude this section, observe that $\mathbb X\stackrel{d}{=}\mathbb B+\mathbb{AXA}$
means that
\begin{eqnarray*}
\tau\left(\mathbb A(z-\mathbb B+f(z)\mathbb{A}^2)^{-1}\mathbb A\right)
& = & \tau\left(\mathbb A(z-\mathbb B-\mathbb{AXA})^{-1}\mathbb A\right)\\
& = & \tau\left((\mathsf f(z)-\mathbb X)^{-1}\right)\\
& = & \tau\left((\mathsf f(z)-\mathbb B-\mathbb{AXA})^{-1}\right)\\
& = & \tau\left((\mathsf f(z)-\mathbb B+f(\mathsf f(z))\mathbb{A}^2)^{-1}\right),
\quad z\in\mathbb C^+.
\end{eqnarray*}
At the same time, Theorem \ref{fixp} tells us also that $\tau\left((\mathsf f(z)-\mathbb X)^{-1}\right)
=\frac{1}{\mathsf f(z)+f(z)}$. Thus, the functions $(\mathsf f,f)$ corresponding to a free perpetuity
must satisfy
$$
\tau\left(\mathbb A(z-\mathbb B+f(z)\mathbb{A}^2)^{-1}\mathbb A\right)=
\frac{1}{\mathsf f(z)+f(z)}
=\tau\left((\mathsf f(z)-\mathbb B+f(\mathsf f(z))\mathbb{A}^2)^{-1}\right).
$$
Unfortunately, there seems to be no direct way to determine the existence of such a pair
$(\mathsf f,f)$ from these two equalities. They are equivalent to the pair
\begin{eqnarray*}
\mathsf f(z)&=&\frac{1}{\tau\left(\mathbb A(z-\mathbb B+f(z)\mathbb{A}^2)^{-1}\mathbb A\right)}
-f(z),\\
f(z)&=&\frac{1}{\tau\left(({\sf f}(z)-\mathbb B+f(\mathsf f(z))\mathbb{A}^2)^{-1}\right)}-{\sf f}(z).
\end{eqnarray*}
Given a function $f_0\colon\mathbb C^+\to\mathbb C^+$ such that
$\lim_{y\to+\infty}\frac{f_0(iy)}{iy}=0$, one defines
$$
\mathsf f_1(z)=\frac{1}{\tau\left(\mathbb A(z-\mathbb B+f_0(z)\mathbb{A}^2)^{-1}\mathbb A
\right)}-f_0(z),$$
and
$$
f_1(z)=\frac{1}{\tau\left(({\sf f}_1(z)-\mathbb B+f_0(\mathsf f_1(z))\mathbb{A}^2)^{-1}\right)}-{\sf
f}_1(z).
$$
It is clear that $\lim_{y\to+\infty}\frac{\mathsf f_1(iy)}{iy}=\frac{1}{\tau(\mathbb A^2)}$ and
$\lim_{y\to+\infty}\frac{f_1(iy)}{iy}=0$. Inductively, given $(\mathsf f_j,f_j),1\le j\le n,$ satisfying
these same conditions, we let
$$
\mathsf f_{n+1}(z)=\frac{1}{\tau\left(\mathbb A(z-\mathbb B+f_n(z)\mathbb{A}^2)^{-1}\mathbb A
\right)}-f_n(z),$$
and
$$
f_{n+1}(z)=\frac{1}{\tau\left(({\sf f}_n(z)-\mathbb B+f_{n-1}(\mathsf f_n(z))\mathbb{A}^2)^{-1}
\right)}-{\sf f}_n(z).
$$
We conjecture that $(\mathsf f_n,f_n)\to(\mathsf f,f)$ as $n\to +\infty$ in the topology of uniform
convergence on compact subsets of $\mathbb C^+$, where $(\mathsf f,f)$ are the functions
provided by Corollary \ref{eq} with $\mathbb X\stackrel{d}{=}\mathbb B+\mathbb{AXA}$
if and only if $\tau(\mathbb A^2)\leq1$ and $\tau(\log^+\B)<+\infty$.

\section{Tails of the solution to the affine fixed-point equation}\label{sec:tails}

In this section using the analytic tools developed in the previous sections we are able to determine the behavior of the tails of free perpetuities.
                                               
\begin{thm}\label{thm:tails+}
        Assume $\tau(\A)=1$, where $\A\in\tcA$ is non-Dirac and $\B\in\cA$, with $\B\geq 0$ and
$\B\neq 0$. If the model \eqref{eq:affine} is irreducible, then the unique solution $\X$ to
\eqref{eq:affine} exists, and moreover:
 \begin{enumerate}
     \item[(i)] if $\tau(\A^2)<+\infty$, then
         \begin{align}\label{eq:tail1}
        \lim_{t\to+\infty}t^{1/2}\mu_{\X}\big((t,+\infty)\big) = \frac{2\sqrt{2}}{\pi}\sqrt{\frac{\tau(\B)}{\mathrm{Var}(\A)}}.
        \end{align}
     \item[(ii)] if $\mu_\A\big((t,+\infty)\big)\sim c\, t^{-\alpha}$ for $\alpha\in(1,2)$ and $c>0$, then
         \begin{align}\label{eq:tail2}
 \lim_{t\to+\infty}t^{1/\alpha}\mu_{\X}\big((t,+\infty)\big) = \frac{\sin(\pi/\alpha)}{\pi/\alpha}\left( \frac{-\sin(\pi\alpha)}{\pi\,c}\tau(\B)\right)^{1/\alpha}.
        \end{align}
 \end{enumerate}
 \end{thm}

\begin{thm}\label{thm:tailsS}
        Assume $\tau(\A)=1$, where $\A\in\tcA$ is non-Dirac, and $(\A,\B)\stackrel{d}{=}(\A,-\B)$ with
$\B\in\cA$. If the model \eqref{eq:affine} is irreducible, then the unique solution $\X$ to
\eqref{eq:affine} exists, and moreover:
 \begin{enumerate}
     \item[(i)] if $\tau(\A^2)<+\infty$, then
        \[
        \lim_{t\to+\infty}t\,\mu_{|\X|}\big((t,+\infty)\big) = \frac{2}{\pi}\sqrt{\frac{\tau(\B^2)}{\mathrm{Var}(\A)}}.
        \]
 \item[(ii)] if $\mu_\A\big((t,+\infty)\big)\sim c\, t^{-\alpha}$ for $\alpha\in(1,2)$ and $c>0$, then
         \begin{align*}
 \lim_{t\to+\infty}t^{2/\alpha}\,\mu_{|\X|}\big((t,+\infty)\big) = \frac{\sin(\pi/\alpha)}{\pi/\alpha}
\left( \frac{-\sin(\pi\alpha)}{2\pi c}\tau(\B^2)\right)^{1/\alpha}.
        \end{align*}
 \end{enumerate}
 \end{thm}
 
\subsection{Examples}
We continue with both examples from Section \ref{sec:ex1} to illustrate our results.
\subsubsection{Free Beta prime distribution}
Let $a>0$, $b\geq1$ and assume that $\A\sim f\mathcal{B}'_{a,a+b}$. By Corollary \ref{thm:exConv},
$\X\sim f\mathcal{B}'_{a,b}$ is a solution to
\begin{align*}
\X \stackrel{d}{=} \A^{1/2}\X\A^{1/2}+\A,\qquad \A\mbox{ and }\X\mbox{ are freely independent}.
\end{align*}

We have $\tau(\A) = a/(a+b-1)<1$ if and only if $b>1$. In this case, the support of $\X$ is compact.

If $\tau(\A)=1$, then $b=1$, and thus $\X\sim f\mathcal{B}'_{a,1}$.
Since $S_\X(-1/x) = 1/(ax-1)\sim a^{-1}x^{-1}$ as $x\to+\infty$, we can apply
\cite[Theorem 4.5]{KK22} for $\alpha=1/2$ to obtain
\begin{align}\label{eq:tailsB'}
\lim_{t\to+\infty} t^{1/2}\mu_{\X}\big((t,+\infty)\big) = \frac{2\sqrt{a}}{\pi}.
\end{align}

Notice that for $a>0$ and $b>1$, we have $m_1(f\mathcal{B}'_{a,b})= \lim_{z\to 0-} S_{a,b}(z)^{-1}
=\frac{a}{b-1}$ and (recalling \eqref{eq:S-2-terms})
\[
\mathrm{Var}(f\mathcal{B}'_{a,b}) = -m_1(f\mathcal{B}'_{a,b})^3 \lim_{z\to 0-} S_{a,b}'(z) = \frac{a(a+b-1)}{(b-1)^3}.
\]
Thus, for $b=1$, we find that $\mathrm{Var}(\A)=2/a$, which allows us to express the right-hand
side of \eqref{eq:tailsB'} as $\frac{2\sqrt{2}}{\pi} \sqrt{\frac{\tau(\A)}{\mathrm{Var}(\A)}}$.

\subsubsection{Inverse Marchenko-Pastur distribution}

In Section \ref{sec:MY}, we established that if $\lambda\geq 1$ and $\A^{1/2}\sim fGIG_{-\lambda}$,
then $\X\sim \mu_{\lambda}^{-1}$ is a solution to
\[
\X\stackrel{d}{=}\A^{1/2}\X\A^{1/2}+\A^{1/2},\qquad \A\mbox{ and }\X\mbox{ are free}.
\]
When $\lambda>1$, then $\tau(\A)<1$, resulting in a compact support of $\X$.

Conversely, if $\lambda=1$, then $\tau(\A)=1$. By direct calculation using \eqref{eq:IMP}, we
may prove that
\begin{align*}
\lim_{t\to+\infty}t^{1/2}\mu_{\X}\big((t,+\infty)\big)=\frac{2}{\pi}.
\end{align*}
Furthermore, by performing a series expansion of the Cauchy transform of $\A^{1/2}$ (see, e.g.,
\cite[page 380]{KSzMY}), we can derive formulas for $\tau(\A^2)$ and $\tau(\A^{1/2})$ in terms of
the endpoints $(a,b)$ of the support of $\A^{1/2}$. Using the system of equations that define these
endpoints, we find that $\mathrm{Var}(\A)=2 \tau(\A^{1/2})$. This leads to the following equality:
\[
\frac{2}{\pi} = \frac{2\sqrt{2}}{\pi}\sqrt{\frac{\tau(\A^{1/2})}{\mathrm{Var}(\A)}}.
\]

\subsection{Proofs}
Let
\begin{align*}
    \beta(z) &=  \tau\left(\left(1-\frac{\delta(z)}{z}\A-\frac{1}{z}\B\right)^{-1}\right)-1
\end{align*}
for $z\in(-\infty,0)$ in the case of positive $\B$ (Subsection \ref{sec:B>0}) and for $z\in i(-\infty,0)$ in
the case of symmetric $\B$ (Subsection \ref{sec:B=-B}). By Lemma \ref{lem:subordination}, we have
\begin{align*}
\psi_{\A^{1/2}\X\A^{1/2}+\B}(z^{-1}) = \beta(z)\qquad\mbox{ and }\qquad
\psi_{\X}\left(\frac{1}{h(z)+\delta(z)}\right)=\frac{\delta(z)}{h(z)},
\end{align*}
where
\[
\frac{1}{h(z)}=\tau\left(\A^{1/2}(z-\B-\delta(z)\A)^{-1}\A^{1/2}\right).
\]
If $\X\stackrel{d}{=}\A^{1/2}\X\A^{1/2}+\B$, then we additionally have $\beta(z)=\psi_\X(z^{-1})$.
Thus, we obtain
\begin{align}\label{eq:alphabeta}
\beta\left(h(z)+\delta(z)\right) = \frac{\delta(z)}{h(z)}.
\end{align}

We begin with a series of technical lemmas that will be utilized in the proofs of Theorem
\ref{thm:tails+} and Theorem \ref{thm:tailsS}.
\begin{lemma}\label{lem:psiA}
Assume that $\A\in\tcA$, $\A\geq 0$ and $\tau(\A)=1$.
\begin{enumerate}
    \item[(i)] If $\tau(\A^2)<+\infty$, Then,  as $z\to 0-$,
    \begin{align*}
    \frac{1}{\psi_\A(z)} &= \frac{1}{z}-\tau(\A^2)(1+o(1)).
    \end{align*}
    \item[(ii)] If $\mu_\A\big((t,+\infty)\big)\sim c\, t^{-\alpha}$ with $\alpha\in(1,2)$ and $c>0$, then
as  $z\to 0-$,
    \begin{align*}
\frac{1}{\psi_\A\left( z\right)} = \frac{1}{z} + c \frac{\pi\alpha}{\sin(\pi\alpha)} (-z)^{\alpha-2}(1+o(1)).
    \end{align*}
\item[(iii)]     If $a_n\to0-$ and $b_n\to 0-$, then     \[
    \psi_{\A}(a_n)-\psi_\A(b_n)=(a_n-b_n)(1+o(1)).
    \]
\end{enumerate}
\end{lemma}
\begin{proof}
\begin{enumerate}
    \item[(i)] We have $\psi_{\A}(z)=z+\tau(\A^2)z^2(1+o(1))$. Therefore,
    \[
    \frac{1}{\psi_{\A}(z)}= \frac{1}{z} \frac{1}{1+\tau(\A^2)z(1+o(1)} = \frac{1}{z}(1-\tau(\A^2)z(1+o(1))).
    \]
    \item[(ii)] By Theorem \ref{thm:taub} we obtain
    \[
    \psi_\A\left( z\right)= z-c \frac{\pi\alpha}{\sin(\pi\alpha)} (-z)^\alpha (1+o(1)).
    \]
    Similarly to (i), we arrive at
    \[
    \frac{1}{\psi_\A\left( z\right)} = \frac{1}{z} \left(1- c \frac{\pi\alpha}{\sin(\pi\alpha)} (-z)^{\alpha-1} (1+o(1))\right)
    \]
    \item[(iii)]     We have
        \begin{align*}
  \frac{  \psi_{\A}(a_n)-\psi_\A(b_n)}{a_n-b_n}=\int_{[0,+\infty)} \frac{x}{(1-a_n x)(1-b_n x)}\mu_{\A}(\dd x).
    \end{align*}
    Since both $a_n$ and $b_n$ are negative, we have $\frac{x}{(1-a_n x)(1-b_n x)}\leq x\in
L^1(\mu_\A(\dd x))$. By the Lebesgue dominated convergence theorem, we obtain the assertion.
\end{enumerate}
\end{proof}

\subsubsection{Positive $\B$}\label{sec:B>0}
Now, we formulate and prove lemmas which will be applicable in the proof of Theorem \ref{thm:tails+}.
\begin{lemma}\label{lem:L1B>0} 
Assume $\tau(\A)=1$, $\A\in\tcA$ is non-Dirac and $\B\in\cA$ is such that $\B\geq 0$ and $\B\neq 0$.
Additionally, suppose that the model \eqref{eq:affine} is irreducible. As $z\to -\infty$, we have
\begin{enumerate}
    \item[(i)] $\delta(z)\to+\infty$,
    \item[(ii)] $\beta(z)= \psi_\A\left(\frac{\delta(z)}{z}\right) + \frac{\tau(\B)}{z}(1+o(1))$,
    \item[(iii)] $ \frac{1}{h(z)}=\frac{1}{\delta(z)}\psi_{\A}\left(\frac{\delta(z)}{z}\right)+O\left(\frac{1}{z^2}\right)$.
\end{enumerate}
\end{lemma}
\begin{proof}
\begin{enumerate}
\item[(i)]
Since for $t\geq 0$, we have $(1+t)^{-1}-1\geq-t$, we get for $z<0$,
\[
\left(1-\frac{\delta(z)}{z}\A-\frac{1}{z}\B\right)^{-1}-1\geq \frac{\delta(z)}{z}\A+\frac{1}{z}\B.
\]
Applying $\tau$ to both sides and multiplying by $z<0$, we obtain $z\,\beta(z)\leq \delta(z)+\tau(\B)$.

Since $\tau(\A)=1$ and $\tau(\B)>0$, \eqref{eq:affine} leads to $\tau(\X)=\tau(\X)+\tau(\B)$, which is only possible if $\tau(\X)=+\infty$. Therefore,
\[
\lim_{z\to -\infty} z\, \psi_\X(z^{-1}) =\tau(\X) =+\infty.
\]
Using the fact that $\psi_\X(z^{-1}) =  \beta(z)$ (recall Lemma \ref{lem:subordination}), we deduce that $\delta(z)\to+\infty$ as $z\to-\infty$.

\item[(ii)]
Let $a=1-z^{-1}\delta(z) \A$ and $b=-z^{-1}\B$. Note that we have $a\geq 1$ and $b\geq 0$. Let $f(x)=x/(1+x)$.
Then,
\begin{align*}
z\left( \left(1- \frac{\delta(z)}{z}\A-\frac{1}{z}\B\right)^{-1}-\left(1-\frac{\delta(z)}{z}\A\right)^{-1}\right) = z\left( (a+b)^{-1}-a^{-1}\right)\\
= -z a^{-1/2}f(a^{-1/2}b a^{-1/2})a^{-1/2}.
\end{align*}
For $x\geq 0$, we have $x-x^2\leq f(x)\leq x$. Therefore, for $z<0$,
\begin{align*}
  -z a^{-1} b a^{-1} +z a^{-1} b a^{-1} b a^{-1} \leq -z a^{-1/2}f(a^{-1/2}b a^{-1/2})a^{-1/2} \leq -z a^{-1} b a^{-1}.
\end{align*}
Moreover,
\begin{align*}
    -z a^{-1} b a^{-1} = a^{-1}\B a^{-1}\quad\mbox{and}\quad
    z a^{-1} b a^{-1} b a^{-1} = \frac{1}{z} a^{-1} \B a^{-1} \B a^{-1}.
\end{align*}
Since $\B$ is bounded and $a\geq 1$, $a^{-1}\B a^{-1} \B a^{-1}$ is also bounded. Therefore,
\[
\lim_{z\to-\infty} z\, \tau( a^{-1} b a^{-1} b a^{-1}) = 0.
\]
Furthermore,
\[
\lim_{z\to-\infty} (-z)\tau(a^{-1} b a^{-1}) = \lim_{z\to-\infty} \tau(a^{-2}\B) = \tau(\B),
\]
because both $a^{-1}$ and $\B$ are bounded and $a\to 1$, which follows from $\delta(z)/z\to 0$ as $z\to -\infty$. Thus, (ii) follows.
\item[(iii)] With the notation from the previous part, we have
\begin{align*}
    \frac{1}{h(z)} - \frac{1}{\delta(z)}\psi_{\A}\left(\frac{\delta(z)}{z}\right) & = \tau\left(\A^{1/2}\left[ (z-\B-\delta(z)\A)^{-1} - (z-\delta(z)\A)^{-1}\right]\A^{1/2}\right) \\
    &= \frac{1}{-z}\tau(\A^{1/2}a^{-1/2}f(a^{-1/2}b a^{-1/2})a^{-1/2}\A^{1/2})   \\
    &\leq \frac{1}{-z}\tau(\A^{1/2}a^{-1}ba^{-1}\A^{1/2})= \frac{1}{z^2}\tau(\A^{1/2}a^{-1}\B a^{-1}\A^{1/2})\\
    &\leq \frac{1}{z^2}\tau(\A a^{-2}) \|\B\| \leq \frac{1}{z^2}\tau(\A) \|\B\|.  
\end{align*}
The left-hand side above is nonnegative, completing the proof.
\end{enumerate}
\end{proof}

\begin{lemma}\label{lem:L2B>0}
Assume $\tau(\A)=1$, $\A\in\tcA$ is non-Dirac and $\B\in\cA$ is such that $\B\geq 0$ and $\B\neq 0$. Additionally, assume that the model \eqref{eq:affine} is irreducible.
\begin{enumerate}
    \item[(i)] If $\tau(\A^2)<+\infty$, then as $z\to -\infty$,
    \[
    h(z) = z-\tau(\A^2) \delta(z)(1+o(1)).
    \]
    \item[(ii)] If $\mu_\A\big( (t,+\infty)\big) \sim c\,t^{-\alpha}$ with $\alpha\in(1,2)$ and $c>0$, then as $z\to -\infty$,
    \[
    h(z) = z -c\frac{\pi\alpha}{\sin(\pi\alpha)} z\left(\frac{\delta(z)}{-z}\right)^{\alpha-1} (1+o(1)).
    \]
\end{enumerate}
\end{lemma}
\begin{proof}
We write
\[
h(z) -z = \left(h(z) -\frac{\delta(z)}{\psi_{\A}\left(\frac{\delta(z)}{z}\right)} \right) + \left(\frac{\delta(z)}{\psi_{\A}\left(\frac{\delta(z)}{z}\right)} -z \right).
\]

By Lemma \ref{lem:L1B>0}, we have $\delta(z)h(z)\sim \psi_{\A}\left(\frac{\delta(z)}{z}\right)$, and thus,
\[
h(z)-\frac{\delta(z)}{\psi_{\A}\left(\frac{\delta(z)}{z}\right)} =
\frac{h(z)\delta(z)}{\psi_{\A}\left(\frac{\delta(z)}{z}\right)}
\left(
\frac{1}{\delta(z)}\psi_{\A}\left(\frac{\delta(z)}{z}\right)
-\frac{1}{h(z)}\right) =O\left(\frac{1}{z^2}\right).
\]
\begin{enumerate}
    \item[(i)] Since $\tau(\A^2)<+\infty$, by Lemma \ref{lem:psiA} (i), we obtain
\begin{align*}
     \frac{\delta(z)}{\psi_{\A}\left(\frac{\delta(z)}{z}\right)} - z = -\delta(z)\tau(\A^2)(1+o(1)).
\end{align*}
    \item[(ii)] By Lemma \ref{lem:psiA} (ii), we obtain
\begin{align*}
    \frac{\delta(z)}{\psi_{\A}\left(\frac{\delta(z)}{z}\right)} -z &= c \frac{\pi\alpha}{\sin(\pi\alpha)} \delta(z) \left(\frac{\delta(z)}{-z}\right)^{\alpha-2}(1+o(1)).
\end{align*}
Since $\delta(z)=o(z)$, we conclude (ii).
\end{enumerate}
\end{proof}

\begin{proof}[Proof of Theorem \ref{thm:tails+}]
By Theorem \ref{thm:taub} (i), the expressions \eqref{eq:tail1} and \eqref{eq:tail2} are equivalent to the following, respectively, as $t\to+\infty$,
\begin{align*}
    -\psi_\X\left( -\frac{1}{t}\right) &\sim \sqrt{\frac{2\,\tau(\B)}{\mathrm{Var}(\A)} } \frac{1}{t^{1/2}},\\
    -\psi_\X\left( -\frac{1}{t}\right) &\sim \left(\frac{-\sin(\pi\alpha)}{\pi c}\tau(\B)\right)^{1/\alpha} \frac{1}{t^{1/\alpha}}.
\end{align*}
By Lemma \ref{lem:L1B>0} (iii), we have $\beta(z)\sim \psi_\A\left( \frac{\delta(z)}{z}\right)\sim \frac{\delta(z)}{z}$. Since $\psi_\X(z^{-1})=\beta(z)$, to complete the proof, it is sufficient to show that, as $t\to+\infty$,
\begin{align}
 \frac{\delta(-t)}{t} &\sim \sqrt{\frac{2\,\tau(\B)}{\mathrm{Var}(\A)} } \frac{1}{t^{1/2}},\label{eq:delta1}\\
\frac{\delta(-t)}{t}&\sim \left(\frac{-\sin(\pi\alpha)}{\pi c}\tau(\B)\right)^{1/\alpha} \frac{1}{t^{1/\alpha}},\label{eq:delta2}
\end{align}
respectively. By Lemma \ref{lem:L1B>0}, we already know that $\delta(z)\to+\infty$ as $z\to-\infty$.

For a fixed $z_0<0$, we define a sequence $(z_n)_{n\geq 1}$ by
\[
z_n = h(z_{n-1})+\delta(z_{n-1}),
\]
and denote $h_n = h(z_n)$ and $\delta_n=\delta(z_n)$ for $n=1,2,\ldots$. We will first show \eqref{eq:delta1} and \eqref{eq:delta2} hold for this sequence.

First, we prove (i).  By Lemma \ref{lem:L2B>0} (i), we have
    \[
   h(z) -z \sim -\tau(\A^2)\delta(z),
    \]
which implies that there exists $z_0<0$ such that for any $z<z_0$,
    \begin{align}\label{eq:z0}
 z-h(z)\geq  (1-\varepsilon) \tau(\A^2)\delta(z)>0.
    \end{align}
 This clearly implies that 
 \[
 z_{n-1}-z_n = z_{n-1}-h_{n-1}-\delta_{n-1}\geq ((1-\varepsilon)\tau(\A^2)-1)\delta_{n-1} = (\mathrm{Var}(\A)-\eps\tau(\A^2))\delta_{n-1}>0
 \]
for sufficiently small $\varepsilon$.  Thus, the sequence $(z_n)_{n\geq 1}$ is decreasing and has a limit, say $g\leq z_0$. We conclude that $g=-\infty$, because if $g\in(-\infty,0)$, this would imply that $\delta_n\to 0$, which is impossible. Hence, $z_n\to -\infty$ as $n\to+\infty$.

The sequences $(z_n)_{n\geq1}$ and $(\delta_n)_{n\geq1}$ satisfy the following properties:
\begin{enumerate}
 \item[(a)] $z_{n-1}-z_{n }\sim \mathrm{Var}(\A)\delta_{n-1}$; thus,  $z_{n-1}\sim z_n$,
  \item[(b)] $\frac{\delta_n}{z_n} - \frac{\delta_{n-1}}{z_{n-1}}\sim - \frac{\tau(\B)}{z_n}$; thus, $\frac{\delta_n}{z_n} \sim \frac{\delta_{n-1}}{z_{n-1}}$,
  \item[(c)] $\lim_{n\to+\infty} \frac{\delta_n^2}{z_n} = -\frac{2\tau(\B)}{\mathrm{Var}(\A)}$.
\end{enumerate}
Property (a) follows from Lemma \ref{lem:L2B>0} (i).
By \eqref{eq:alphabeta}, we have $\beta(z_n)=\delta_{n-1}/h_{n-1}$. By Lemma \ref{lem:L1B>0} and the fact that $z_{n-1}\sim z_n$, we obtain
\begin{align*}
   \frac{\delta_{n-1}}{h_{n-1}} &= \psi_\A\left( \frac{\delta_{n-1}}{z_{n-1}}\right) + O\left(\frac{\delta_{n-1}}{z_{n-1}^2} \right) = \psi_\A\left( \frac{\delta_{n-1}}{z_{n-1}}\right)+o\left(\frac{1}{z_{n}}\right), \\
    \beta(z_{n}) &= \psi_\A\left( \frac{\delta_{n}}{z_{n}}\right) + \frac{\tau(\B)}{z_{n}}(1+o(1)).
\end{align*}
Thus, by Lemma \ref{lem:psiA} (iii), we get
\begin{align*}
\frac{\delta_n}{z_n} - \frac{\delta_{n-1}}{z_{n-1}} \sim \psi_\A\left( \frac{\delta_{n}}{z_{n}}\right) - \psi_\A\left( \frac{\delta_{n-1}}{z_{n-1}}\right)
\sim - \frac{\tau(\B)}{z_{n}},
\end{align*}
which establishes property (b).

By properties (a) and (b), we have
\begin{align*}
\frac{\left(\frac{\delta_n}{z_n}\right)^2 - \left(\frac{\delta_{n-1}}{z_{n-1}}\right)^2}{\frac{1}{z_n}-\frac{1}{z_{n-1}}}
\sim  \frac{\left(\frac{\delta_n}{z_n} - \frac{\delta_{n-1}}{z_{n-1}}\right) 2 \frac{\delta_n}{z_n}}{(z_{n-1}-z_n)\frac{1}{z_n^2}} \sim  \frac{\left( -\frac{\tau(\B)}{z_n}\right) 2 \frac{\delta_n}{z_n}}{\mathrm{Var}(\A)\delta_{n-1} \frac{1}{z_n^2}} \sim  -\frac{2\tau(\B)}{\mathrm{Var}(\A)}.
\end{align*}
The Stolz–Ces\`{a}ro theorem implies (c). Indeed, since $(z_n)_{n\geq 1}$ is monotonic and diverging, we obtain
\[
 \lim_{n\to+\infty} \frac{\left(\frac{\delta_n}{z_n}\right)^2}{\frac{1}{z_n}} =\lim_{n\to+\infty} \frac{\left(\frac{\delta_n}{z_n}\right)^2 - \left(\frac{\delta_{n-1}}{z_{n-1}}\right)^2}{\frac{1}{z_n}-\frac{1}{z_{n-1}}}.
\]

Next, we show that $\lim_{z\to -\infty} \frac{\delta(z)^2}{z} =  -\frac{2\tau(\B)}{\mathrm{Var}(\A)}$.
Since the function $(-\infty,z_0)\ni z\mapsto \delta(z)$ is monotonic and diverges to $+\infty$ as $z\to -\infty$, for $z\in[z_n,z_{n-1})$, we have
$\delta_{n-1} \leq  \delta(z) \leq \delta_n$
which leads to
\[
\frac{z_n}{z_{n-1}} \frac{\delta_{n}^2}{z_{n}} \leq \frac{\delta(z)^2}{z} \leq \frac{z_{n-1}}{z_{n}} \frac{\delta_{n-1}^2}{z_{n-1}}.
\]
Since $z_n\sim z_{n-1}$, we get
\[
\limsup_{z\to-\infty} \frac{\delta(z)^2}{z}\leq \limsup_{n\to+\infty} \frac{z_{n-1}}{z_{n}} \frac{\delta_{n-1}^2}{z_{n-1}} = -\frac{2\,\tau(\B)}{\mathrm{Var}(\A)}.
\]
A similar argument can be applied to the $\liminf$.
Therefore, we conclude that as $z\to-\infty$,
\[
\delta(z)\sim \sqrt{-\frac{2\,\tau(\B)}{\mathrm{Var}(\A)} z},
\]
which is equivalent to \eqref{eq:delta1}. This completes the proof of (i).

We now turn to point (ii). Similarly as before, Lemma \ref{lem:L2B>0} (ii)  implies that there exists $z_0<0$ such that the sequence $(z_n)_{n\geq 1}$ is decreasing and $z_n\to-\infty$ as $n\to+\infty$.

We will show that the sequences $(z_n)_{n\geq 1}$ and $(\delta_n)_{n\geq 1}$ satisfy the following properties:
\begin{enumerate}
    \item[(a')]  $z_{n-1}-z_{n}\sim c \frac{\pi\alpha}{\sin(\pi\alpha)} \left(\frac{\delta_{n-1}}{-z_{n-1}}\right)^{\alpha-1} z_{n-1}$; thus, $z_n\sim z_{n-1}$,
    \item[(b')]  $\frac{\delta_n}{z_n}-\frac{\delta_{n-1}}{z_{n-1}}\sim -\frac{\tau(\B)}{z_n}$; thus $\frac{\delta_n}{z_n}\sim\frac{\delta_{n-1}}{z_{n-1}}$,
    \item[(c')]  $\lim_{n\to+\infty} z_n \left(\frac{\delta_n}{-z_n}\right)^\alpha =
\frac{\sin(\pi\alpha)}{\pi c}\tau(\B)$.
\end{enumerate}
Property (a') follows from Lemma \ref{lem:L2B>0} (ii), while property (b') mirrors our earlier findings.

Now we calculate
    \begin{align*}
    \frac{\left(\frac{\delta_n}{-z_n}\right)^\alpha-\left(\frac{\delta_{n-1}}{-z_{n-1}}\right)^\alpha}{\frac{1}{z_n}-\frac{1}{z_{n-1}}}
&\sim \frac{ \alpha \left(\frac{\delta_n}{-z_n}\right)^{\alpha-1} \left( \frac{\delta_n}{-z_n}-\frac{\delta_{n-1}}{-z_{n-1}} \right) }{ (z_{n-1}-z_n) \frac{1}{z_n^2}} \sim   \frac{  \alpha \left(\frac{\delta_n}{-z_n}\right)^{\alpha-1} \frac{\tau(\B)}{z_n}}{c \frac{\pi\alpha}{\sin(\pi\alpha)} \left(\frac{\delta_{n-1}}{-z_{n-1}}\right)^{\alpha-1} z_{n-1}\frac{1}{z_n^2}}\\
&\sim
\frac{\sin(\pi\alpha)}{\pi c}\tau(\B).
    \end{align*}
Thus, by the Stolz–Ces\`{a}ro theorem, we obtain (c').
Arguing similarly as before, we find that
\[
\lim_{z\to-\infty} z \left( \frac{\delta(z)}{-z}\right)^\alpha = \frac{\sin(\pi\alpha)}{\pi c}\tau(\B),
\]
which is \eqref{eq:delta2}. This concludes the proof.
\end{proof}

\subsubsection{Symmetric $\B$}\label{sec:B=-B}

Denote for $t\in\R$,
\[
\Delta(t) = \delta(i t)/i.
\]
Then, for $z<0$,
\[
\beta(i z) =  \tau\left(\left(1-\frac{\Delta(z)}{z}\A+\frac{i}{z}\B\right)^{-1}\right)-1.
\]
Under the assumption that $(\A,\B)\stackrel{d}{=}(\A,-\B)$ and $\X\stackrel{d}{=}-\X$, the function $(-\infty,0)\ni z\mapsto\beta(i z)$ is real. In this case, we have
\[
\beta(i z) =  \tau\left(\Re\left(\left(1-\frac{\Delta(z)}{z}\A+\frac{i}{z}\B\right)^{-1}\right)\right)-1.
\]

\begin{lemma}\label{lem:simpleineq}
    Let $t\in\mathbb{C}$ with $\Re(t)\geq 0$. Then,
    \[
\Re\left(\frac{1}{1+t}\right)\geq 1 -\Re(t)-\Im(t)^2.
    \]
\end{lemma}
\begin{proof}
    Let $t=x+i y$ with $x,y\in\R$. Then,
    \[
 \Re\left(\frac{1}{1+t}\right) = \frac{1+x}{(1+x)^2+y^2}.  
    \]
 The inequality
 \[
1+x\geq(1-x-y^2)\left( (1+x)^2+y^2\right)
 \]
is equivalent to
    \[
x^2(1+x)+x(3+x)y^2+y^4\geq 0,
    \]
    which is clearly satisfied for all $x\geq 0$.
\end{proof}

\begin{lemma}\label{lem:L1Bs}\
Assume $(\A,\B)\stackrel{d}{=}(\A,-\B)$, $\tau(\A)=1$, $\A\in\tcA$ is non-Dirac and $\B\in\cA$ with $\B\neq 0$. Additionally, suppose that the model \eqref{eq:affine} is irreducible. As $z\to -\infty$, we have
\begin{enumerate}
    \item[(i)] $z\Delta(z)\to-\infty$,
    \item[(ii)] $ \beta(i z)=\psi_\A\left(\frac{\Delta(z)}{z}\right)- \frac{\tau(\B^2)}{z^2}(1+o(1))$,
    \item[(iii)] $\frac{i}{h(iz)}= \frac{1}{\Delta(z)} \psi_{\A}\left(\frac{\Delta(z)}{z}\right)+O\left(\frac{1}{z^3}\right)$.
\end{enumerate}
\end{lemma}
\begin{proof}
\begin{enumerate}
    \item[(i)]
Since $\tau(\A)=1$ and $\B\neq0$, \eqref{eq:affine} implies that $\tau(\X^2)=+\infty$. Therefore,
\[
\lim_{z\to -\infty} z^2 \psi_\X\left(\frac{1}{i z}\right) = \lim_{z\to -\infty} z^2 \psi_{\X^2}\left(\frac{1}{-z^2}\right)=-\tau(\X^2) =-\infty.
\]
Applying Lemma \ref{lem:simpleineq} to $t=-\frac{\Delta(z)}{z}\A+\frac{i}{z}\B$ for $z<0$ and taking $\tau$ of both sides, we arrive at
\[
\beta(i z)\geq  \frac{\Delta(z)}{z}+\frac{1}{z^2}\tau(\B^2).
\]
Since $\beta(iz)=\psi_\X((iz)^{-1})$, we obtain
\[
-\infty = \lim_{z\to-\infty} z^2 \beta(i z) \geq \limsup_{z\to +\infty} \left(z \Delta(z)+\tau(\B^2)\right),
\]
which establishes (i).
\item[(ii)]
 Let $a=1-z^{-1}\Delta(z) \A$ and $b=z^{-1}\B$. Note that $a\geq 1$ for $z<0$. Let $f(x)=x/(1+x)$.
Then,
\begin{align*}
 \left(1- \frac{\Delta(z)}{z}\A+\frac{i}{z}\B\right)^{-1}-\left(1-\frac{\Delta(z)}{z}\A\right)^{-1} =  (a+i b)^{-1}-a^{-1}\\
= - a^{-1/2}f(i a^{-1/2}b a^{-1/2})a^{-1/2}
\end{align*}
For $x\in\R$, we have $x^2-x^4\leq \Re(f(ix)) = f(x^2) \leq x^2$. Therefore, for $z<0$,
\begin{align*}
 a^{-1} b a^{-1} b a^{-1} - (a^{-1}b)^4 a^{-1}\leq  a^{-1/2}\Re\left(f(i a^{-1/2}b a^{-1/2})a^{-1/2} \right) \leq  a^{-1} b a^{-1} b a^{-1}.
\end{align*}
Moreover, we have
\begin{align*}
     a^{-1} b a^{-1} b a^{-1} = \frac{1}{z^2} a^{-1} \B a^{-1} \B a^{-1}\quad\mbox{and}\quad
      (a^{-1}b)^4 a^{-1} = \frac{1}{z^4}(a^{-1}\B)^4 a^{-1}.
\end{align*}
Clearly, $(a^{-1}\B)^4 a^{-1}$ is bounded because $\B$ is bounded and $a\geq 1$. 
Therefore
\[
\lim_{z\to-\infty} z^2\, \tau\left( (a^{-1}b)^4 a^{-1}\right) = 0.
\]
Moreover, since $\Delta(z)/z\to 0$, we have $a \to 1$. Therefore,
\[
\lim_{z\to-\infty} z^2 \tau\left(a^{-1} b a^{-1} b a^{-1}\right) = \lim_{z\to-\infty} \tau(a^{-1}\B a^{-1}\B a^{-1}) = \tau(\B^2)<+\infty.
\]
This gives (ii).
\item[(iii)] We leave the proof of (iii) to the reader, as it is similar to the proof of the Lemma \ref{lem:L1B>0} (iii).
\end{enumerate}
\end{proof}

\begin{lemma}\label{lem:L2Bs}
Assume $(\A,\B)\stackrel{d}{=}(\A,-\B)$, $\tau(\A)=1$, $\A\in\tcA$ is non-Dirac and $\B\in\cA$ with $\B\neq 0$. Additionally, assume that the model \eqref{eq:affine} is irreducible.
\begin{enumerate}
    \item[(i)] If $\tau(\A^2)<+\infty$, then as $z\to-\infty$,
        \[
    \frac{1}{i}h(iz) = z  -\tau(\A^2)\Delta(z)(1+o(1)).
    \]
    \item[(ii)] If $\mu_\A\big( (t,+\infty)\big) \sim c\, t^{-\alpha}$ with $\alpha\in(1,2)$ and $c>0$, then  as $z\to-\infty$,
        \[
       \frac{1}{i}h(iz) = z  - c\frac{\pi\alpha}{\sin(\pi\alpha)} z\left(\frac{\Delta(z)}{-z}\right)^{\alpha-1} (1+o(1)).
    \]
\end{enumerate}
\end{lemma}
\begin{proof}
We write
\[
 \frac{1}{i}h(iz) - z = \left( \frac{1}{i}h(iz) -\frac{\Delta(z)}{\psi_{\A}\left(\frac{\Delta(z)}{z}\right)} \right) + \left(\frac{\Delta(z)}{\psi_{\A}\left(\frac{\Delta(z)}{z}\right)} - z \right).
\]
By Lemma \ref{lem:L1Bs}, we obtain
\begin{align*}
\frac{1}{i}h(iz) -\frac{\Delta(z)}{\psi_{\A}\left(\frac{\Delta(z)}{z}\right)} &=  \frac{1}{i}h(iz)\frac{\Delta(z)}{\psi_{\A}\left(\frac{\Delta(z)}{z}\right)}\left(\frac{1}{\Delta(z)} \psi_{\A}\left(\frac{\Delta(z)}{z}\right)-\frac{i}{h(iz)}\right) \\
&\sim \frac{1}{z^3}\tau(\A\B^2).
\end{align*}
Similarly as in the proof of Lemma \ref{lem:L2B>0}, by applying (i) and (ii) from Lemma \ref{lem:psiA}, we obtain the assertion.
\end{proof}

\begin{proof}[Proof of Theorem \ref{thm:tailsS}]
Define a sequence $(z_n)_{n\geq 1}$ by
\[
z_n = \frac{h(i z_{n-1})}{i} +\Delta(z_{n-1}).
\]
Note that all $z_n$ are real. Following the arguments in the proof of Theorem \ref{thm:tails+}, we can show that $(z_n)_{n\geq 1}$ is monotonic and diverges to $-\infty$. 
Denote $\Delta_n=\Delta(z_n)$ for $n\geq 1$.

First, consider (i).
We will prove that the sequences $(z_n)_{n\geq 1}$ and $(\Delta_n)_{n\geq 1}$ satisfy the following properties:
\begin{enumerate}
 \item[(a)] $z_{n-1}-z_{n}\sim \mathrm{Var}(\A)\Delta_{n-1}$; thus, $z_{n-1}\sim z_n$,
  \item[(b)] $\frac{\Delta_n}{z_n} - \frac{\Delta_{n-1}}{z_{n-1}}\sim \frac{\tau(\B^2)}{z_n^2}$; thus, $\frac{\Delta_n}{z_n} \sim \frac{\Delta_{n-1}}{z_{n-1}}$,
  \item[(c)] $\lim_{n\to+\infty} \Delta_n^2 = \frac{\tau(\B^2)}{\mathrm{Var}(\A)}$.
\end{enumerate}
By Lemma \ref{lem:L2Bs} (i), we have
\[
z_n - z_{n-1} = \frac{h(i z_{n-1})}{i}-z_{n-1}+\Delta_{n-1}\sim -\tau(\A^2)\Delta_{n-1}(1+o(1)) +\tau(\A)^2 \Delta_{n-1},
\]
which establishes (a). Using \eqref{eq:alphabeta}, we get
\[
\beta(i z_n) =  \frac{i \Delta(z_{n-1})}{h(iz_{n-1})}.
\]
From Lemma \ref{lem:L1Bs}, we obtain
\begin{align*}
    \frac{i\, \Delta(z_{n-1})}{h(iz_{n-1})} &= \psi_\A\left( \frac{\Delta_{n-1}}{z_{n-1}}\right) + O\left(\frac{1}{\Delta_{n-1}z_{n-1}^3}\right) = \psi_\A\left( \frac{\Delta_{n-1}}{z_{n-1}}\right) + o\left( \frac{1}{z_{n}^2}\right),\\
    \beta(i z_{n}) &= \psi_\A\left( \frac{\Delta_{n}}{z_{n}}\right) - \frac{\tau(\B^2)}{z_{n}^2}(1+o(1)).
\end{align*}
Thus, by Lemma \ref{lem:psiA} (iii), we get
\begin{align*}
\frac{\Delta_n}{z_n} - \frac{\Delta_{n-1}}{z_{n-1}} \sim \psi_\A\left( \frac{\Delta_{n}}{z_{n}}\right) - \psi_\A\left( \frac{\Delta_{n-1}}{z_{n-1}}\right) = \frac{\tau(\B^2)}{z_{n}^2}(1+o(1)),
\end{align*}
which establishes (b).

Moreover,
\begin{align*}
\frac{\left(\frac{\Delta_n}{z_n}\right)^2 - \left(\frac{\Delta_{n-1}}{z_{n-1}}\right)^2}{\frac{1}{z_n^2}-\frac{1}{z_{n-1}^2}}
\sim  \frac{\left(\frac{\Delta_n}{z_n} - \frac{\Delta_{n-1}}{z_{n-1}}\right) 2 \frac{\Delta_n}{z_n}}{(z_{n-1}-z_n)\frac{2}{z_n^3}}
\sim  \frac{ \frac{\tau(\B^2)}{z_n^2} 2 \frac{\Delta_n}{z_n}}{\mathrm{Var}(\A)\Delta_{n-1} \frac{2}{z_n^3}} \sim  \frac{\tau(\B^2)}{\mathrm{Var}(\A)}.
\end{align*}
The Stolz–Ces\`{a}ro theorem implies property (c).
Since the function $(-\infty,0)\ni z\mapsto\Delta(z)$ is decreasing, for $z\in [z_n,z_{n-1})$, we have $\Delta_{n-1}\leq \Delta(z)\leq \Delta_n$. Consequently, we obtain
\[
\limsup_{n\to+\infty}\Delta_{n-1} \leq \liminf_{z\to-\infty}\Delta(z)\leq \limsup_{z\to+\infty}\Delta(z)\leq \liminf_{n\to+\infty}\Delta_n.
\]
This implies that 
\[
\lim_{z\to -\infty} \Delta(z) = \sqrt{\frac{\tau(\B^2)}{\mathrm{Var}(\A)}}.
\]
Thus, we find
\[
-\psi_\X\left( -\frac{1}{it}\right) = -\beta(-it) \sim \frac{\Delta(-t)}{t} \sim \sqrt{\frac{\tau(\B^2)}{\mathrm{Var}(\A)} } \frac{1}{t}.
\]
By Theorem \ref{thm:taub2} for $\alpha=1$ and $L\equiv \frac{2}{\pi}\sqrt{\frac{\tau(\B^2)}{\mathrm{Var}(\A)} }$, we conclude the assertion.

We now prove (ii). The sequences $(z_n)_{n\geq 1}$ and $(\Delta_n)_{n\geq 1}$ satisfy the following properties
\begin{enumerate}
 \item[(a')] $z_{n-1} - z_{n} \sim c\frac{\pi\alpha}{\sin(\pi\alpha)}\left(\frac{\Delta_{n-1}}{-z_{n-1}}\right)^{\alpha-1} z_{n-1}$; thus, $z_{n-1}\sim z_n$,
  \item[(b')] $\frac{\Delta_n}{z_n} - \frac{\Delta_{n-1}}{z_{n-1}}\sim \frac{\tau(\B^2)}{z_n^2}$; thus, $\frac{\Delta_n}{z_n} \sim \frac{\Delta_{n-1}}{z_{n-1}}$,
  \item[(c')] $\lim_{n\to+\infty} z_n^2\left(\frac{\Delta_n}{-z_n}\right)^\alpha = \frac{ - \sin(\pi\alpha) }{2 \pi  c}\tau(\B^2)$.
\end{enumerate}
Property (a') follows from Lemma \ref{lem:L2B>0} (ii), while property (b') coincides with (b) from part (i). By properties (a') and (b'), we have
\begin{align*}
    \frac{\left(\frac{\Delta_n}{-z_n}\right)^\alpha-\left(\frac{\Delta_{n-1}}{-z_{n-1}}\right)^\alpha}{\frac{1}{z_n^2}-\frac{1}{z_{n-1}^2}}
&\sim \frac{ \alpha \left(\frac{\Delta_n}{-z_n}\right)^{\alpha-1} \left( \frac{\Delta_n}{-z_n}-\frac{\Delta_{n-1}}{-z_{n-1}} \right) }{ (z_{n-1}-z_n) \frac{2}{z_n^3}} \sim   \frac{ - \alpha \left(\frac{\Delta_n}{-z_n}\right)^{\alpha-1} \frac{\tau(\B^2)}{z_n^2}}{c \frac{\pi\alpha}{\sin(\pi\alpha)} \left(\frac{\Delta_{n-1}}{-z_{n-1}}\right)^{\alpha-1} z_{n-1}\frac{2}{z_n^3}},
\end{align*}
which simplifies to
\[
  \frac{ - \sin(\pi\alpha) }{2 \pi  c}\tau(\B^2).
\]
By the Stolz–Ces\`{a}ro theorem, we obtain (c').

For $z\in[z_n,z_{n-1})$, we have
\[
z_{n-1}^2\left(\frac{\Delta_{n-1}}{-z_n}\right)^\alpha \leq z^2\left(\frac{\Delta(z)}{-z}\right)^\alpha \leq z_n^2\left(\frac{\Delta_n}{-z_{n-1}}\right)^\alpha.
\]
Since $z_n\sim z_{n-1}$, using (c'), we conclude as $z\to -\infty$,
\[
\Delta(z) \sim \left( \frac{ - \sin(\pi\alpha) }{2 \pi  c}\tau(\B^2) \right)^{1/\alpha} \frac{1}{(-z)^{2/\alpha-1}}.
\]
Thus, we find that
\[
-\psi_\X\left( -\frac{1}{it}\right) \sim \frac{\Delta(-t)}{t} \sim \left( \frac{ - \sin(\pi\alpha) }{2 \pi  c}\tau(\B^2) \right)^{1/\alpha} \frac{1}{t^{2/\alpha}}.
\]
Applying Theorem \ref{thm:taub2} with index $2/\alpha\in(1,2)$ and $L\equiv\frac{\sin(\pi/\alpha)}{\pi/\alpha}\left( \frac{-\sin(\pi\alpha)}{2\pi c}\tau(\B^2)\right)^{1/\alpha}$, we conclude the result.
\end{proof}

\section*{Appendix}
In the Appendix, we provide several supplementary results related to the limit theorem introduced in \cite{SY13}. In particular, we prove that the convergence, originally established in distribution, can be strengthened to hold in $L^2$. 

\begin{thm}\label{thm:SY}
Assume that $\mu\in\mathcal{M}_+$ is non-Dirac, and let $(\Pi_{n,k})_{1\leq k\leq n}$ be freely independent copies of $\Pi_n\sim \mu^{\boxtimes n}$ for $n\in\mathbb{N}$.
    \begin{enumerate}
    \item[(i)]
  If $m_2(\mu)<+\infty$, define
    \[
    \Y_n = \frac{1}{n \,m_1(\mu)^n} \sum_{k=1}^n \Pi_{n,k}.
    \]
Then, $\Y_n$ converges in distribution to a random variable $\Y$ with $S_{\Y}(z) = \exp(-\eta z)$, $z\in(-1,0)$, where $\eta = \mathrm{Var}(\mu)/m_1(\mu)^2$.
   \item[(ii)] If $\mu\big((t,+\infty)\big)\sim c\, t^{-\alpha}$ with $\alpha\in(1,2)$, define
   \[
   \Y_n^{(\alpha)} = \frac{1}{n^\beta m_1(\mu)^n} \sum_{k=1}^{\lfloor n^\beta\rfloor} \Pi_{n,k},
   \]
   where $\beta=1/(\alpha-1)$. Then, $\Y_n^{(\alpha)}$ converges in distribution to a random variable $\Y^{(\alpha)}$ with
   $S_{\Y^{(\alpha)}}(z) = \exp(\eta_\alpha(-z)^{\alpha-1})$, $z\in(-1,0)$, where
   \[
   \eta_\alpha = c \frac{\pi\alpha}{\sin(\pi\alpha)m_1(\mu)^{\alpha}}.
   \]
    \end{enumerate}
\end{thm}
\begin{proof}
    The first part was proved in \cite[Theorem 3.3]{SY13}.
The distribution $\mu_n^{\alpha}$ of $\Y_n^{(\alpha)}$ can be expressed as
\[
\mu_n^{\alpha} = D_{\tfrac{1}{n^\beta m_1(\mu)^n}} (\mu^{\boxtimes n})^{\boxplus \lfloor n^\beta\rfloor},
\]
where $D_a$, $a>0$, is a dilation operator defined by $D_a \mu(\cdot) = \mu(a^{-1}\cdot)$.
By \cite[(3.7) and (3.11)]{BN08}, we have 
\[
S_{D_a \mu} = \frac{1}{a}S_\mu\quad\mbox{and}\quad S_{\mu^{\boxplus n}}(z) = \frac{1}{n} S_\mu\left(\frac{z}{n}\right).
\]
Moreover, by \cite[Theorem 4.9 and Remark 4.10]{KK22} applied to $L\equiv c$, we obtain, as $z\to 0^-$,
\[
S_\mu(z) = \frac{1}{m_1(\mu)}\left(1+\eta_\alpha (-z)^{\alpha-1}(1+o(1))\right).
\]
Thus,
\begin{align*}
S_{\mu_n^{\alpha}}(z) &= n^\beta m_1(\mu)^n S_{(\mu^{\boxtimes n})^{\boxplus \lfloor n^\beta\rfloor}}(z) = \frac{n^\beta}{\lfloor n^\beta\rfloor} m_1(\mu)^n S_{\mu^{\boxtimes n}}\left(\frac{z}{\lfloor n^\beta \rfloor}\right) \\
& = \frac{n^\beta}{\lfloor n^\beta\rfloor} m_1(\mu)^n S_{\mu}\left(\frac{z}{\lfloor n^\beta\rfloor}\right)^n\\
&= \frac{n^\beta}{\lfloor n^\beta\rfloor} \left(1 + \frac{\eta_\alpha(-z)^{\alpha-1}}{\lfloor n^\beta\rfloor^{\alpha-1}}(1+o(1))\right)^n.
\end{align*}
Since $\beta(\alpha-1)=1$, the right-hand side above converges as $n\to+\infty$ for any $z\in(0,1)$ to $\exp(\eta_\alpha(-z)^{\alpha-1})$. The pointwise convergence of the $S$-transforms implies the weak convergence of the underlying measures.
\end{proof}

Under the assumptions of Theorem \ref{thm:SY} (i), define
     \[
    \Y_n^\uparrow = \frac{1}{n\, m_1(\mu)^n} \sum_{k=1}^n \Pi_{n,k}^\uparrow,
    \] 
    where $\Pi_{n,k}^\uparrow = \A_{1,k}^{1/2}\ldots\A_{n-1,k}^{1/2}\A_{n,k}\A_{n-1,k}^{1/2}\ldots \A_{1,k}^{1/2}$ and $(\A_{n,k})_{1\leq k\leq n}$ is an array of freely identically distributed random variables with distribution $\mu$. Then, $(\Y_n^\uparrow)_{n\in\mathbb{N}}$ is a non-negative, non-commutative martingale with unit mean. By Corollary \ref{lem:secondCumulant}, we have
        \begin{align*}
        \mathrm{Var}(\Y_n^\uparrow)=\frac{1}{n^2m_1(\mu)^{2n}}\sum_{k=1}^n\mathrm{Var}(\Pi_{n,k}^\uparrow)=\frac{1}{n\,m_1(\mu)^{2n}} \mathrm{Var}(\mu^{\boxtimes n})=\frac{\mathrm{Var}(\mu)}{m_1(\mu)^2}.
    \end{align*}
    Thus, we obtain $\sup_n \tau((\Y_n^{\uparrow})^2)<+\infty$. Therefore, by \cite[Proposition 8]{Cu71}, $(\Y_n^\uparrow)_{n\in\mathbb{N}}$ converges almost uniformly and in $L^2(\tcA,\tau)$.

\begin{remark}
Let $\mu\in\mathcal{M}_+$ and assume that $m_p(\mu)<+\infty$ for some $p\in\mathbb{N}$.  By the definition of $\Y_n$ from Theorem \ref{thm:SY} (i), we have
    \[
  \kappa_p(\Y_n) = \frac{\kappa_p(\Pi_n)}{n^{p-1} m_1(\mu)^{p \,n}}.
    \]
In \eqref{eqn:cumulantAsymp}, we proved that
\[
\lim_{n\to+\infty} \kappa_\alpha(\Y_n)  = \frac{(\gamma \alpha)^{\alpha-1}}{\alpha!} =  \kappa_\alpha(\Y),\qquad \alpha\leq p,
\]
where the last equality follows from \cite[Theorem 4.1 (3)]{SY13}.
This implies the convergence of moments of order less than or equal to $p$ and uniform integrability of $(\Y_n^p)_{n\in\mathbb{N}}$ (see, e.g., \cite[Theorem 3.6]{Bill}).
\end{remark}

\bibliographystyle{plain}

\bibliography{Bibl}

\end{document}